\documentclass[11pt,leqno]{amsart}
\usepackage{epsfig}
\usepackage{amssymb}
\usepackage{amscd}
\usepackage[all,cmtip,matrix,arrow]{xy}
\usepackage{graphicx}

\newtheorem{theo}{Theorem}[section]
\newtheorem{lemma}[theo]{Lemma}
\newtheorem{defi}[theo]{Definition}
\newtheorem{prop}[theo]{Proposition}

\newtheorem{conj}[theo]{Conjecture}
\newtheorem{cor}[theo]{Corollary}
\newtheorem{remark}[theo]{Remark}

\newtheorem{example}[theo]{Example}
\numberwithin{equation}{section}

\def\bL{\mathbb{L}}
\def\R{\mathbb{R}}

\def\Z{\mathbb{Z}}
\def\Q{\mathbb{Q}}

\def\F{{\mathcal F}}

\def\bR{{\mathbf R}}
\def\bL{{\mathbf L}}

\def\pre-tr{\operatorname{pre-tr}}

\def\Hom{\operatorname{Hom}}
\def\End{\operatorname{End}}

\def\gr{\operatorname{gr}}
\def\hd{\operatorname{hd}}

\DeclareMathOperator*{\colim}{colim}

\newcommand{\Cone}{\operatorname{Cone}}
\newcommand{\Hoch}{\operatorname{Hoch}}
\newcommand{\Pol}{\operatorname{Pol}}

\newcommand{\bbF}{{\mathbb F}}

\newcommand{\mk}{\mathrm k}

\newcommand{\cO}{{\mathcal O}}
\newcommand{\cP}{{\mathcal P}}
\newcommand{\cL}{{\mathcal L}}
\newcommand{\cM}{{\mathcal M}}

\newcommand{\cD}{{\mathcal D}}

\newcommand{\cA}{{\mathcal A}}
\newcommand{\cB}{{\mathcal B}}

\newcommand{\cC}{{\mathcal C}}

\newcommand{\cK}{{\mathcal K}}

\newcommand{\la}{\langle}
\newcommand{\ra}{\rangle}

\newcommand{\cross}{\operatorname{cr}}

\newcommand{\Fun}{\operatorname{Fun}}
\newcommand{\Tot}{\operatorname{Tot}}
	
\newcommand{\Mix}{\operatorname{Mix}}
\newcommand{\Sets}{\operatorname{Sets}}
\newcommand{\Ab}{\operatorname{Ab}}

\newcommand{\Forg}{\operatorname{Forg}}
\newcommand{\Perf}{\operatorname{Perf}}

\newcommand{\supp}{\operatorname{Supp}}

\newcommand{\coker}{\operatorname{Coker}}

\newcommand{\im}{\operatorname{Im}}

\newcommand{\cha}{\operatorname{char}}

\newcommand{\Crit}{\operatorname{Crit}}

\newcommand{\Mat}{\operatorname{Mat}}

\newcommand{\Ext}{\operatorname{Ext}}
\newcommand{\Tor}{\operatorname{Tor}}

\newcommand{\Tr}{\operatorname{Tr}}

\newcommand{\ord}{\operatorname{ord}}

\newcommand{\Spec}{\operatorname{Spec}}

\newcommand{\sgn}{\rm sgn}

\newcommand{\Ho}{\operatorname{Ho}}

\newcommand{\id}{\operatorname{id}}

\newcommand{\Comp}{\operatorname{Comp}}

\newcommand{\Acycl}{\operatorname{Acycl}}

\usepackage{epsf}
\usepackage{amscd}

\newcommand{\T}{\mathcal{T}}

\newcommand{\mS}{\mathfrak{S}}

\newcommand{\m}{\mathfrak{m}}

\title[Mac Lane (co)homology of the second kind and Wieferich primes]
{Mac Lane (co)homology of the second kind and Wieferich primes}

\author{Alexander I. Efimov}
\address{Steklov Mathematical Institute of RAS, Gubkin str. 8, GSP-1, Moscow 119991, Russia}
\address{AG Laboratory, HSE, 7 Vavilova str., Moscow, Russia, 117312.}
\email{efimov@mccme.ru}
\thanks{The author was partially supported by RFBR, grant 2998.2014.1, and RFBR, research project 15-51-50045.}

\begin{document}

\begin{abstract}In this paper we investigate the connection between the Mac Lane (co)homology and Wieferich primes in finite localizations of global number rings. Following the ideas of Polishchuk-Positselski \cite{PP}, we define the Mac Lane (co)homology of the second kind of an associative ring with a central element. We compute this invariants for finite localizations of global number rings with an element $w$ and obtain that the result is closely related with the Wieferich primes to the base $w.$ In particular, for a given non-zero integer $w,$ the infiniteness of Wieferich primes to the base $w$ turns out to be equivalent to the following: for any positive integer $n,$ we have $HML^{II,0}(\Z[\frac1{n!}],w)\ne\Q.$ 

As an application of our technique, we identify the ring structure on the Mac Lane cohomology of a global number ring and compute the Adams operations (introduced in this case by McCarthy \cite{McC}) on its Mac Lane homology. 
\end{abstract}


\maketitle

\tableofcontents

\section{Introduction}
\label{sec:introduction}

In this paper we investigate the connection between the Mac Lane (co)homology and Wieferich primes in finite localizations of global number rings.

We first recall the motivating geometric picture. Let $X=\Spec B$ be an affine variety over a field $k$ of characteristic zero. By Hochschild-Kostant-Rosenberg theorem \cite{HKR},
we have isomorphisms
$$HH_n(B)\cong\Omega^n(X),\quad HH^n(B)\cong\Lambda^nT_X,\quad n\geq 0,$$
where $HH_{\bullet}$ and $HH^{\bullet}$ denote the Hochschild homology and cohomology, which are recalled in Section \ref{ssec:HH}.

In the paper \cite{PP}, Polishchuk and Positselski define and thoroughly study the notion of Hochschild homology $HH_{\bullet}^{II}(\cC,M)$ (resp. cohomology $HH^{II,\bullet}(\cC,M)$)  of the second kind of a curved DG category $\cC$ with coefficients in a curved DG bimodule $M.$ As usual, one writes simply $HH_{\bullet}^{II}(\cC)$ (resp. $HH^{II,\bullet}(\cC)$) when $M$ is the diagonal bimodule.

Within the above notation, any element $w\in B$ gives a $\Z/2$-graded curved DG algebra $B_w,$ which equals $B$ (concentrated in degree zero) as a $\Z/2$-graded algebra, has zero differential and curvature $w.$ It was shown in \cite{E1} that we have natural isomorphisms
$$HH_{0}^{II}(B_w)\cong \bigoplus\limits_i H^{2i}(\Omega^{\bullet}(X),dw\wedge),\quad HH_1^{II}(B_w)\cong\bigoplus\limits_i H^{2i-1}(\Omega^{\bullet}(X),dw\wedge).$$
Here the differential $dw\wedge$ is the wedge multiplication by $dw.$ Similarly one can show the analogous result for Hochschild cohomology of the second kind (which is also a corollary by duality):
$$HH^{II,0}(B_w)\cong \bigoplus\limits_i H^{2i}(\Lambda^{\bullet}T_X,\iota_{dw}),\quad HH^{II,1}(B_w)\cong \bigoplus\limits_i H^{2i-1}(\Lambda^{\bullet}T_X,\iota_{dw}).$$

We see that in both cases, the Hochschild (co)homology of the second kind of $B_w$ is a $\Z/2$-graded $B$-module supported exactly on the critical locus of $w$ in $X.$

The aim of this paper is to formulate and prove the analogous results when 

\smallskip

$\bullet$ the algebra $B$ is replaced by a finite localization $S^{-1}A$ of a global number ring $A$ (that is, $A$ is a ring of integers in a number field). More precisely, $B$ is replaced by the Eilenberg-Mac Lane spectrum $H(S^{-1}A).$ 

$\bullet$ the field $k$ is replaced by the sphere spectrum $S^{\bullet}.$

$\bullet$ the function $w$ on $X$ is still just an element of $S^{-1}A.$ 

\smallskip

We first introduce the notion of a critical point in this context. First consider the case $A=\Z,$ and $w\in\Z$ an integer number. Then the critical point of $w$ is a prime $p$ such that
$$w^p\equiv w\text{ mod }p^2.$$
We denote by $Crit(w)\subset\Spec_m(\Z)$ the set of critical points of $w.$ Here for any commutative ring $R,$ we denote by $\Spec_m(R)$ its maximal spectrum.

Recall that $p$ is a Wieferich prime to the base $w\ne 0$ if $p\nmid w$ and $$p^2\mid (w^{p-1}-1).$$
We see that for $w\ne 0$ we have
$$Crit(w)=\{\text{Weiferich primes }p\leq X\text{ to the base }w\}\cup\{\text{primes }p\text{ such that }p^2\mid w\}.$$

Statistical considerations (\cite{CDP}, Section 3, and \cite{Katz}, Section 2) motivate the following well known conjecture.

\begin{conj}\label{conj:infinitely_many_wieferich}For any integer $w\ne 0$ there are infinitely many Wieferich primes to the base $w.$ Moreover, if $w\not\in\{1,-1\},$ then
$$\#\{\text{Weiferich primes }p\leq X\text{ to the base w}\}\sim \log\log X$$
as $X\to\infty.$\end{conj}

At the moment the infiniteness of Wieferich primes to the base $w$ is unknown for any integer $w\not\in\{-1,0,1\}.$ It is worth mentioning that the infiniteness of non-Wieferich primes to the base $w$ is known only to follow from the ABC-conjecture \cite{S}.

More generally, let $A$ be a global number ring and $K$ its field of fractions. For any (possibly infinite) subset $T\subset\Spec_m A$ we put
$$T^{-1}A=\{x\in K\mid x\in A_{\rho}\text{ for }\rho\in T\}.$$ 

From now on we fix a finite subset $S\subset\Spec_m A.$ Note that
$$S^{-1}A=\cO(\Spec A\setminus S).$$
Let now $w\in S^{-1}A$ be an element. For any maximal ideal $\rho\subset S^{-1}A$ we denote by $k(\rho):=S^{-1}A/\rho$ its residue field. The set of critical points of $w$ is defined by the formula
$$Crit(w)=\{\rho\in\Spec_m(S^{-1}A)\mid w^{|k(\rho)|}\equiv w\text{ mod }\rho^2\}\subset \Spec_m(S^{-1}A).$$

To justify this definition of a critical point, consider the above case $X=\Spec B,$ and $w\in B$ an element. Then a $k$-rational closed point $x\in X$ is a critical point of $w$ iff 
\begin{equation}\label{eq:crit_point_classical} w-w(x)\in\m_x^2,\end{equation}
where $\m_x\subset B$ is the corresponding maximal ideal. Here $k(x)=k$ is naturally embedded into $B,$ hence the LHS of \eqref{eq:crit_point_classical} is well defined. For a general closed point $x\in X,$ the residue field $k(x)$ is some finite extension of $k$ and there is no embedding $k(x)\hookrightarrow B.$ However, if we replace $B$ by the completed local ring $\widehat{B}_x,$ then we have a natural injective homomorphism $k(x)\hookrightarrow\widehat{B}_x.$ Thus, \eqref{eq:crit_point_classical} makes perfect sense in $\widehat{B}_x,$ and moreover this condition exactly defines the set of closed points in the critical locus of $w$ in $X.$

Returning to the number ring case, there is of course no homomorphism $k(\rho)\to \widehat{(S^{-1}A)}_{\rho}$ for a maximal ideal $\rho\subset S^{-1}A.$
But there is a unique multiplicative section
$$\omega:k(\rho)\to \widehat{(S^{-1}A)}_{\rho},$$ which sends an element $x\in k(\rho)$ to its Teichm\"uller representative.
Note that we have $$\omega(\bar{w})\equiv w^{|k(\rho)|}\text{ mod }\hat{\rho}^2.$$
It follows that $\rho$ is contained in $Crit(w)$ iff we have
$$w-\omega(\bar{w})\in\hat{\rho}^2,$$ which is exactly the analogue of \eqref{eq:crit_point_classical}.

We now turn to the analogues of Hochschild (co)homology (of the second kind) in this context.
In \cite{Bok}, B\"okstedt defines topological Hochschild homology of any symmetric ring spectrum with coefficients in a bimodule.
Given an associative ring $R$ and an $R\text{-}R$-bimodule $M,$ one writes $THH_{\bullet}(R,M)$ for the topological Hochschild homology of $HR$ with coefficients in $HM.$

It was proved in \cite{PW} that topological Hochschild homology is naturally isomorphic to another invariant which is called {\it Mac Lane homology}. It was originally introduced in \cite{ML} using cubical construction which we recall in Section \ref{sec:cubical_construction}. The cubical construction assigns to an abelian group $C$ a functorial chain complex $Q_{\bullet}(C)$ concentrated in non-negative homological degrees. We always have $H_0(Q_{\bullet}(C))=C$ and $H_1(Q_{\bullet}(C))=0,$ and $Q_n(C)$ is torsion for $n\geq 2.$ More generally, it was shown in \cite{EM2}, \cite{EM3} that
$$H_n(Q_{\bullet}(C))=H_{n+m}(K(C,m))\text{ for }m>n,$$
see also \cite{Pir} for a different proof. Here $K(C,m)$ are the Eilenberg and Mac Lane spaces.

If $R$ is an associative ring, then $Q_{\bullet}(R)$ is naturally a DG ring. The natural projection $Q_{\bullet}(R)\to R$ is a morphism of DG rings. For an $R\text{-}R$-bimodule $M,$ the Mac Lane (co)homology of $R$ with coefficients in $M$ is defined by
$$HML_{\bullet}(R,M)=HH_{\bullet}(Q_{\bullet}(R),M),\quad HML^{\bullet}(R,M)=HH^{\bullet}(Q_{\bullet}(R),M).$$
It is proved in \cite{PW} that $$THH_n(R,M)\cong HML_n(R,M).$$

We use the above-mentioned notion of the Hochschild (co)homology of the second kind to define the Mac Lane (co)homology of the second kind $$HML^{II}_{\bullet}(R,w;M),\quad HML^{II,\bullet}(R,w;M),$$
see Subsections \ref{ssec:HH_second_kind} and \ref{ssec:HML}. Here $w$ is a central element of $R,$ which is also central for $M,$ that is, left and right multiplications by $w$ coincide on $M.$ These invariants are $\Z/2$-graded abelian groups.

Recall that a different ideal $\cD_A\subset A$ can be defined by the formula
$$\cD_A^{-1}=\{x\in K\mid \Tr_{K/\Q}(x)\in\Z\}.$$ We put
$$\cD_{S^{-1}A}:=\cD_A\otimes_A S^{-1}A.$$ Our main result is the following theorem (the more precise version is Theorem \ref{th:HML^II_for_loc_of_number_rings}).

\begin{theo}\label{th:HML^II_number_rings_intro}Within the above notation, let $w\in S^{-1}A$ an element. We always have $HML^{II,1}(S^{-1}A,w)=0.$ Assume that $$Crit(w)\cap \supp(S^{-1}A/2\cD_{S^{-1}A})=\emptyset$$ and put $$T:=\Spec_m(S^{-1}A)\setminus Crit(w).$$ Then we have
$$HML^{II,0}(S^{-1}A,w)=(T\cup S)^{-1}A\oplus\bigoplus\limits_{\substack{\rho\in Crit(w);\\ n>0}}S^{-1}A/\rho^{\ord_p(n)}$$
\end{theo}

We have the following remarkable corollary.

\begin{cor}Let $w\ne 0$ be an integer. The following are equivalent:

(i) there are infinitely many Wieferich primes $p$ to the base $w;$

(ii) For any positive integer $n\in\Z_{>0}$ we have $$HML^{II,0}\left(\Z\left[\frac1{n!}\right],w\right)\not\simeq \Q.$$\end{cor}

By duality, we obtain the similar result for the homology (see Theorem \ref{th:HML^II_dual_of_loc_of_number_rings}).

\begin{theo}\label{th:HML^II_dual_intro}Within the notation and assumptions of Theorem \ref{th:HML^II_number_rings_intro}, assume that $T\ne\emptyset.$ Then we have
$$HML^{II}_0(S^{-1}A,w)=0$$
$$HML^{II}_1(S^{-1}A,w)\cong(\prod\limits_{\rho\in T}\widehat{A}_{\rho})/S^{-1}A\oplus \prod\limits_{\substack{\rho\in Crit(w);\\ n>0}}S^{-1}A/\rho^{\ord_p(n)}.$$
\end{theo}

As an easy application, we compute the graded ring structure on $HML^{\bullet}(A)$ (see Theorem \ref{th:ring_structure_on_HML_S^-1A}).

\begin{theo}\label{th:ring_structure_intro}We have an isomorphism of graded rings $$HML^{\bullet}(A)\cong \Gamma_{A}(x)/(\cD_A\cdot x),$$ where $\deg(x)=2$ and $\Gamma_A(x)$ denotes the algebra of divided powers on $x$ over $A.$\end{theo}

In the case $A=\Z,$ the ring structure on $HML^{\bullet}(\Z)$ was computed in \cite{FP}.

Another application is the computation of the Adams operations on Mac Lane homology. Recall that McCarthy \cite{McC} defines for any commutative ring $R$ the sequence of operations
$$\psi^r:HML_{\bullet}(R)\to HML_{\bullet}(R),$$ such that $\psi^{rs}=\psi^r\psi^s$ and $\psi^1=\id.$ These generalize the well-known lambda operations on the Hochschild homology of a commutative algebra, which were defined independently by Loday in \cite{L2} and Gerstenhaber and Schack in \cite{GS}. We have the following result.

\begin{theo}\label{th:Adams_operations_intro}Let $A$ be a global number ring. The operation $\psi^r$ on $HML_0(A)$ is equal to identity and we have
$$\psi^r(x)=r^nx\text{ for }x\in HML_{2n-1}(A),\,n\geq 1.$$\end{theo}

As in the paper \cite{McC}, to deal with Adams operations we use the functor homology interpretation of Mac Lane homology which was obtained by Jibladze and Pirashvili \cite{JP}. We actually prove a more general statement for a generalized cubical construction, see Proposition \ref{prop:HML^II_via_Q^n}.

Finally, we would like to mention that Mac Lane homology $HML_{\bullet}(R,M)$ is also identified with the so-called stable K-theory $K_{\bullet}^s(R,M).$ For the definition, see for example \cite{L}, Section 13.3. Without going into details, the stable K-theory of $R$ with coefficients in a bimodule $N$ is defined as the homology of the Volodin space $X(R)$ with coefficients in a certain local coefficient system $\cM(N).$ It might be useful to interpret the Mac Lane homology of the second kind in this context.

The paper is organized as follows.

In Section \ref{sec:preliminaries_on_DG} we recall the basic notions and facts about DG categories, DG modules and derived categories. The material of this section is mostly covered by \cite{Ke1}, \cite{Ke2} and \cite{Dr}.

In Section \ref{sec:cross_effects_polynomial} we recall the cross-effects for maps of abelian groups and for functors between (pre-)additive categories. After that we recall the polynomial maps and functors. The material is essentially contained in \cite{EM}.

Section \ref{sec:cubical_construction} is devoted to the generalized cubical construction $Q^n_{\bullet}(-),$ which was introduced in \cite{JM}. It provides for any abelian group a functorial non-negative chain complex. For $n=1,$ one gets the Mac Lane cubical construction mentioned above. The functors $Q^n_{\bullet}(-)$ are lax monoidal, hence produce a DG ring from an associative ring and more generally produce a DG category from a pre-additive category (see Definition \ref{defi:Q^n_on_categories}).

In Section \ref{sec:DG_quotients_and_Q^n} we interpret the functor $Q^n_{\bullet}$ applied to a small additive category in terms of DG quotients. Namely, the DG category $Q^n_{\bullet}(\cC)$ is quasi-equivalent (up to Karoubi completion) to the DG quotient of the reduced additivization $\Z[\bar{\cC}]$ by the subcategory of $(n+1)$-th cross-effects of the tautological functor $\cC\to\Z[\bar{\cC}]$ (see Theorem \ref{th:exact_seq_Q^n_cr_n+1}).

Section \ref{sec:mixed complexes} is devoted to mixed complexes over a commutative ring. We define the Laurent-type totalization of a mixed complex which is a $\Z/2$-graded complex (Definition \ref{defi:totalizations}). This totalization is invariant under quasi-isomorphism (Proposition \ref{prop:functoriality_of_Tot}). We also list some basic properties of the associated $\Z$-graded spectral sequence. These are actually the same as bi-graded spectral sequences $\{(E_r^{\bullet,\bullet},d_r)\}_{r\geq r_0}$ with additional periodicity isomorphisms $E_r^{p+1,q+1}\cong E_r^{p,q}$ such that all the relevant diagrams commute. We also prove a useful technical result about mixed complexes coming from DG modules over strongly commutative DG algebras with divided powers (Lemma \ref{lem:q_is_divided_powers}).

Section \ref{sec:HH_and_HML_second_kind} is devoted to the Hochschild and Mac Lane (co)homology and their versions of the second kind. In Subsection \ref{ssec:HH} we recall the Hochschild chain and cochain complexes of DG categories with coefficients in a DG bimodule. We also prove here a technical result for gluing of DG categories.

In Subsection \ref{ssec:HH_second_kind} we introduce the structure of a mixed complex on the Hochschild (co)chain complex induced by a degree zero closed element $w$ of the center of the DG category, which is also central for the bimodule. For an element $w$ the additional differential is denoted by $\delta_w.$ We define the Hochschild (co)homology of the second kind to be the (co)homology of the totalization of this mixed complex. We prove some technical results including the form of a Leibniz rule for the differential $\delta_w$ when $w$ is a product $w_1\dots w_k$ (Lemma \ref{lem:q_is_for_delta_w_1...w_k}).

In Subsection \ref{ssec:HML} we recall the Mac Lane (co)homology and its version of the second kind. In particular, we show that it can be computed using any of the generalized cubical constructions, as an application of Theorem \ref{th:exact_seq_Q^n_cr_n+1}. We mention the well-known result about the bijection between $HML^2(R,M)$ and isomorphism classes of square-zero extensions of $R$ by $M$ (Theorem \ref{th:HML^2_sq_zero_ext}). Also, we prove some technical statements about duality between the (mixed) complexes computing Mac Lane (co)homology (of the second kind) of commutative rings (Lemmas \ref{lem:duality_of_Hoch_complexes} and \ref{lem:duality_for_homology}). 

In Section \ref{sec:change_of_rings_sp_seq} we recall the change of rings spectral sequence for Mac Lane (co)homology following \cite{L}, Section 13.4.22. We formulate its applications for localizations and completions of commutative rings (Propositions \ref{prop:HML_localizations} and \ref{prop:HML_completions}).

In Section \ref{sec:HML_finite_fields} we recall the Mac Lane (co)homology of finite fields. We also prove a certain technical result about the complex computing Mac Lane cohomology in this case (Lemma \ref{lem:reduced_for_HML_of_F_q}).

Section \ref{sec:HML_secomd_kind_discr_valuation} is devoted to the computation of Mac Lane (co)homology of the second kind for complete discrete valuation rings. More precisely, we consider the ring of integers $R=\cO_L$ in the finite extensions $L/\Q_p,$ with a maximal ideal $\m\subset R$ and the residue field $k(\m).$ 
We first recall the computation of $THH_{\bullet}(R)$ which was done by Lindenstrauss and Madsen \cite{LM}. As a direct corollary, we get the computation of Mac Lane (co)homology of $R$ (Corollary \ref{cor:HML_for_discr_valuation}).
 
In Subsection \ref{ssec:ramified} we compute the Mac Lane (co)homology of the second kind of $(R,w)$ when the element $w$ satisfies $w\not\equiv w^{|k(\m)|}\text{ mod }\m^2.$ We show that in this case $HML^{II,0}(R,w)$ is isomorphic to $L$ (Theorem \ref{th:HML^II_discr_valuation_ramified}). As an application, we identify the ring structure on $HML^{\bullet}(R)$ (Corollary \ref{cor:product_on_HML_ramified}). Also, by duality we show that $HML_{\bullet}^{II}(R,w)=0$ (Theorem \ref{th:HML^II_dual_for_discr_valuation_ramified}).

In Subsection \ref{ssec:nonramified} we obtain the analogous results when $R$ is non-ramified. Moreover, we also consider the case when $w\equiv w^{|k(\m)|}\text{ mod }\m^2.$ In this case, assuming that $\cha(k(\m))=p\ne 2,$ we show that $HML^{II,\bullet}(R,w)$ is non-canonically identified with the totalization of $HML^{\bullet}(R)$ (Theorem \ref{th:HML^II_for_discr_valuation} 2)), and similarly for the homology (Theorem \ref{th:HML^II_dual_for_discr_valuation} 2)).

Section \ref{sec:HML_loc_number_rings} is devoted to the main case of a finite localization of a global number ring $A.$ Again, we recall the computation of $THH_{\bullet}(A)$ from \cite{LM} and deduce the computation of Mac Lane (co)homology (Corollary \ref{cor:HML_for_loc_of_number_rings}). The main result here is Theorem \ref{th:HML^II_for_loc_of_number_rings} which is an extended version of Theorem \ref{th:HML^II_number_rings_intro}. The result essentially reduces to the case of discrete valuation ring considered in Section \ref{sec:HML_secomd_kind_discr_valuation}. As an application, we prove Theorem \ref{th:ring_structure_intro} (this is Theorem \ref{th:ring_structure_on_HML_S^-1A}). By duality ,we also compute the Mac Lane homology of the second kind (Theorem \ref{th:HML^II_dual_of_loc_of_number_rings}).

Section \ref{sec:Adams_on_HML} is devoted to Adams operations on Mac Lane homology. In Subsection \ref{ssec:operations_on_simplicial_sets} we follow \cite{McC} to define the notion of natural system of operators on a simplicial set. In Subsection \ref{ssec:operations_on_simplicial_abelian_groups} we consider the case of a simplicial abelian group and use the construction of \cite{McC}, Section 3, to construct from a system of natural operators a collection of chain maps on the associated chain complex. Finally, in Subsection \ref{ssec:computing_Adams_on_HML} we use the construction from the previous Subsection to define the operations on Mac Lane homology. We use our considerations with Mac Lane homology of the second kind to prove Theorem \ref{th:Adams_operations_intro} (this is Theorem \ref{th:operations_psi^r_number_rings}).

Appendix \ref{sec:appendix_on_sp_seq} is devoted to $\Z$-graded spectral sequences.

{\noindent{\bf Acknowledgements.}} I am grateful to V. Ginzburg, D. Kaledin and L. Positselski for useful discussions. The paper was prepared within the framework of a subsidy granted to the HSE by the Government of the Russian Federation for the implementation of the Global Competitiveness Program.

\section{Preliminaries on DG categories}
\label{sec:preliminaries_on_DG}

We refer to \cite{Ke1}, \cite{Ke2} for a general introduction to DG categories and DG modules. The DG quotients of DG categories are defined in \cite{Dr}.

We fix some basic commutative ring $k.$ All the categories in this section will be assumed to be $k$-linear.

\begin{defi}\label{defi:DG_category} A DG category $\cC$ over $k$ is given by the following data:

- a class of objects $Ob(\cC);$

- for any two objects $X,Y\in Ob(\cC),$ a complex of $k$-modules $\cC(X,Y)=\Hom_{\cC}(X,Y);$

- for any three objects $X,Y,Z\in Ob(\cC),$ a composition morphism of complexes
$$-\circ-:\cC(Y,Z)\otimes_k\cC(X,Y)\to \cC(X,Z).$$

It is required that the composition morphisms satisfy associativity property, and for any object $X\in Ob(\cC)$ there is a (unique) closed element $1_X\in\cC^0(X,X)$ of degree zero, such that for any morphism $f\in\cC(Y,Z)$ we have $1_Z\circ f=f=f\circ 1_Y.$ 
\end{defi}

We will usually write $fg$ instead of $f\circ g$ for composable morphisms in a DG category.

A (unital) DG algebra is the same as a DG category with one object. A DG algebra over $\Z$ is also called a DG ring.

With any DG category one can associate a graded $k$-linear category, as well as non-graded one.

\begin{defi}\label{defi:homotopy_category} Let $\cA$ be a DG category. Its graded homotopy category $\Ho^{\bullet}(\cA)$ is defined as follows. The class of objects of $\Ho^{\bullet}(\cA)$ is the same as for $\cA.$ The morphisms are defined by the formula
$$\Hom_{\Ho^{\bullet}(\cA)}(X,Y):=H^{\bullet}(\cA(X,Y)).$$
The composition in $\Ho^{\bullet}(\cA)$ is induced by the composition in $\cA.$

The (non-graded) homotopy category $\Ho(\cA)$ has the same class of objects as $\cA$ and $\Ho^{\bullet}(\cA),$ and the morphisms are defined by the formula $$\Hom_{\Ho(\cA)}(X,Y)=H^0(\cA(X,Y)).$$\end{defi}

Given a DG category $\cA,$ we will also denote by $\cA^{gr}$ the graded category which is obtained from $\cA$ by forgetting the differential on morphisms.

Next  we recall the notion of a DG functor between DG categories.

\begin{defi}Let $\cA$ and $\cB$ be DG categories. A DG functor $F:\cA\to\cB$ is given by the following data:

i) for each object $X\in Ob(\cA),$ an object $F(X)\in Ob(\cB);$

ii) for any two objects $X,Y\in Ob(\cA),$ a ($k$-linear) morphism of complexes $F(X,Y):\cA(X,Y)\to \cB(X,Y).$

The maps on morphisms are required to commute with the composition in $\cA$ and $\cB,$ and to preserve identity morphisms.
\end{defi}

It is clear that a DG functor $F:\cA\to \cB$ induces the graded functors $F^{gr}:\cA^{gr}\to\cB^{gr},$ $\Ho^{\bullet}(F):\Ho^{\bullet}(\cA)\to \Ho^{\bullet}(\cB),$ as well as the functor $\Ho(F):\Ho(\cA)\to \Ho(\cB).$

\begin{defi}A DG functor $F:\cA\to\cB$ is called a quasi-equivalence if the induced functor $\Ho^{\bullet}(F):\Ho^{\bullet}(\cA)\to \Ho^{\bullet}(\cB)$ is an equivalence of graded categories.\end{defi}

We recal that (up to set-theoretic issues) the DG functors form a DG category.

\begin{defi}Let $\cA$ be a small DG category, and $\cB$ any DG category. The DG category $\Fun(\cA,\cB)$ is defined as follows. Its objects are DG functors $F:\cA\to\cB.$ Given two DG functors $F,G:\cA\to\cB,$ the $k$-module $\Hom^n(F,G)$ is the set of all degree $n$ natural transformations $f:F^{gr}\to G^{gr}.$ The differential $d:\Hom^n(F,G)\to\Hom^{n+1}(F,G)$ is defined by the formula
$$d(f)_{X,Y}=d(f_{X,Y}),\quad X,Y\in Ob(\cA).$$
The composition is the same as for natural transformations.\end{defi}

\begin{example}The basic example of a $k$-linear DG category is the DG category $\Comp(k).$ Its objects are complexes of $k$-modules. Given two objects $M^{\bullet},N^{\bullet}\in \Comp(k),$ the complex of morphisms $\Hom(M^{\bullet},N^{\bullet})$ is defined by the formulas $$\Hom(M^{\bullet},N^{\bullet})^n=\prod\limits_{i\in\Z}\Hom_k(M^i,N^{i+n}),$$
$$d(f)(m)=d(f(m))-(-1)^{|f|}f(d(m)).$$
The composition is obvious.\end{example}

\begin{defi}Let $\cA$ be a DG category. The opposite DG category $\cA^{op}$ has the same class of objects as $\cA,$ and the morphisms are defined by the formula $\cA^{op}(X,Y):=\cA(Y,X).$ For each morphism $f\in\cA(X,Y),$ we denote by $f^{op}$ the corresponding morphism in $\cA^{op}(Y,X).$

Given homogeneous morphisms $f\in\cA(X,Y),$ $g\in\cA(Y,Z),$ the composition $f^{op}g^{op}$ is defined by the formula $$f^{op}g^{op}=(-1)^{|f|\cdot|g|}(gf)^{op}.$$\end{defi}

It is clear that $(\cA^{op})^{op}=\cA.$

Now we recall the homotopy and derived categories of (right) DG modules.

\begin{defi}1) Let $\cA$ be a small DG category. The DG category of right DG $\cA$-modules is defined by the formula $\text{Mod-}\cA:=\Fun(\cA^{op},\Comp(k)).$

The homotopy category of right DG $\cA$-modules is defined by the formula $K(\cA):=\Ho(\text{Mod-}\cA).$ This is a triangulated category.

2) The DG $\cA$-module $M$ is called acyclic if for any object $X\in Ob(\cA)$ the complex $M(X)$ is acyclic. We denote by $\Acycl(\cA)\subset K(\cA)$ the full triangulated subcategory of acyclic complexes.

3) The derived category of $\cA$ is defined to be the quotient triangulated category $D(\cA):=K(\cA)/\Acycl(\cA).$
\end{defi}

All DG modules are assumed to be right unless otherwise stated. We will often write just  ``$\cA$-module'' instead of ``DG $\cA$-module''

Let $\cA$ be a small DG category, and $M$ (resp. $N$) an $\cA$-module (resp. $\cA^{op}$-module). For the homogeneous $a\in \cA(X_1,X_2)$ and $m\in M(X_2)$ (resp. $n\in N(X_1)$) we put
$$ma:=(-1)^{|a||m|}M(a^{op})(m),\quad an=N(a)(n).$$

For $\cA$-modules $M,N\in\text{Mod-}\cA,$ we put $\Hom_{\cA}(M,N):=(\text{Mod-}\cA) (M,N).$

The $\cA$-module $P$ (resp. $I$) is called h-projective (resp. h-injective) if for any acyclic $\cA$-module $N$ the complex $\Hom_{\cA}(P,N)$ (resp. $\Hom_{\cA}(N,I)$) is acyclic. We denote by $\text{h-proj}(\cA)\subset K(\cA)$ (resp. $\text{h-inj}(\cA)\subset K(\cA)$) the full triangulated subcategory of h-projective (resp. h-injective) $\cA$-modules.

\begin{prop}\label{prop:h-proj,h-inj} We have semi-orthogonal decompositions $$K(\cA)=\langle\Acycl(\cA),\text{h-proj}(\cA)\rangle,\quad K(\cA)=\langle \text{h-inj}(\cA),\Acycl(\cA)\rangle.$$ In particular, we have natural equivalences of triangulated  categories $\text{h-proj}(\cA)\simeq D(\cA)\simeq \text{h-inj}(\cA).$\end{prop}

\begin{proof}This is proved in \cite{Ke1}, Sections 3.1 and 3.2.\end{proof}

We recall the notion of tensor product of DG categories.

\begin{defi}\label{defi:tensor_product_of_DG_cat}1) Let $\cB$ and $\cC$ be DG categories over $k.$ Their tensor product $\cB\otimes_k \cC$ is defined as follows.
The objects are defined by the formula $Ob(\cB\otimes_k\cC):=Ob(\cB)\times Ob(\cC).$
The morphisms are defined by the formula $$(\cB\otimes\cC)((X_1,Y_1),(X_2,Y_2)):=\cB(X_1,X_2)\otimes_k\cC(Y_1,Y_2),\quad X_1,X_2\in Ob(\cB),\, Y_1,Y_2\in Ob(\cC).$$
The definition of composition is obvious.

2) For a pair of DG functors $F_1:\cA_1\to\cB_1,$ $F_2:\cA_2\to\cB_2,$ we denote by
$F_1\otimes F_2:\cA_1\otimes\cA_2\to\cB_1\otimes\cB_2$ their tensor product, given by
$$(F_1\otimes F_2)(X_1,X_2)=(F_1(X_1),F_2(X_2)),\quad X_1\in Ob(\cA_1),\,X_2\in Ob(\cA_2);$$ 
$$(F_1\otimes F_2)((X_1,X_2),(Y_1,Y_2))=F_1(X_1,Y_1)\otimes F_2(X_2,Y_2),\quad X_1,Y_1\in Ob(\cA_1),\,X_2,Y_2\in Ob(\cA_2).$$\end{defi}

Given a small DG category $\cA,$ we have the natural tensor product DG functor
\begin{equation}\label{eq:tensor_product_DG}-\otimes_{\cA}-:\text{Mod-}\cA\otimes\text{Mod-}\cA^{op}\to\text{Mod-}k.\end{equation}
On objects, it is given by the formula
$$M\otimes_{\cA}N=\coker(\bigoplus\limits_{X_1,X_2\in Ob(\cA)}M(X_2)\otimes\cA(X_1,X_2)\otimes N(X_1)\stackrel{\phi}{\to}\bigoplus\limits_{X\in Ob(\cA)}M(X)\otimes N(X)),$$
where $\phi(m\otimes a\otimes n)=ma\otimes n-m\otimes an,$ $m\in M(X_2),$ $n\in N(X_1),$ $a\in \cA(X_1,X_2).$

The DG functor \eqref{eq:tensor_product_DG} induces the biexact bifunctor $-\otimes_{\cA}-:K(\cA)\times K(\cA^{op})\to K(k).$ Replacing one of the arguments by an h-projective resolution, we get a biexact bifunctor
$-\stackrel{\bL}{\otimes}_{\cA}-:D(\cA)\times D(\cA^{op})\to D(k).$ 

Given a DG functor $F;\cA\to\cB$ between small DG categories, we have the restriction of scalars DG functor $F_*:\text{Mod-}\cB\to\text{Mod-}\cA,$ $F_*(M)(X)=M(F(X)).$ It has a left adjoint $F^*:\text{Mod-}\cA\to\text{Mod-}\cB,$ given by the formula
$$F^*(N)(Y)=N\otimes_{\cA}\cB(Y,F(-))$$
(extension of scalars).
These DG functors induce the adjoint pair of exact functors $F^*:K(\cA)\to K(\cB),$ $F_*:K(\cB)\to K(\cA).$ Similarly, we have an adjoint pair of functors on derived categories: $\bL F^*:D(\cA)\to D(\cB),$ $F_*:D(\cB)\to D(\cA).$  

\begin{defi}\label{defi:h-projective_DGcat} A DG category $\cA$ is called h-projective over $k$ if for any two objects $X,Y\in Ob(\cA)$ the complex of $k$-modules $\cA(X,Y)$ is h-projective.\end{defi}

We will assume the $k$-linear DG categories to be h-projective, unless otherwise stated.

For any small DG category $\cA,$ we have a diagonal bimodule $I_{\cA}\in\text{Mod-}(\cA\otimes\cA^{op}),$ given by
$$I_{\cA}(X,Y)=\cA(X,Y).$$
For any DG functor $F:\cA\to\cB$ between small DG categories we have a natural morphism of bimodules
\begin{equation}\label{eq:morphism_for_diagonals}f:I_{\cA}\to (F\otimes F^{op})_*I_{\cB},\quad f_{(X,Y)}(a)=F_{X,Y}(a)\text{ for }a\in\cA(X,Y).\end{equation}
By adjunction, the morphism \eqref{eq:morphism_for_diagonals} gives a natural morphism $\bL(F\otimes F^{op})^*I_{\cA}\to I_{\cB}.$

\begin{prop}\label{prop:homological_epi_equiv_cond} Let $F:\cA\to\cB$ be a DG functor between small h-projective DG categories. The following are equivalent.

(i) The functor $\bL F^*:D(\cA)\to D(\cB)$ is a localization.

(ii) The functor $F_*:D(\cB)\to D(\cA)$ is fully faithful.

(iii) The morphism $\bL(F\otimes F^{op})^*I_{\cA}\to I_{\cB}$ is an isomorphism in $D(\cB\otimes\cB^{op}).$\end{prop}

\begin{proof}This is proved in \cite{E}, Proposition 3.5.\end{proof}

\begin{defi}\label{defi:homological_epi} A DG functor $F:\cA\to\cB$ between small DG categories is said to be a homological epimorphism if the equivalent conditions of Proposition \ref{prop:homological_epi_equiv_cond} hold.\end{defi}

Let $\cA$ be a small DG category. For any object $X\in Ob(\cA),$ we have a representable DG module $h_X\in\text{Mod-}\cA,$ given by the formula
$$h_X(Y)=\cA(Y,X),\quad Y\in Ob(\cA).$$
By Yoneda, for any DG module $M\in\text{Mod-}\cA,$ we have $$\Hom_{\cA}(h_X,M)=M(X).$$
In particular, the DG module $h_X$ is h-projective.

\begin{defi}\label{defi:perfect_complexes} Let $\cA$ be a small DG category. The subcategory $\Perf(\cA)\subset D(\cA)$ of perfect complexes is defined to be the full triangulated subcategory, classically generated by the objects $h_X,$ $X\in Ob(\cA).$ That is, $\Perf(\cA)$ is the smallest strictly full subcategory, which contains the objects $h_X,$ and is closed under taking cones, shifts and direct summands.\end{defi}

\begin{prop}(\cite{Ke1}, Section 5.3) \label{prop:compact_are_perfect} Let $\cA$ be a small DG category. The triangulated category $D(\cA)$ is compactly generated, and its subcategory of compact objects $D(\cA)^c$ coincides with $\Perf(\cA).$\end{prop}

Given a DG functor $F:\cA\to\cB,$ for each object $X\in Ob(\cA)$ we have the isomorphism of $\cB$-modules $F^*(h_X)\cong h_{F(X)}.$ Since $h_X$ is h-projective, we also have the isomorphism $\bL F^*(h_X)\cong h_Y$ in $D(\cB).$ In particular, it follows that the DG functor $\bL F^*$ takes $\Perf(\cA)$ to $\Perf(\cB).$ 

We recall the notion of DG quotient, introduced by Drinfeld.

\begin{defi}\label{defi:DG_quotient} Let $\cA$ be a small h-projective DG category, and $\cB\subset\cA$ a full small DG subcategory. The DG quotient category $\cA/\cB$ has the same objects as $\cA,$ and the morphisms in $\cA/\cB$ are obtained from the morphisms in $\cA$ by freely adding for each object $X\in\cB$ the morphism $h(X):X\to X$ of degree $-1,$ with the only relation $d(h(X))=\id_X.$ In particular, the graded $k$-modules of morphisms in $\cA/\cB$ are given by the formula
$$(\cA/\cB)(X,Y)=\cA(X,Y)\oplus\bigoplus\limits_{\substack{n\geq 1;\\ Z_1,\dots,Z_n\in Ob(\cB)}}\cA(Z_n,Y)\otimes\cA(Z_{n-1},Z_n)\otimes\dots\otimes\cA(X,Z_1)[n].$$\end{defi}

Clearly, we have a natural projection DG functor $\pi:\cA\to\cA/\cB.$

\begin{prop}\label{prop:DG_quotient_localization} Let $\cA$ be a small h-projective DG category, and $\cB\subset\cA$ its full DG subcategory. Then the extension of scalars functor $\bL\pi^*:\Perf(\cA)\to\Perf(\cA/\cB)$ is a localization, up to direct summands.\end{prop}

\begin{proof}This follows from \cite{Dr}, Theorem 1.6.2.\end{proof}

\begin{prop}\label{prop:DG_quotient_is_homological_epi} Let $F:\cA\to\cB$ be a DG functor, and suppose that the functor $\bL F^*:\Perf(\cA)\to\Perf(\cB)$ is a localization, up to direct summands. Then $F$ is a homological epimorphism.

In particular, if $\cC\subset\cA$ is a full DG subcategory, then the DG functor $\pi:\cA\to\cA/\cC$ is a homological epimorphism.\end{prop}

\begin{proof}This follows from \cite{E}, Proposition 3.7.\end{proof}

\begin{prop}\label{prop:criterion_for_quotient}(\cite{Dr}, Proposition 1.4) Let $\cA$ and $\cC$ be small DG categories, and assume that $\cA$ is h-projective. Let $\cB\subset\cA$ a full DG subcategory and $\iota:\cB\to\cA$ the tautological embedding. Let $F:\cA\to\cC$ be a DG functor such that the functor $\Ho(F):\Ho(\cA)\to\Ho(\cC)$ is essentially surjective and vanishes on the subcategory $\Ho(\cB)\subset\Ho(\cA).$ Then the DG functor $F$ induces a (non-unique) DG functor $\bar{F}:\cA/\cB\to\cC.$ The following are equivalent.

(i) The DG functor $\bar{F}:\cA/\cB\to\cC$ is a quasi-equivalence.

(ii) For any object $X\in Ob(\cA),$ the object
$$Cone(h_X\to F_*(\bL F^* h_X))\in D(\cA)$$ is contained in the localizing subcategory generated by $\bL\iota^*(\Perf(\cB)).$
\end{prop}

We recall the notion of gluing of small DG categories.

\begin{defi}\label{defi:gluing} Let $\cA$ and $\cB$ be small DG categories, and $M\in\text{Mod-}(\cA\otimes\cB^{op})$ a DG bimodule.
The glued DG category $\cC=\cA\sqcup_M\cB$ is defined as follows:

1) The objects are given by $Ob(\cC)=Ob(\cA)\sqcup Ob(\cB).$

2) The morphisms are given by $$\cC(X,Y)=\begin{cases}\cA(X,Y) & \text{if }X,Y\in Ob(\cA);\\
                                                     \cB(X,Y) & \text{if }X,Y\in Ob(\cB);\\
                                                     M(X,Y) & \text{if }X\in Ob(\cA),\,Y\in Ob(\cB);\\
                                                     0 & \text{if }X\in Ob(\cB),\,Y\in Ob(\cA).
                                                      \end{cases}$$
                                                      
3) The composition in $\cC$ comes from the compositions in $\cA$ and $\cB,$ and from the bimodule structure on $M.$
\end{defi}

In the notation of Definition \ref{defi:gluing}, let us denote by $\iota_{\cA}:\cA\to\cC,$ $\iota_{\cB}:\cB\to\cC$ the natural inclusions. Then we have a natural exact triangle in $D(\cC\otimes\cC^{op})$ (see e.g. \cite{LS})
\begin{equation}\label{eq:diagonal_of_gluing}\bL(\iota_{\cA}\otimes\iota_{\cB}^{op})^*M\to\bL(\iota_{\cA}\otimes\iota_{\cA}^{op})^*I_{\cA}\oplus \bL(\iota_{\cB}\otimes\iota_{\cB}^{op})^*I_{\cB}\to I_{\cA\sqcup_M\cB}.\end{equation}

We will need some results on h-projective DG modules. 

\begin{defi}\label{defi:semi-free}Let  $\cA$ be a small DG category. A DG $\cA$-module $M$ is called semi-free if it has an exhaustive filtration $$0=F_0M\subset F_1M\subset\dots$$ by DG submodules such that for each $n>0$ the DG module $\gr_n^F M$ is a (possibly infinite) direct sum of shifts of representable DG modules $h_X,$ $X\in Ob(\cA).$\end{defi}

\begin{prop}\label{prop:semi-free_h_proj}(\cite{Dr}, Appendix C.8) Any semi-free DG $\cA$-module is h-projective. Moreover, any h-projective DG module is isomorphic to some semi-free $\cA$-module in $K(\cA).$\end{prop}

Let us consider the case when $\cA=R$ is a $k$-algebra (considered as a DG category with a single object). Let $u$ be a formal variable of cohomological degree $2.$ Then we can  identify $\Z/2$-graded complexes of $R$-modules with DG $R[u^{\pm 1}]$-modules. By abuse of notation, we denote by $\text{Mod-}R$ the abelian category of $R$-modules. Clearly, DG $R$-modules are just complexes of $R$-modules. Note that any bounded above complex of projective $R$-modules is a direct summand of a semi-free complex, hence h-projective. 

\begin{prop}\label{prop:h-proj_unbounded} Assume that the abelian category $\text{Mod-}R$ has finite homological dimension. 

1) Let $P^{\bullet}$ be a complex of projective $R$-modules. Then $P^{\bullet}$ is h-projective as a DG $R$-module.

2) Let $P^{\bullet}$ be a $\Z/2$-graded complex of projective $R$-modules. Then $P^{\bullet}$ is h-projective as a DG $R[u^{\pm 1}]$-module.\end{prop}

\begin{proof}We will prove 1), and the proof of 2) is analogous.

By Propositions \ref{prop:h-proj,h-inj} and \ref{prop:semi-free_h_proj}, we can find a semi-free complex of $R$-modules $P'^{\bullet}$ and a quasi-isomorphism $f:P'^{\bullet}\to P^{\bullet}.$ It suffices to prove that the acyclic complex $Cone(f)$ is h-rpojective. By definition, $P'^{\bullet}$ is a complex of projectives, hence so is $Cone(f).$ It follows that we may replace $P^{\bullet}$ by $Cone(f).$ Thus, we may and will assume that $P^{\bullet}$ is acyclic.

We need to show that $P^{\bullet}$ is null-homotopic. Let us put $Z^n=\ker(d^n)=\im(d^{n-1})\subset P^n,$ where $d^i:P^i\to P^{i+1}$ are the components of $d.$ Let us put $m:=\hd(\text{Mod-}R).$ Then the exact sequence
$$0\to Z^n\to P^n\to P^{n+1}\to\dots\to P^{n+m}\to Z^{n+m+1}\to 0$$ implies that $Z^n$ is projective. Hence, the surjection $P^{n-1}\to Z^n$ splits, and by choosing the splittings we obtain the direct sum decompositions $P^n=Z^n\oplus Z^{n+1}$ for all $n\in\Z.$ Define the homotopy $h_n:P^n\to P^{n-1}$ to be the composition $P^n\to Z^n\to P^{n-1}.$ It is easy to see that $dh+hd=\id,$ hence $P^{\bullet}$ is null-homotopic.
\end{proof}

\section{Cross-effects and polynomial functors}
\label{sec:cross_effects_polynomial}

The cross-effects of functors between abelian categories were introduced in \cite{EM}. In this section we recall the definitions and basic properties of cross-effects and polynomial functors. 

\begin{defi}\label{defi:polynomial_map}Let $A$ and $B$ be abelian groups and $f:A\to B$ a set-theoretic map.

1) The cross-effects $(a_1\mid\dots\mid a_n)_f,$ $a_i\in A,$ $n\geq 0,$ are defined as follows. First, we put $()_f:=f(0)$ for $n=0.$ 

In the case $f(0)=0$ we put  $(a_1)_f:=f(a_1)$ for $n=1,$ and 
$$(a_1\mid\dots\mid a_n)_f:=(a_1\mid\dots \mid a_{n-2}\mid a_{n-1}+a_n)_f-(a_1\mid\dots\mid a_{n-1})_f-(a_1\mid\dots\mid a_{n-2}\mid a_n)_f$$
for $n\geq 2.$

In the case $f(0)\ne 0,$ we put $$(a_1\mid\dots\mid a_n)_f:=(a_1\mid\dots\mid a_n)_{f'}$$ for $n\geq 1,$
where $f'(a)=f(a)-f(0).$

2) The map $f$ is said to be polynomial of degree $\leq n$ (where $n\geq 0$) if for any collection of $(n+1)$ elements $a_1,\dots,a_{n+1}\in A$ we have
$$(a_1\mid\dots\mid a_{n+1})_f=0.$$

3) The map $f$ is said to be polynomial if for some $n\geq 0$ the map $f$ is polynomial of degree $\leq n.$
\end{defi}

In particular, a polynomial map of degree $\leq 0$ is just the constant map. Polynomial map of degree $\leq 1$ is a sum of a constant map and a homomorphism of abelian groups.

It is easy to check by induction that we have
$$(a_1\mid\dots\mid a_n)_f=\sum_{\substack{0\leq k\leq n\\ 1\leq i_1<\dots<i_k\leq n}}(-1)^{n-k}f(a_{i_1}+\dots+a_{i_k}).$$
Also, we have an equality
$$f(a_1+\dots+a_n)=\sum_{\substack{0\leq k\leq n\\ 1\leq i_1<\dots<i_k\leq n}}(a_{i_1}\mid\dots\mid a_{i_k})_f.$$

\begin{defi}Let $\cC$ and $\cD$ be additive categories, and $F:\cC\to \cD$ an abstract functor (not necessarily additive).

1) The functor $F$ is said to be polynomial of degree $\leq n$ if for any two objects $X,Y$ of $\cC$ the map $$F(X,Y):\Hom_{\cC}(X,Y)\to\Hom_{\cD}(F(X),F(Y))$$ is polynomial of degree $\leq n.$

2) The functor $F$ is said to be polynomial if for some $n\geq 0$ the functor $F$ is polynomial of degree $\leq n.$\end{defi}

A polynomial functor of degree $\leq 0$ is just a constant functor. A polynomial functor of degree $\leq 1$ is a direct sum of a constant functor and an additive functor.

Now we define the cross-effect functors.

\begin{defi}\label{defi:cross_effect} Let $\cC$ and $\cD$ be additive categories, and $F:\cC\to\cD$ an abstract functor. We assume that $\cD$ is Karoubi complete.

Let $X_1,\dots,X_n\in\cC$ be a collection of objects. For any sequence $1\leq i_1<\dots< i_k\leq n,$ we denote by $$\pi_{(i_1,\dots,i_k)}:\bigoplus\limits_{j=1}^{n}X_j\to \bigoplus\limits_{j=1}^{n}X_j$$ the projector onto the direct summand $\bigoplus\limits_{j=1}^k X_{i_j}.$

We define the cross-effect $\cross_{n}F(X_1,\dots,X_n)\in\widetilde{\cD}$ to be the image of the idempotent $$e:F(\bigoplus\limits_{j=1}^{n}X_j)\to F(\bigoplus\limits_{j=1}^{n}X_j),$$
given by the formula
$$e=\sum_{\substack{0\leq k\leq n\\ 1\leq i_1<\dots<i_k\leq n}} (-1)^{n-k}F(\pi_{(i_1,\dots,i_k)}).$$\end{defi}

\begin{remark}In the notation of Definition \ref{defi:cross_effect}, we always have a direct sum decomposition $F=F(0)\oplus F',$ where $F(0)$ is the constant functor, and the functor $F':\cC\to\cD$ satisfies $F'(0)=0.$ It is clear that for $n\geq 1$ we have a natural isomorphism $\cross_n F\cong \cross_n F'.$\end{remark}

\begin{prop}\label{prop:polynomiality_cross_effects}Let $\cC,$ $\cD$ and $F$ be as in Definition \ref{defi:cross_effect}.

1) The following are equivalent:

(i) $F$ is polynomial of degree $\leq n;$
 
(ii) for any objects $X_1,\dots,X_{n+1}\in\cC,$ we have $\cross_{n+1}F(X_1,\dots,X_{n+1})=0.$

2) For any objects $X_1,\dots,X_n\in\cC,$ we have a natural isomorphism
\begin{equation}\label{eq:decomposition_cross}F(\bigoplus\limits_{j=1}^n X_j)\cong\bigoplus_{\substack{0\leq k\leq n\\ 1\leq i_1<\dots<i_k\leq n}}\cross_k F(X_{i_1},\dots,X_{i_k}).\end{equation}
\end{prop}

\begin{proof}1) $(i)\Rightarrow(ii).$ Let $X_1,\dots,X_{n+1}\in\cC$ be a collection of objects. Put $X:=\bigoplus\limits_{j=1}^{n+1}X_i,$ and denote by $e_i:X\to X,$ $i=1,\dots,n+1,$ the projector onto the direct summand $X_i.$ By definition, the object $\cross_{n+1}F(X_1,\dots,X_{n+1})$ is the image of the idempotent $$(e_1\mid\dots\mid e_{n+1})_{F(X,X)}:F(X)\to F(X).$$ But by our assumption the map $F(X,X)$ is polynomial of degree $\leq n,$ hence we have $$\cross_{n+1}F(X_1,\dots,X_{n+1})=0.$$

$(ii)\Rightarrow(i).$ Let $X,Y\in\cC$ be a pair of objects, and $a_1,\dots,a_{n+1}\in\Hom(X,Y)$ a collection of morphisms. Let us denote by $A:X^{\oplus n+1}\to Y$ the morphism with components $a_i,$ and by $\Delta:X\to X^{\oplus n+1}$ the diagonal morphism. Also, we denote by $e_i:X^{\oplus n+1}\to X^{\oplus n+1},$ $i=1,\dots,n+1,$ the projection onto the $i$-th summand. 
The vanishing of $\cross_{n+1}F(X,\dots,X)$ implies that $(e_1\mid\dots\mid e_{n+1})_{F(X^{\oplus n+1},X^{\oplus n+1})}=0.$

Therefere, we have
\begin{multline*}(a_1\mid\dots\mid a_{n+1})_{F(X,Y)}=(A\circ e_1\circ \Delta\mid\dots\mid A\circ e_{n+1}\circ \Delta)_{F(X,Y)}\\
=F(A)\circ (e_1\mid\dots\mid e_{n+1})_{F(X^{\oplus n+1},X^{\oplus n+1})}\circ F(\Delta)=0,\end{multline*}

2) is proved in \cite{EM}, Theorem 9.1. 
\end{proof}

\section{Generalized cubical construction}
\label{sec:cubical_construction}

In this section we define the $n$-th cubical construction of an abelian group, following \cite{JM}, Section 1.

For a set $S,$ we denote by $\Z[S]$ the free abelian group generated by the elements of $S.$ Clearly, the functor $\Z[-]:\Sets\to \Ab$ from sets to abelian groups is left adjoint to the forgetful functor $\Forg:\Ab\to\Sets.$. We will apply the definitions of \cite{JM}, Section 1, to the functor $\Z[-]\circ\Forg:\Ab\to\Ab.$ 

Let $A$ be an abelian group, and $n\geq 1$ an integer. We first define the non-negative chain complex $Q'^n_{\bullet}(A).$ We put
$$Q'^n_k(A):=\Z[A^{[n]^k}],$$
where $[n]=\{0,1,\dots,n\}.$

The differential on $Q'^n_{\bullet}(A)$ is defined as follows.
For any $n_1,\dots,n_l\in\Z_{>0},$ let us treat the elements of $A^{[n_1]\times\dots\times[n_l]}$ as functions from $[n_1]\times\dots\times[n_l]$ to $A.$ For any $1\leq j\leq l,$ we define the maps
$$\alpha_j,\beta_j,\gamma_j:A^{[n_1]\times\dots\times[n_l]}\to A^{[n_1]\times\dots[n_j-1]\dots\times[n_l]}$$
by the formulas
$$\alpha_j(f)(i_1,\dots,i_l)=f(i_1,\dots,i_l);$$
$$\beta_j(f)(i_1,\dots,i_l)=\begin{cases}f(i_1,\dots,i_l) & \text{if }i_j<n_j-1;\\
f(i_1,\dots,n_j,\dots,i_l) & \text{if }i_j=n_j-1;\end{cases}$$
$$\gamma_j(f)(i_1,\dots,i_l)=\begin{cases}f(i_1,\dots,i_l) & \text{if }i_j<n_j-1;\\
f(i_1,\dots,n_j-1,\dots,i_l)+f(i_1,\dots,n_j,\dots,i_l) & \text{if }i_j=n_j-1.\end{cases}$$
We define the map
$$\delta_j:\Z[A^{[n_1]\times\dots\times[n_l]}]\to \Z[A^{[n_1]\times\dots[n_j-1]\dots\times[n_l]}]$$
by the formula
$$\delta_j:=\Z[\gamma_j]-\Z[\alpha_j]-\Z[\beta_j].$$
Note that if for some $j$ we have $n_j=0,$ then we have an identification $$[n_1]\times\dots\times[n_l]=[n_1]\times\dots[n_{j-1}]\times[n_{j+1}]\times\dots\times [n_l].$$
The differential
$$\delta':Q'^n_k(A)=\Z[A^{[n]^k}]\to \Z[A^{[n]^{k-1}}] =Q'^n_{k-1}(A)$$
is defined by the formula
$$\delta':=\sum\limits_{j=1}^k (-1)^j\delta_j^n.$$ It is shown in \cite{JM}, Proposition 1.4, that we have $\delta'^2=0,$ so $\delta'$ is indeed a differential.

The basis element $[f]$ of $Q'^n_k(A)$ is called slab if $f=0$ for $k=0,$ and an $i$-slab, $i=1,\dots,k,$ for $k\geq 1,$ if for some $j\in[n]$ we have that $$f(x_1,\dots,x_k)=0\quad \text{if}\quad x_i=j.$$
 Further, $[f]$ is called an $i$-diagonal, $i=1,\dots,k-1,$ if we have that
$$f(x_1,\dots,x_k)=0\quad\text{if}\quad x_i\ne x_{i+1}.$$  We denote by $N^n_k(A)\subset Q'^n_k(A)$ the abelian subgroup generated by all slabs and diagonals. By \cite{JM}, Lemma 1.14, $N^n_{\bullet}(A)\subset Q'^n_{\bullet}(A)$ is a subcomplex. We define the $n$-th cubical construction of $A$ by the formula
\begin{equation}\label{eq:cubical_construction}Q^n_{\bullet}(A):=Q'^n_{\bullet}(A)/N^n_{\bullet}(A).\end{equation}

 For each $k\geq 1,$ we have the map
$$q_n:Q^n_k(A)\to Q^{n-1}_k(A),\quad q_n=\delta_1\circ\delta_2\dots\circ\delta_k.$$
By \cite{JM}, Lemmas 1.7 and 1.15, these maps
define a functorial morphism of complexes $q_n:Q^n_{\bullet}(A)\to Q^{n-1}_{\bullet}(A).$

Given two abelian groups $A$ and $B,$ for any $k,l\geq 0$ we have the following morphism:
\begin{equation}\label{eq:mu_A,B,k,l}\mu_{A,B}:Q^n_k(A)\otimes Q^n_l(B)\to Q^n_{k+l}(A\otimes B),\quad \mu_{A,B}([f]\otimes [g])=[f\cdot g],\end{equation}
where $$f\cdot g(x_1,\dots,x_k,y_1,\dots,y_l):=f(x_1,\dots,x_k)\otimes g(y_1,\dots,y_l).$$ It is clear that the maps \eqref{eq:mu_A,B,k,l} are well-defined, and altogether they define a morphism of complexes
\begin{equation}\label{eq:mu_A,B}\mu_{A,B}:Q^n_{\bullet}(A)\otimes Q^n_{\bullet}(B)\to Q^n_{\bullet}(A\otimes B).\end{equation}

We also have an obvious morphism of complexes
\begin{equation}\label{eq:epsilon}\epsilon:\Z\to Q^n_{\bullet}(\Z),\quad \epsilon(1):=[1]\in Q^n_0(\Z).\end{equation}

\begin{prop}\label{prop:Q^n_lax_monoidal}1) The maps \eqref{eq:mu_A,B} and \eqref{eq:epsilon} define on $Q^n_{\bullet}(-)$ the structure of a lax monoidal functor from the category of abelian groups to the category of non-negative chain complexes of abelian groups.

2) The natural transformations $q_n:Q^n_{\bullet}(-)\to Q^{n-1}_{\bullet}(-)$ are compatible with the structures of lax monoidal functors.\end{prop}

\begin{proof}Straightforward checking.\end{proof}

Proposition \ref{prop:Q^n_lax_monoidal} allows us to extend the functor $Q^n_{\bullet}(-)$ to pre-additive categories.

\begin{defi}\label{defi:Q^n_on_categories}Let $\cC$ be a pre-additive category. We define the DG category $Q^n_{\bullet}(\cC)$ (over $\Z$) as follows:

- $Ob(Q^n_{\bullet}(\cC))=Ob(\cC);$

- $Q^n_{\bullet}(\cC)(X,Y)=Q^n_{\bullet}(\Hom_{\cC}(X,Y))$ for $X,Y\in Ob(\cC);$

- the composition maps $Q^n_{\bullet}(\cC)(Y,Z)\otimes Q^n_{\bullet}(\cC)(X,Y)\to Q^n_{\bullet}(\cC)(X,Z)$ are obtained from the lax monoidal structure on the functor $Q^n_{\bullet}(A).$\end{defi}

An associative (unital) ring $R$ can be considered as a pre-additive category with one object. In this case, $Q^n_{\bullet}(R)$ is a DG ring.

\begin{prop}For any pre-additive category $\cC,$ we have a functor $q_n(\cC):Q^n_{\bullet}(\cC)\to Q^{n-1}_{\bullet}(\cC),$ which is identity on objects, and on morphisms is given $q_n:Q^n_{\bullet}(-)\to Q^{n-1}_{\bullet}(-).$\end{prop}

\begin{proof} This is a direct consequence of Proposition \ref{prop:Q^n_lax_monoidal}, 2).\end{proof}

For any $1\leq i\leq k,$ $0\leq j\leq n,$ we put
$$S_j^i:=\{(x_1,\dots,x_k)\in[n]^k\mid x_i\ne j\}\subset[n]^k.$$
Also, for $1\leq i\leq k-1,$ we put $$D_i:=\{(x_1,\dots,x_k)\in[n]^k\mid x_i=x_{i+1}\}\subset[n]^k.$$

We call a subset $S\subset[n]^k$ {\it non-degenerate} if it is not contained in any of the subsets $S_j^i$ and $D_i.$

\begin{prop}\label{prop:Q^n_cross_effect}1) For any collection of abelian groups $A_1,\dots,A_{n+1},$ the complex
$$\cross_{n+1}Q^n_{\bullet}(A_1,\dots,A_{n+1})$$
is acyclic.

2) For an abelian group $A,$ we have a natural direct sum decomposition
$$Q^n_k(A)\cong\bigoplus\limits_{S\subset[n]^k}\cross_{|S|}\Z[-](A,\dots,A),$$
where the sum is taken over all non-degenerate subsets $S.$\end{prop}

\begin{proof} 1)  is proved in \cite{JM}, Theorem 3.1.

2) By Proposition \ref{prop:polynomiality_cross_effects} 2), we have a similar decomposition for $Q'^n_k(A),$ where the sum is taken over all subsets $S\subset[n]^k.$ Further, the subgroup $N^n_k(A)\subset Q'^n_k(A)$ by definition is exactly the direct sum of those factors which correspond to degenerate subsets $S\subset[n]^k.$ This implies the desired decomposition of $Q^n_k(A).$\end{proof}

                                                                                                                                                                               \section{DG quotients and the functors $Q^n_{\bullet}(-).$}
\label{sec:DG_quotients_and_Q^n}                                                                                                                                                                               
                                                                                                                                                                               For a pre-additive category, $\cD,$ we denote by $\widetilde{\cD}$ its Karoubian envelope. We also use similar notation for DG categories. Here is the precise definition.
                                                                                                                                                                               
                                                                                                                                                                               \begin{defi}\label{defi:Karoubian_for_DGcat}Let $\cD$ be a DG category. The DG category $\widetilde{\cD}$ is defined as follows. Its objects are pairs $(X,e),$ where $X\in Ob(\cD)$ and $e\in \cD^0(X,X)$ is a closed idempotent, i.e. we have $d(e)=0$ and $e^2=e.$
                                                                                                                                                                               
The morphisms in $\widetilde{\cD}$ are defined by the formula
$$\widetilde{\cD}((X,e),(Y,f))=f\cD(X,Y)e,$$
and the composition in $\widetilde{\cD}$ is induced by the composition in $\cD.$ 
\end{defi}                                                                                                                                                                               
                                                                                                                                                                        Note that Karoubian envelope of DG categories is not invariant under quasi-equivalence (since we take images of actual idempotents, not homotopy idempotents). If $\cD$ is small, then the natural functor $\cD\to\widetilde{\cD},$ $X\mapsto (X,\id_X)$ is a Morita equivalence of DG categories.

                                                                                                                                                                        \begin{defi}1) For any category $\cC,$ the pre-additive category $\Z[\cC]$ is defined as follows. It has the same objects as $\cC,$ and the morphisms are defined by the formula
$$\Z[\cC](X,Y)=\Z[\cC(X,Y)],\quad X,Y\in Ob(\cC).$$
The composition in $\Z[\cC]$ is induced by the composition in $\cC.$

2) For a pre-additive category $\cC,$ the category $\Z[\bar{\cC}]$ is defined as follows. Again, it has the same objects as $\cC,$ and the morphisms are defined by $$\Z[\bar{\cC}](X,Y)=\Z[\cC(X,Y)]/\Z[0].$$
The composition in $\Z[\bar{\cC}]$ is induced by the composition in $\cC.$\end{defi}

To avoid confusion, in some cases we will write $[X]$ for the object of $\Z[\cC]$ corresponding to $X\in Ob(\cC).$
    
Let $\cC$ be a small additive category. We have a natural non-additive functor $\eta:\cC\to\Z[\cC].$ For a collection of objects $X_1,\dots,X_n\in\cC,$ we put $$\cross_n(X_1,\dots,X_n):=\cross_n \eta(X_1,\dots,X_n)\in\widetilde{\Z[\cC]}.$$
Further, we denote by $\cross_n\cC\subset\widetilde{\Z[\cC]}$ the Karoubi closed additive subcategory generated by the objects $\cross_n(X_1,\dots,X_n).$
We have an obvious functor
$$F_n:\Z[\cC]\to Q^n_{\bullet}(\cC),$$
which is identity on objects.

Note that we have a natural equivalence $\cross_1\cC\simeq\Z[\bar{\cC}].$ Further, for the object $[0]\in\Z[\cC]$ we have $\End_{\Z[\cC]}([0])=\Z.$ For each $X\in Ob(\Z[\cC]),$ we have 
\begin{equation}\label{eq:decomp_of_Z[C]}X\cong [0]\oplus\cross_1(X),\quad \Hom([0],\cross_1(X))=\Hom(\cross_1(X),[0])=0.\end{equation}
It follows that we have an equivalence of additive categories
$$\widetilde{\Z[\cC]}\simeq P(\Z)\oplus\widetilde{\Z[\bar{\cC}]},$$ where for any associative ring $R$ we denote by $P(R)$ the category of finitely generated projective (right) $R$-modules. In particular, we have an equivalence of triangulated categories
\begin{equation}\label{eq:decomp_of_Perf_Z[C]}\Perf(\Z[\cC])\simeq \Perf(\Z)\oplus\Perf(\Z[\bar{\cC}]).\end{equation}

Note that for $n>m$ we have an inclusion $\cross_n\cC\subset\cross_m\cC,$ hence for $n\geq 1$ we can treat $\cross_n\cC$ as a subcategory of $\widetilde{\Z[\bar{\cC}]}\simeq\cross_1\cC.$

We have a natural functor
$$F_n:\Z[\cC]\to Q^n_{\bullet}(\cC),$$
which is identity on objects. By definition of the functors $Q^n_{\bullet}(-),$ we have $F_n([0])=0,$ hence we have a functor $\overline{F_n}:\Z[\bar{\cC}]\to Q^n_{\bullet}(\cC),$ which we denote by the same symbol.

\begin{theo}\label{th:exact_seq_Q^n_cr_n+1} 
The functor $\overline{F_n}:\Z[\bar{\cC}]\to Q^n_{\bullet}(\cC)$ induces a quasi-equivalence
$$\widetilde{\Z[\bar{\cC}]}/\cross_{n+1}\cC\simeq \widetilde{Q^n_{\bullet}(\cC)}.$$
Equivalently, we have a short exact sequence of triangulated categories, up to direct summands:
\begin{equation}\label{eq:exact_seq_Q^n_cr_n+1}0\to \Perf(\cross_{n+1}\cC)\to\Perf(\Z[\bar{\cC}])\stackrel{\bL \overline{F_n}^*}{\to}\Perf(Q^n_{\bullet}(\cC))\to 0.\end{equation}
\end{theo}

\begin{proof}For any collection of objects $[X_1],\dots,[X_{n+1}]\in\Z[\bar{\cC}],$ and any object $Y\in Q^n_{\bullet}(\cC),$ we have an isomorphism
$$\bR\Hom_{Q^n_{\bullet}(\cC)}(\bL \overline{F_n}^*(\cross_{n+1}(X_1,\dots,X_{n+1})),Y)\cong \cross_{n+1}Q^n_{\bullet}(\Hom_{\cC}(X_1,Y),\dots,\Hom_{\cC}(X_{n+1},Y)).$$
By Proposition \ref{prop:Q^n_cross_effect} 1), this complex is acyclic. Hence, the functor $$\bL \overline{F_n}^*:\Perf(\Z[\cC])\to\Perf(Q^n_{\bullet}(\cC))$$
vanishes on the subcategory $\Perf(\cross_{n+1}\cC)\subset \Perf(\Z[\bar{\cC}]).$

It follows that we can apply Proposition \ref{prop:criterion_for_quotient} to the functor $\widetilde{\overline{F_n}}:\widetilde{\Z[\bar{\cC}]}\to \widetilde{Q^n_{\bullet}(\cC)}$ and to the full subcategory $\cross_{n+1}(\cC)\subset\widetilde{\Z[\bar{\cC}]}.$

By Proposition \ref{prop:Q^n_cross_effect} 2), for any $X\in\cC$ we have a direct sum decomposition
$$Q^n_k(\Hom_{\cC}(-,X))\cong \bigoplus\limits_{S\subset[n]^k}\Hom_{\widetilde{\Z[\cC]}}(-,\cross_{|S|}(X,\dots,X)),$$
where $S$ runs over all non-degenerate subsets. By the pigeonhole principle, all non-degenerate subsets $S\subset[n]^k$ have cardinality at least $n+1.$ It follows that for $k>0$ the object $Q^n_k(\Hom_{\cC}(-,X))\in D(\Z[\bar{\cC}])$ is contained in $\Perf(\cross_{n+1}\cC)\subset\Perf(\Z[\bar{\cC}]).$ Therefore, the object $$\Cone(h_X\to \overline{F_n}_*(\bL \overline{F_n}^*(h_X)))\in D(\Z[\bar{\cC}])$$ is contained in the localizing subcategory generated by $\Perf(\cross_{n+1}\cC).$ By Proposition \ref{prop:criterion_for_quotient}, we obtain a quasi-equivalence $\widetilde{\Z[\bar{\cC}]}/\cross_{n+1}(\cC)\simeq \widetilde{Q^n_{\bullet}(\cC)}.$ This proves the theorem.
\end{proof}

\begin{cor}\label{cor:F_n_homological_epi}1) For a small additive category $\cC,$ the functors $F_n:\Z[\cC]\to Q^n_{\bullet}(\cC)$ and $\overline{F_n}:\Z[\bar{\cC}]\to Q^n_{\bullet}(\cC)$ are homological epimorphisms.

2) For $n\geq 1$ we have $F_n=q_{n+1}(\cC)F_{n+1}$ and similarly for $\overline{F_n}.$ In particular, $q_n(\cC)$ is also a homological epimorphism for $n\geq 2.$\end{cor}

\begin{proof}1) For the functor $\overline{F_n}:\Z[\bar{\cC}]\to Q^n_{\bullet}(\cC),$ this follows immediately from Theorem \ref{th:exact_seq_Q^n_cr_n+1} and Proposition \ref{prop:DG_quotient_is_homological_epi}.

The functor $F_n:\Z[\cC]\to Q^n_{\bullet}(\cC)$ is the composition of the projection $\Z[\cC]\to\Z[\bar{\cC}]$ (which is a homological epimorphism by \eqref{eq:decomp_of_Z[C]}) and the functor $\overline{F_n}:\Z[\bar{\cC}]\to Q^n_{\bullet}(\cC).$ It follows that the functor $F_n:\Z[\cC]\to Q^n_{\bullet}(\cC)$ is also a homological epimorphism.

2) is straightforward.\end{proof}

For a small additive category $\cC,$ we denote by $\Pol^{\leq n}(\bar{\cC})\subset\text{Mod-}\Z[\bar{\cC}]$ the full subcategory of polynomial functors $\cC^{op}\to \text{Mod-}\Z$ of degree $\leq n,$ which vanish on $0\in\cC.$ We denote by $D_{\Pol^{\leq n}(\bar{\cC})}(\Z[\bar{\cC}])\subset D(\Z[\bar{\cC}])$ the full subcategory of complexes whose cohomology is contained in $\Pol^{\leq n}(\bar{\cC}).$

\begin{cor}\label{cor:R^n_gen_of_Q^n(P(R))}1) Let $\cC$ be a small additive category. We have an equivalence of triangulated categories
$$D_{\Pol^{\leq n}(\bar{\cC})}(\Z[\bar{\cC}])\simeq D(Q^n_{\bullet}(\cC)).$$

2) Let $R$ be an associative ring. We have an equivalence of triangulated categories
$$D_{\Pol^{\leq n}(\bar{P(R)})}(\Z[\overline{P(R)}])\simeq D(Q^n_{\bullet}(M_n(R))).$$\end{cor}

\begin{proof}1) Indeed, the subcategory $D_{\Pol^{\leq n}(\bar{\cC})}(\Z[\bar{\cC}])\subset D(\Z[\bar{\cC}])$ is exactly the right orthogonal to $\Perf(\cross_{n+1}\cC).$ Hence, by Theorem \ref{th:exact_seq_Q^n_cr_n+1} we have an equivalence
$$D_{\Pol^{\leq n}(\bar{\cC})}(\Z[\bar{\cC}])\simeq D(Q^n_{\bullet}(\cC)).$$

2) By Yoneda Lemma, we have an isomorphism of DG algebras
$$Q^n_{\bullet}(M_n(R))\cong\End_{Q^n_{\bullet}(P(R))}(h_{R^n}).$$
It remains to show that the object $h_{R^n}$ generates the category $\Perf(Q^n_{\bullet}(P(R))).$ But for any $m\geq 0$ the object $h_{R^m}\in\Perf(Q^n_{\bullet}(P(R)))$ is a finite direct sum of objects $\bL\overline{F_n}^*\cross_{k}(R,\dots,R),$ where $1\leq k\leq n.$ These objects are in turn the direct summands of $h_{R^n}.$ Hence, the object $h_{R^n}\in\Perf(Q^n_{\bullet}(P(R)))$ is indeed a generator. This proves 2).\end{proof}

\section{Mixed complexes}
\label{sec:mixed complexes}

Fix some commutative base ring $A.$

\begin{defi}\label{defi:mixed_complexes} 1) A mixed complex over $A$ is a triple $(\cK^{\bullet},b,\delta),$ where $\cK^{\bullet}$ is a $\Z$-graded $A$-module, $b$ (resp. $\delta$) is a differential on $\cK^{\bullet}$ of degree $1$ (resp. $-1$), satisfying
$$b\delta+\delta b=0.$$

2) An $A_{\infty}$-morphism of mixed complexes
$$f:(\cK^{\bullet},b,\delta)\to (\cK'^{\bullet},b^\prime,\delta^\prime)$$ is a sequence of homogeneous morphisms of graded $A$-modules $f_n:\cK^{\bullet}\to \cK'^{\bullet},$ $\deg(f_n)=2-2n,$ $n\geq 1,$ such that
$$b^\prime f_1=f_1b,\quad b'f_n+\delta'f_{n-1}-f_nb-f_{n-1}\delta=0,\quad n\geq 2.$$

3) An $A_{\infty}$-morphism $f:(\cK^{\bullet},b,\delta)\to (\cK'^{\bullet},b^\prime,\delta^\prime)$ is called a quasi-isomorphism if it induces a quasi-isomorphism of complexes
$$f_1:(\cK^{\bullet},b)\to (\cK'^{\bullet},b^\prime).$$\end{defi}

The composition of $A_{\infty}$-morphisms of mixed complexes is defined by the formula
$$(f\circ g)_n=\sum\limits_{i=1}^n f_i\circ g_{n+1-i},\quad n\geq 1.$$
We denote by $D(\Mix_A)$ the localization of the category of mixed complexes (and $A_{\infty}$-morphisms) by the class of quasi-isomorphisms.

\begin{remark} Let us consider the DG $A$-algebra $\Lambda_A\la \epsilon\ra:=A[\epsilon]/(\epsilon^2),$ where $\epsilon$ is a variable of degree $-1,$ and $d(\epsilon)=0.$ It is clear that a mixed complex $(\cK^{\bullet},b,\delta)$ can be considered as a DG $\Lambda_A\la \epsilon\ra$-module, in which $b$ is a differential, and $\epsilon$ acts by $\delta.$ Conversely, any DG $\Lambda_A\la \epsilon\ra$-module can be considered as a mixed complex. An $A_{\infty}$-morphism of mixed complexes is just a strictly unital morphism of $A_{\infty}$-modules over $\Lambda_A\la \epsilon\ra.$ Finally, the notion of a quasi-isomorphism is obviously compatible with this interpretation of mixed complexes.

For any DG algebra, its derived categories of DG modules and of $A_{\infty}$-modules are naturally equivalent. Hence, we have $D(\Mix_A)\simeq D(\Lambda_A\la\epsilon\ra).$\end{remark}

\begin{defi}\label{defi:totalizations}1) For any mixed complex $(\cK^{\bullet},b,\delta),$ we denote by
$\Tot^L(\cK^{\bullet})=\Tot^{L}(\cK^{\bullet},b,\delta)$ the $\Z/2$-graded complex with components $$(\Tot^{L}(\cK^{\bullet}))^0=\colim\limits_{k\geq 0}\prod\limits_{n\leq k}K^{2n},\quad (\Tot^{L}(\cK^{\bullet}))^1=\colim\limits_{k\geq 0}\prod\limits_{n\leq k}K^{2n+1},$$
with differential $b+\delta.$ We call this complex a Laurent-type totalization of the mixed complex $(\cK^{\bullet},b,\delta).$

2) We say that a mixed complex $(\cK^{\bullet},b,\delta)$ is bounded below if $\cK^n=0$ for $n<<0.$ In this case, we see that $\Tot^L(\cK^{\bullet})=\Tot^{\oplus}(\cK^{\bullet}),$ where we denote by
$\Tot^{\oplus}(\cK^{\bullet})$ the $\Z/2$-graded complex with components $$(\Tot^{\oplus}(\cK^{\bullet}))^0=\bigoplus\limits_{n\in\Z}K^{2n},\quad (\Tot^{\oplus}(\cK^{\bullet}))^1=\bigoplus\limits_{n\in\Z}K^{2n+1},$$
and with differential $b+\delta.$

3) We say that a mixed complex $(\cK^{\bullet},b,\delta)$ is bounded above if $\cK^n=0$ for $n>>0.$ In this case, we see that $\Tot^L(\cK^{\bullet})=\Tot^{\Pi}(\cK^{\bullet}),$ where we denote by
$\Tot^{\Pi}(\cK^{\bullet})$ the $\Z/2$-graded complex with components $$(\Tot^{\Pi}(\cK^{\bullet}))^0=\prod\limits_{n\in\Z}K^{2n},\quad (\Tot^{\Pi}(\cK^{\bullet}))^1=\prod\limits_{n\in\Z}K^{2n+1},$$
and with differential $b+\delta.$\end{defi}



Given a mixed complex $(\cK^{\bullet},b,\delta),$ put $C^{\bullet}:=\Tot^L(\cK^{\bullet},b,\delta).$ The increasing filtration $F_{\bullet}C^{\bullet}$ is given by the formula
\begin{equation}\label{eq:filtration_on_Tot}F_n C^0=\prod\limits_{2k\leq n}\cK^{2k},\quad F_n C^1=\prod\limits_{2k+1\leq n}\cK^{2k+1}.\end{equation}
Clearly, it satisfies $d(F_nC^{\bullet})\subset F_{n+1}C^{\bullet+1}.$ Moreover, we have $F_n C^{\bar{n}}=F_{n+1}C^{\bar{n}}.$ We define the filtration on $F_{\bullet}H^{\bullet}(C^{\bullet})$ by the formula
\begin{equation}\label{eq:filtration_on_H_of_Tot}F_nH^{k}(C^{\bullet})=\im(\ker(d:F_nC^k\to F_{n+1}C^{k+1})\to H^k(C^{\bullet})).\end{equation} This filtration is exhaustive and also satisfies $F_n H^{\bar{n}}(C^{\bullet})=F_{n+1}H^{\bar{n}}(C^{\bullet}).$ We define the graded $A$-module $\gr_{\bullet}^FH^{\bullet}(C^{\bullet})$ by the formula
$$\gr_n^FH^{\bullet}(C^{\bullet})=F_n H^{\bar{n}}(C^{\bullet})/F_{n-2}H^{\bar{n}}(C^{\bullet}).$$

From now on we denote by $u$ a formal variable of cohomological degree $2.$ Let us identify $\Z/2$-graded complexes of $A$-modules with DG $A[u^{\pm 1}]$-modules. We need the following useful observation.

\begin{lemma}\label{lem:interpretation_of_totalization} The DG algebra $\bR\Hom_{\Lambda_A\la\epsilon\ra}(A,A)$ is naturally quasi-isomorphic to $A[u].$ Further, for any mixed complex $(\cK^{\bullet},b,\delta)$ we have an isomorphism
$$\Tot^L(\cK^{\bullet})\cong\bR\Hom_{\Lambda_A\la\epsilon\ra}(A,\cK^{\bullet})\stackrel{\bL}{\otimes}_{A[u]}A[u^{\pm 1}]$$ in $D(A[u^{\pm 1}]).$\end{lemma}

\begin{proof}To see this explicitly, let us take the DG category $\cC$ of strictly unital $A_{\infty}$-modules over $\Lambda_A\la\epsilon\ra.$ Then one can easily check that we have an isomorphism of DG algebras $\Hom_{\cC}(A,A)\cong A[u].$ Further, we have a natural isomorphism of DG $A[u]$-modules
$$\Hom_{\cC}(A,\cK^{\bullet})=(\cK^{\bullet}[[u]],b+u\delta).$$
Since $A[u^{\pm 1}]$ is an h-flat DG $A[u]$-module, we have an isomorphism
\begin{equation}
\label{eq:ext_of_scalars_Laurent}
(\cK^{\bullet}[[u]],b+u\delta)\stackrel{\bL}{\otimes}_{A[u]}A[u^{\pm 1}]\cong (\cK^{\bullet}((u)),b+u\delta)
\end{equation}
in $D(A[u^{\pm 1}]).$ It remains to notice that the RHS of \eqref{eq:ext_of_scalars_Laurent} corresponds to $\Tot^L(\cK^{\bullet})$ under the identification of DG $k[u^{\pm 1}]$-modules with $\Z/2$-graded complexes of $A$-modules..\end{proof}

\begin{prop}\label{prop:functoriality_of_Tot} The totalization $\Tot^L$ is functorial under $A_{\infty}$-morphisms of mixed complexes. Moreover, it sends $A_{\infty}$-quasi-isomorphisms to quasi-isomorphisms.\end{prop}

\begin{proof} This follows easily from Lemma \ref{lem:interpretation_of_totalization} and its proof.\end{proof}

In Appendix \ref{sec:appendix_on_sp_seq} we define the notion of $\Z$-graded spectral sequences (Definition \ref{defi:Z-graded_spectral_sequence}) and study some of their properties. One of the sources of $\Z$-graded spectral sequences is a mixed complex. 

\begin{prop}\label{prop:functoriality_of_sp_seq} Given a mixed complex $(\cK^{\bullet},b,\delta),$ we have a $\Z$-graded spectral sequence starting with
\begin{equation}\label{eq:sp_seq_mixed_complex}E_1^n=H^n(\cK^{\bullet},b).\end{equation}
It is functorial under $A_{\infty}$-morphisms. Moreover, $A_{\infty}$-quasi-isomorphisms of mixed complexes induce isomorphisms of the spectral sequences.\end{prop}

\begin{proof}See Corollary \ref{cor:sp_seq_from_mixed_appendix} below.\end{proof}

Convergence of $\Z$-graded spectral sequences is also defined in Appendix \ref{sec:appendix_on_sp_seq} (Definition \ref{defi:convergence_of_sp_seq}).

\begin{prop}\label{prop:convergence_of_sp_seq} Suppose that a mixed complex $(\cK^{\bullet},b,\delta)$ is bounded below. Then the spectral sequence \eqref{eq:sp_seq_mixed_complex} converges to $H^{\bullet}(\Tot^{\oplus}(K^{\bullet})).$\end{prop}

\begin{proof}See Corollary \ref{cor:convergence_of_sp_seq_from_mixed}.\end{proof}

\begin{prop}\label{prop:product_on_Tot_of_DGAs} Let $(\cA,\delta)$ be a mixed complex, where $\cA$ is a DG algebra. Let us assume that $\delta$ is a derivation on $\cA.$

1) The totalization $\Tot^L(\cA,\delta)$ is naturally a ($\Z/2$-graded) DG algebra. The filtration on $H^{\bullet}(\Tot^L(\cA,\delta))$ given by \eqref{eq:filtration_on_H_of_Tot} satisfies $F_n\cdot F_m\subset F_{n+m}.$ In particular, the associated graded $\gr_{\bullet}^F(H^{\bullet}(\Tot^L(\cA,\delta)))$ is a graded algebra. 

2) Let $\{(E_r^{\bullet},d_r)\}_{r\geq 1}$ be the associated spectral sequence. Then each $E_r^{\bullet}$ is naturally a graded algebra, and $d_r$ is a derivation on $E_r^{\bullet}.$ Moreover, the product on $E_1^{\bullet}=H^{\bullet}(\cA)$ is induced by that on $\cA,$ and the product on $E_{r+1}^{\bullet}$ is induced by that on $E_r^{\bullet}.$

3) If $\cA$ is bounded below, then we have an isomorphism of graded algebras
$$E_{\infty}^{\bullet}\cong \gr_{\bullet}^FH^{\bullet}\Tot^L(\cA,\delta).$$\end{prop}

\begin{proof}This is proved in Appendix \ref{sec:appendix_on_sp_seq}.\end{proof}

Now let $B_{\bullet}$ be a non-positive DG algebra (concentrated in non-negative homological degrees) over $A.$

\begin{defi}\label{defi:super_commutative_DGA} We say that $B_{\bullet}$ is strongly super-commutative if the following conditions are satisfied:

1) (super-commutativity) for any homogeneous $x,y\in \cB_{\bullet},$ we have $xy=(-1)^{|x||y|}yx;$

2) for any odd homogeneous $x\in\cB_{2k+1},$ we have $x^2=0.$\end{defi}

Note that if $B_{odd}$ has no $2$-torsion, then the condition 2) of Definition \ref{defi:super_commutative_DGA} is redundant.

From now on in this section we assume that the DGA $B_{\bullet}$ is strongly super-commutative.

Further, we say that $B_{\bullet}$ is a DG algebra with divided powers if the ideal $B_{even,\geq 2}\subset B_{even}$ has a divided power structure $\{\gamma_n\}_{n\geq 0},$ and moreover
$$\gamma_n(B_{2k})\subset B_{2nk},$$
$$d(\gamma_n(x))=d(x)\gamma_{n-1}(x),\quad x\in B_{2k},\quad\text{for }n,k\geq 1.$$

Let $M$ be a left DG module over $B_{\bullet}.$ Then for any closed element $\alpha\in B_1,$ we have a structure of a mixed complex on $M,$ such that the differential $\delta=\delta_{\alpha}$ is the multiplication by $\alpha.$

\begin{lemma}\label{lem:q_is_divided_powers}Let $B_{\bullet}$ be a non-positive DG algebra over $A,$ which is strongly super-commutative and has a divided power structure. Let $\alpha_1,\alpha_2\in\cB_1$ be homologous closed elements, and $M$ a DG module over $B_{\bullet}$. Then the identity morphism $M\to M$ can be extended to an $A_{\infty}$-quasi-isomorphism of mixed complexes $(M,\delta_{\alpha_1})\stackrel{\sim}{\to}(M,\delta_{\alpha_2}).$\end{lemma}

\begin{proof}Take some $\beta\in B_2$ such that $d(\beta)=\alpha_1-\alpha_2.$ Consider the sequence of maps of graded $A$-modules
$\{f_n:M\to M\}_{n\geq 1}$ given by the formula
$$f_n(m)=\gamma_{n-1}(\beta)\cdot m.$$
It is straightforward to check that these maps define an $A_{\infty}$-morphism
$$f:(M,\delta_{\alpha_1})\to (M,\delta_{\alpha_2}).$$
Since $f_1$ is the identity map, the $A_{\infty}$-morphism $f$ is a quasi-isomorphism.\end{proof}

\section{Hochschild and Mac Lane (co)homology of the second kind}
\label{sec:HH_and_HML_second_kind}

\subsection{Hochschild (co)homology}
\label{ssec:HH}

Again, we fix a commutative base ring $A.$

All DG categories in this section will be assumed to be h-projective over $A$ unless otherwise stated (see Definition \ref{defi:h-projective_DGcat}).

Let $\cB$ be a small DG category over $A.$ Let $M$ be a $\cB\otimes\cB^{op}$-module. We recall the Hochschild chain and cochain complexes of $\cB$ with coefficients in $M.$

The complex $\Hoch_{\bullet}(\cB,M)$ is defined as follows. In the cohomological grading, it is the sum-total complex of the $\Z\times\Z_{\leq 0}$-graded bicomplex, whose $(-n)$-th column is the complex
$$\bigoplus\limits_{X_0,\dots,X_n\in\cB}M(X_n,X_0)\otimes_A \cB(X_{n-1},X_n)\otimes_A\dots\otimes_A\cB(X_0,X_1).$$
The horizontal differential is given by the formula
\begin{multline*}\partial(m\otimes a_n\otimes\dots\otimes a_1)=ma_n\otimes a_{n-1}\otimes\dots\otimes a_1-m\otimes a_na_{n-1}\otimes a_{n-2}\otimes\dots\otimes a_1\\
\dots+(-1)^{n+|a_1|(|m|+|a_n|+\dots+|a_2|)}a_1m\otimes a_n\otimes\dots\otimes a_2.\end{multline*}
We denote by $b$ the differential on $\Hoch_{\bullet}(\cB,M).$
Recall that we have an isomorphism
$$H_{\bullet}(\Hoch_{\bullet}(\cB,M))\cong HH_{\bullet}(\cB,M)\cong \Tor_n^{\cB\otimes\cB^{op}}(I_{\cB},M),$$
where $I_{\cB}$ is the diagonal $\cB\otimes\cB^{op}$-module.

The complex $\Hoch^{\bullet}(\cB,M)$ is the product-total complex of the $\Z\times\Z_{\geq 0}$-graded bicomplex, whose $n$-th column is the complex
$$\prod\limits_{X_0,\dots,X_n\in\cB}\Hom_A(\cB(X_{n-1},X_n)\otimes_A\dots\otimes_A\cB(X_0,X_1),M(X_0,X_n)).$$
The horizontal differential is given by the formula
\begin{multline*}(\partial f)(a_{n+1},\dots,a_1)=(-1)^{|f|\cdot|a_{n+1}|}a_{n+1} f(a_n,\dots,a_1)-f(a_{n+1}a_n,a_{n-1},\dots,a_1)\\\dots+(-1)^n f(a_{n+1},\dots,a_3,a_2a_1)+(-1)^{n+1}f(a_{n+1},\dots,a_2)a_1.\end{multline*}
Recall that we have an isomorphism
$$H^{\bullet}(\Hoch^{\bullet}(\cB,M))\cong HH^{\bullet}(\cB,M)\cong \Ext^n_{\cB\otimes\cB^{op}}(I_{\cB},M).$$

In some cases Hochschild cochain complex carries a natural structure of a DG algebra.

\begin{defi}\label{defi:Hochschild_DGA} Let $\cB$ be a small h-projective DG category, and $F:\cB\to\cC$ a DG functor to an arbitrary small DG category $\cC.$ Let us put $\Hoch^{\bullet}(\cB,\cC):=\Hoch^{\bullet}(\cB,(F\otimes F^{op})_*I_{\cC}).$ The complex $\Hoch^{\bullet}(\cB,\cC)$ carries a structure of a DG algebra, where the product is given by the formula
\begin{equation}\label{eq:formula_for_product_in_HH}(f\cdot g)_n(a_n,\dots,a_1)=\sum\limits_{i=0}^n (-1)^{|g_i|\cdot (|a_{i+1}|+\dots+|a_n|)} f_{n-i}(a_n,\dots,a_{i+1})g_i(a_i,\dots,a_1).\end{equation}
Here in the RHS we take the composition in $\cC.$\end{defi}

We will need the following technical result.

\begin{lemma}\label{lem:q_is_on_HH_for_gluing} Let $\cC_1$ and $\cC_2$ be small DG categories, and $\Phi:\cC_1\to\cC_2$ a DG functor. Denote by $N=N_{\Phi}$ the DG $(\cC_1^{op}\otimes\cC_2)$-module, given by the formula
\begin{equation}\label{eq:bimodule_N_Phi} N_{\Phi}(X,Y)=\cC_2(Y,\Phi(X)),\quad X\in\cC_1,\,Y\in\cC_2.\end{equation}
Take the gluing $\widetilde{\cC}=\cC_2\sqcup_{N_{\Phi}}\cC_1.$

Take some $(\cC_2\otimes\cC_2^{op})$-module $M.$ Define the $(\widetilde{\cC}\otimes\widetilde{\cC}^{op})$-module $\widetilde{M}$ by the formula
\begin{equation}\label{eq:bimodule_M_tilde} \widetilde{M}(X,Y)=\begin{cases}M(\Phi(X),\Phi(Y)) & \text{for }X,Y\in Ob(\cC_1);\\
M(X,Y) & \text{for }X,Y\in Ob(\cC_2);\\
M(X,\Phi(Y)) & \text{for }X\in Ob(\cC_2),\,Y\in Ob(\cC_1);\\
0 & \text{for }X\in Ob(\cC_1),Y\in Ob(\cC_2).\end{cases}\end{equation}

1) The natural map $HH^{\bullet}(\widetilde{\cC},\widetilde{M})\to HH^{\bullet}(\cC_2,M)$ is an isomorphism.

2) Assume that the DG functor $\Phi$ is a homological epimorphism. Then the natural map $HH^{\bullet}(\widetilde{\cC},\widetilde{M})\to HH^{\bullet}(\cC_1,(\Phi\otimes\Phi^{op})_*M)$ is an isomorphism.
\end{lemma}

\begin{proof}As in Section \ref{sec:preliminaries_on_DG}, we denote by $\iota_{\cC_1}:\cC_1\to\widetilde{\cC},$ $\iota_{\cC_2}:\cC_2\to\widetilde{\cC}$ the natural inclusion DG functors. By the exact triangle \eqref{eq:diagonal_of_gluing}, we have an exact triangle
\begin{multline}\label{eq:exact_triangle_for_HH}\bR\Hom_{\widetilde{\cC}\otimes\widetilde{\cC}^{op}}(I_{\widetilde{\cC}},\widetilde{M})\to\\ \bR\Hom_{\widetilde{\cC}\otimes\widetilde{\cC}^{op}}(\bL(\iota_{\cC_2}\otimes\iota_{\cC_2}^{op})^*(I_{\cC_2}),\widetilde{M})\oplus \bR\Hom_{\widetilde{\cC}\otimes\widetilde{\cC}^{op}}(\bL(\iota_{\cC_1}\otimes\iota_{\cC_1}^{op})^*(I_{\cC_1}),\widetilde{M})\\ \to \bR\Hom_{\widetilde{\cC}\otimes\widetilde{\cC}^{op}}(\bL(\iota_{\cC_2}\otimes\iota_{\cC_1}^{op})^*(N_{\Phi}),\widetilde{M}).\end{multline}
Let us note that $N_{\Phi}\cong\bL(\Phi\otimes\id_{\cC_1^{op}})^*(I_{\cC_1}).$ 
Applying adjunctions, we see that the triangle \eqref{eq:exact_triangle_for_HH} is isomorphic to the exact triangle
\begin{multline}
\label{eq:exact_triangle_for_HH_mod} \bR\Hom_{\widetilde{\cC}\otimes\widetilde{\cC}^{op}}(I_{\widetilde{\cC}},\widetilde{M})\to\\ \bR\Hom_{\cC_2\otimes\cC_2^{op}}(I_{\cC_2},M)\oplus \bR\Hom_{\cC_1\otimes\cC_1^{op}}(I_{\cC_1},(\Phi\otimes\Phi^{op})_*M) \stackrel{(f,g)}{\to}\\ 
\bR\Hom_{\cC_1\otimes\cC_1^{op}}(I_{\cC_1},(\Phi\otimes\Phi^{op})_*M).
\end{multline}

To prove 1), it suffices to note that the map $g$ from \eqref{eq:exact_triangle_for_HH_mod} is the identity.

To prove 2), let us note note that the map $f$ from \eqref{eq:exact_triangle_for_HH_mod} is given by the composition
$$\bR\Hom_{\cC_2\otimes\cC_2^{op}}(I_{\cC_2},M)\to \bR\Hom_{\cC_2\otimes\cC_2^{op}}(\bL(\Phi\otimes\Phi^{op})^*I_{\cC_1},M)\cong \bR\Hom_{\cC_1\otimes\cC_1^{op}}(I_{\cC_1},(\Phi\otimes\Phi^{op})_*M).$$
This composition is in turn an isomorphism provided that $\Phi$ is a homological epimorphism.
This proves 2) and the lemma.
 \end{proof}
 
We will need the following observation.
 
 \begin{prop}\label{prop:Hoch_complexes_bounded} Suppose that the DG category $\cB$ is non-positive, and the DG $\cB\otimes\cB^{op}$-module $M$ is concentrated in degree zero.

1) The Hochschild chain complex $\Hoch_{\bullet}(\cB,M)$ is bounded above.

2) The Hochschild cochain complex $\Hoch^{\bullet}(\cB,M)$ is bounded below.\end{prop}

\begin{proof}This follows immediately from the definitions.\end{proof}

\subsection{Hochschild (co)homology of the second kind}
\label{ssec:HH_second_kind}

Now let $w$ be a closed element of degree zero of the center of $\cB.$ That is, for each object $X$ of $\cB$ we have
an element $w_X\in\cB^0(X,X),$ such that $d(w_X)=0$ and for any morphism $f\in\cB(X,Y)$
we have $w_Yf=fw_X.$ We denote by $Z^0_{cl}(\cB)$ the $A$-module of all such $w.$

Let us assume that moreover $w$ is central for the bimodule $M,$ i.e. for any $m\in M(X,Y)$ we have $mw_X=w_Ym.$
Then we have an additional differential $\delta=\delta_w$ of (cohomological) degree $-1$ on $\Hoch_{\bullet}(\cB,M).$
It is given by the formula
\begin{multline*}\delta(m\otimes a_{n}\otimes\dots\otimes a_1)=m\otimes w_{X_{n}}\otimes a_{n}\otimes\dots a_1-m\otimes a_{n}\otimes w_{X_{n-1}}\otimes
a_{n-1}\otimes\dots\otimes a_1\\\dots+(-1)^{n} m\otimes a_{n}\otimes\dots\otimes a_1\otimes w_{X_0},\end{multline*}
where $$a_i\in\cB(X_{i-1},X_i),\quad 1\leq i\leq n,\quad m\in M(X_n,X_0).$$
It is easy to check that $\delta$ anti-commutes with the Hochschild differential $b,$ hence we have a mixed complex
$(\Hoch_{\bullet}(\cB,M),\delta_w).$

Similarly, we have a differential $\delta=\delta_w$ (which we denote by the same symbol) on the Hochschild cochain complex
$\Hoch^{\bullet}(\cB,M).$ It is given by the formula
\begin{multline*}(\delta f)(a_{n-1},\dots,a_1)=-f(w_{X_{n-1}},a_{n-1},\dots,a_1)+f(a_{n-1},w_{X_{n-2}},a_{n-2},\dots,a_1)\\
\dots +(-1)^n f(a_{n-1},\dots,a_1,w_{X_0}),\end{multline*}
where $$a_i\in\cB(X_{i-1},X_i),\quad 1\leq i\leq n-1.$$
Again, it is easy to check that we have a mixed complex $(\Hoch^{\bullet}(\cB,M),\delta_w).$

\begin{defi}Let $\cB,$ $w$ and $M$ be as above.

1) The ($\Z/2$-graded) Hochschild chain complex of the second kind is defined by the formula
$$\Hoch^{II}_{\bullet}(\cB,w;M):=\Tot^{L}(\Hoch_{\bullet}(\cB,M),\delta_w).$$
We denote its homology by $HH^{II}_{\bullet}(\cB,w;M)$ and call it Hochschild homology of the second kind of $\cB$ with curvature $w$ with coefficients in $M.$

2) The ($\Z/2$-graded) Hochschild cochain complex of the second kind is defined by the formula
$$\Hoch^{II,\bullet}(\cB,w;M):=\Tot^{L}(\Hoch^{\bullet}(\cB,M),\delta_w).$$
We denote its cohomology by $HH^{II,\bullet}(\cB,w;M)$ and call it Hochschild cohomology of the second kind of $\cB$ with curvature $w$ with coefficients in $M.$\end{defi}

\begin{cor}\label{cor:sp_seq_for_Hochschild_second_kind} 1) A triple $(\cB,w,M)$ as above provides natural spectral sequences
\begin{equation}\label{eq:sp_seq_HH^*}\{(E_r^{\bullet},d_r)\}_{r\geq 1},\quad E_1^n=HH^n(\cB,M),\,d_1=\delta_w,\end{equation}
\begin{equation}\label{eq:sp_seq_HH_*}\{(E_r'^{\bullet},d_r)\}_{r\geq 1},\quad E_1'^{-n}=HH_n(\cC,M),\,d_1=\delta_w.\end{equation}

2) If the DG category $\cB$ is non-positive and the bimodule $M$ is concentrated in degree zero, then the spectral sequence \eqref{eq:sp_seq_HH^*} converges to $HH^{II,\bar{n}}(\cB,w;M).$
\end{cor}

\begin{proof}1) follows immediately from Proposition \ref{prop:functoriality_of_sp_seq} applied to mixed Hochschild chain and cochain complexes.

2) follows from Propositions \ref{prop:convergence_of_sp_seq} and \ref{prop:Hoch_complexes_bounded}.\end{proof}

\begin{prop}\label{prop:q_is_homological_epi}Let $\cC_1,\cC_2$ be small DG categories, and $\Phi:\cC_1\to \cC_2$ a DG functor. Let $w_i\in Z^0_{cl}(\cC_i),$  and assume that $(w_2)_{\Phi(X)}=\Phi((w_1)_X)$ for any object $X\in\cC_1.$ Let $M$ be a
$\cC_2\otimes\cC_2^{op}$-module, such that $w_2$ is central for $M$ (hence $w_1$ is central for $(\Phi\otimes\Phi^{op})_*M$).
Finally, assume that $\Phi$ is a homological epimorphism. 

1) The mixed complexes $(\Hoch_{\bullet}(\cC_1,(\Phi\otimes\Phi^{op})_*M),\delta_{w_1})$ and
$(\Hoch_{\bullet}(\cC_2,M),\delta_{w_2})$
are naturally quasi-isomorphic. In particular, the corresponding spectral sequences \eqref{eq:sp_seq_HH_*} are naturally isomorphic.

2) The mixed complexes $(\Hoch^{\bullet}(\cC_1,(\Phi\otimes\Phi^{op})_*M),\delta_{w_1})$ and
$(\Hoch^{\bullet}(\cC_2,M),\delta_{w_2})$
are naturally quasi-isomorphic. In particular, the corresponding spectral sequences \eqref{eq:sp_seq_HH^*} are naturally isomorphic.\end{prop}

\begin{proof} 1) The DG functor $\Phi$ induces a natural morphism of mixed complexes
$$\phi:(\Hoch_{\bullet}(\cC_1,(\Phi\otimes\Phi^{op})_*M),\delta_{w_1})\to (\Hoch_{\bullet}(\cC_2,M),\delta_{w_2}).$$
It is given by the formula
$$\phi(m\otimes a_n\otimes\dots\otimes a_1)=m\otimes\Phi(a_n)\otimes\dots\otimes\Phi(a_1).$$
It is easy to check that $\phi$ indeed commutes with both differentials. To show that $\phi$ is a quasi-isomorphism of Hochschild complexes, it suffices to note that in $D(A)$ it corresponds to the composition of isomorphisms
$$I_{\cC_1}\stackrel{\bL}{\otimes}_{\cC_1\otimes\cC_1^{op}}(\Phi\otimes\Phi^{op})_*M\cong\bL(\Phi^{op}\otimes\Phi)^*I_{\cC_1}
\stackrel{\bL}{\otimes}_{\cC_2\otimes\cC_2^{op}}M\cong I_{\cC_2}\stackrel{\bL}{\otimes}_{\cC_2\otimes\cC_2^{op}}M.$$
Here the last isomorphism is implied by the assumption on $\Phi$ to be a homological epimorphism. This proves 1).

2) We will apply Lemma \ref{lem:q_is_on_HH_for_gluing}. Denote by $N_{\Phi}\in D(\cC_1^{op}\otimes\cC_2)$ the bimodule given by \eqref{eq:bimodule_N_Phi}.
Consider the glued DG category $\widetilde{\cC}=\cC_2\sqcup_{N_{\Phi}}\cC_1.$ We have a $\widetilde{\cC}\otimes\widetilde{\cC}^{op}$-module $\widetilde{M},$ given by \eqref{eq:bimodule_M_tilde}.

We have a natural degree zero element $\widetilde{h}$ of the center of $\widetilde{\cC},$ given by the formula
$$\widetilde{h}_X=(w_1)_X\text{ for }X\in\cC_1,\quad \widetilde{h}_Y=(w_2)_Y\text{ for }Y\in\cC_2.$$
By our assumption on $w_1$ and $w_2,$ we have that $\widetilde{h}$ is central for $\widetilde{M}.$

We have natural projection morphisms of mixed complexes:
$$(\Hoch^{\bullet}(\cC_1,(\Phi\otimes\Phi^{op})_*M),\delta_{w_1})\stackrel{\phi_1}{\leftarrow}
(\Hoch^{\bullet}(\widetilde{\cC},\widetilde{M}),\delta_{\widetilde{h}})\stackrel{\phi_2}{\to}(\Hoch^{\bullet}(\cC_2,M),\delta_{w_2}).$$
By Lemma \ref{lem:q_is_on_HH_for_gluing}, both $\phi_1$ and $\phi_2$ are quasi-isomorphisms. This proves 2).
\end{proof}

\begin{lemma}\label{lem:q_is_for_delta_w_1...w_k} Let $\cC$ be a small DG category, and $M$ a DG $\cC\otimes\cC^{op}$-module. Let $w_1,\dots,w_k$ be a collection of closed degree zero elements of the center of $\cC$ which are central for $M.$ Then we have $A_{\infty}$-quasi-isomorphisms of mixed complexes:
$$(\Hoch^{\bullet}(\cC,M),\delta_{(w_1\dots w_k)})\simeq (\Hoch^{\bullet}(\cC,M),\sum\limits_{i=1}^k(\prod\limits_{\substack{1\leq j\leq k;\\j\ne i}}w_j)\delta_{w_i}),$$
$$(\Hoch_{\bullet}(\cC,M),\delta_{(w_1\dots w_k)})\simeq (\Hoch_{\bullet}(\cC,M),\sum\limits_{i=1}^k(\prod\limits_{\substack{1\leq j\leq k;\\j\ne i}}w_j)\delta_{w_i}).$$\end{lemma}

\begin{proof}We will prove the statement for cochain Hochschild complex, and the proof for chain Hochschild complex is analogous.

Let us notice that the elements $w_1,\dots,w_k$ of the center of $\cC$ provide a structure of an $A[x_1,\dots,x_k]$-linear DG category on $\cC.$ complex $\Hoch^{\bullet}(\cC,M)$ is a module over the DG algebra $\Hoch_{\bullet}(A[x_1,\dots,x_k]).$ The latter DG algebra is strongly super-commutative, and has a divided power structure. The differential $\delta_{(w_1\dots w_k)}$ on $\Hoch^{\bullet}(\cC,M)$ is just a multiplication by the chain $$1\otimes (x_1\dots x_k)\in\Hoch_1(A[x_1,\dots,x_k]).$$
Further, the differential $\sum\limits_{i=1}^k(\prod\limits_{\substack{1\leq j\leq k;\\j\ne i}}w_i)\delta_{w_i})$ on $\Hoch^{\bullet}(\cC,M)$ is a multiplication by the chain
$$\sum\limits_{i=1}^k (\prod\limits_{\substack{1\leq j\leq k;\\j\ne i}}x_j)\otimes x_i \in\Hoch_1(A[x_1,\dots,x_k]).$$ But those two chains are homologous, hence the desired $A_{\infty}$-quasi-isomorphism of mixed complexes follows from Lemma \ref{lem:q_is_divided_powers}.\end{proof}

We need one more observation about duality of chain and cochain Hochschild complexes. For an object of $X\in D(A),$ we put $\bR(-)^*(X)=\bR\Hom_A(X,A).$ We use the same notation for the objects of $D(\Mix_A).$

\begin{lemma}\label{lem:duality_of_Hoch_complexes} Let $\cB$ be an h-projective DG category, and $M$ a DG $\cB\otimes\cB^{op}$-module, which is h-projective over $k.$ Take an element $w\in Z^0_{cl}(\cB),$ which is central for $M.$
Then we have natural isomorphisms in $D(\Mix_A):$
$$\bR(-)^*((\Hoch_{\bullet}(\cB,M),\delta_w))\cong (\Hoch_{\bullet}(\cB,M),\delta_w)^*\cong (\Hoch^{\bullet}(\cB,M^*),\delta_w).$$
\end{lemma}

\begin{proof}1) By our assumptions on $\cB$ and $M,$ the complex of $A$-modules $\Hoch_{\bullet}(\cB,M)$ is h-projective. Thus, its derived dual is isomorphic to the actual dual. The second isomorphism follows immediately from the definition of these mixed complexes.\end{proof}

We will need the following result under stronger assumptions.

\begin{lemma}\label{lem:duality_for_homology}Let $\cB,$ $w$ and $M$ be as in Lemma \ref{lem:duality_of_Hoch_complexes}. Assume that the following holds:

i) the basic ring $A$ is noetherian and regular;

ii) the DG category $\cB$ is non-positive, and $M$ is bounded above;

iii) the $A$-modules $HH_n(\cB,M)$ are finitely generated for all $n.$

Then

1) the natural morphism 
\begin{equation}\label{eq:duality_of_Hoch_mixed_complexes}(\Hoch_{\bullet}(\cB,M),\delta_w)\to \bR(-)^*(\Hoch^{\bullet}(\cB,M^*),\delta_w)\end{equation}
in $D(\Mix_A)$ is an isomorphism;

2) there is a natural isomorphism
\begin{equation}\label{eq:duality_of_Hoch_of_second_kind}\Hoch^{II}_{\bullet}(\cB,w;M)\cong \bR(-)^*(\Hoch^{II,\bullet}(\cB,w;M^*))\end{equation}
in $D(A[u^{\pm 1}]).$
\end{lemma}

\begin{proof}1) Let us put $\cK^{\bullet}:=\Hoch_{\bullet}(\cB,M).$ By Lemma \ref{lem:duality_of_Hoch_complexes}, we have $\Hoch^{\bullet}(\cB,M^*)=(\cK^{\bullet})^*=\bR(-)^*(\cK^{\bullet}).$ Thus, we need to show that the morphism $\cK^{\bullet}\to\bR(-)^*((\cK^{\bullet})^*)$ is an isomorphism in $D(A)$ (so we can forget about the structure of mixed complex). 

By our assumptions i)-iii), the complex $\cK^{\bullet}$ is quasi-isomorphic to a bounded above complex $P^{\bullet}$ of finitely generated projective $A$-modules. It follows that $(P^{\bullet})^*$ is a bounded below complex of finitely generated projectives. Since $A$ is regular, it follows from Proposition \ref{prop:h-proj_unbounded} that the complex $(P^{\bullet})^*$ is h-projective. Hence, we have $\bR(-)^*((P^{\bullet})^*)=(P^{\bullet})^{**}\cong P^{\bullet}.$ This proves 1).

2) It follows from i)-iii) that the mixed complex $(\Hoch_{\bullet}(\cB,M),\delta_w)$ is isomorphic in $D(\Mix_A)$ to some $(\cP^{\bullet},\delta),$ where the $P^{\bullet}$ is a bounded above complex of finitely generated projective $A$-modules. By Lemma \ref{lem:duality_of_Hoch_complexes}, we have an isomorphism $$(\Hoch^{\bullet}(\cB,M^*),\delta_w)\cong ((P^{\bullet})^*,\delta^*)$$
in $D(\Mix_A).$ The latter mixed complex is bounded below. By Proposition \ref{prop:functoriality_of_Tot}, we have the following isomorphisms in $D(A[u^{\pm 1}]):$
$$\Hoch^{II}_{\bullet}(\cB,w;M)\cong \Tot^L(\cP^{\bullet},\delta)=\Tot^{\Pi}(\cP^{\bullet},\delta);$$
$$\Hoch^{II,\bullet}(\cB,w;M^*)\cong \Tot^L((\cP^{\bullet})^*,\delta^*)=\Tot^{\oplus}((\cP^{\bullet})^*,\delta^*).$$
Since $\cP^n$ are finitely generated projective and the complex $\Tot^{\oplus}((\cP^{\bullet})^*,\delta^*)$ is h-projective (by Proposition \ref{prop:h-proj_unbounded}), we have
$$\Tot^{\Pi}(\cP,\delta)\cong (\Tot^{\oplus}((\cP^{\bullet})^*,\delta^*))^*\cong \bR(-)^*\Tot^{\oplus}((\cP^{\bullet})^*,\delta^*).$$
This proves 2).\end{proof}

\begin{prop}\label{prop:product_on_HH_second_kind} Let $\cB$ be a small h-projective DG category, and $F:\cB\to\cC$ a DG functor to a small DG category. Let $w\in Z^0_{cl}(\cB)$ be an element which is central for $(F\otimes F^{op})_*I_{\cC}.$ 

1) The operator $\delta_w$ is a derivation on the DG algebra $\Hoch^{\bullet}(\cB,\cC).$

2) The formula \eqref{eq:formula_for_product_in_HH} defines a ($\Z/2$-graded) DGA structure on $\Hoch^{II,\bullet}(\cB,w;\cC).$ The filtration $F_{\bullet}HH^{II}(\cB,w;\cC)$ satisfies the condition $F_n\cdot F_m\subset F_{n+m}.$

3) Take the spectral sequence $\{(E_r^{\bullet},d_r)\}_{r\geq 1}$ associated with the mixed complex $(\Hoch^{\bullet}(\cB,\cC),\delta_w).$ Then each $E_r^{\bullet}$ carries a graded algebra structure and $d_r$ is a derivation on $E_r.$ The product on $E_1^{\bullet}=HH^{\bullet}(\cB,\cC)$ is induced by that on $\Hoch^{\bullet}(\cB,\cC),$ and the product on $E_{r+1}^{\bullet}$ is induced by that on $E_r^{\bullet}.$

4) Assume that $\cB$ is non-positive and $M$ is bounded below. Then we have an isomorphism of graded algebras $$E_{\infty}^{\bullet}\cong\gr_{\bullet}^FHH^{II,\bullet}(\cB,w;\cC).$$
\end{prop}

\begin{proof}1) is checked straightforwardly.

2),3) and 4) follow from Proposition \ref{prop:product_on_Tot_of_DGAs}.\end{proof}

\subsection{Mac Lane (co)homology}
\label{ssec:HML}

We refer to \cite{L}, Chapter 13, for a nice introduction to Mac Lane (co)homology. The original definition was given in \cite{ML}.

Let $R$ be an associative ring. Recall that we denote by $P(R)$ the category of finitely generated projective right $R$-modules. Recall that the ring $R$ is Morita equivalent to the additive category $P(R),$ and the ring $R\otimes R^{op}$ is Morita equivalent to the pre-additive category $P(R)\otimes P(R)^{op}.$ In particular, any $R\text{-}R$-bimodule can be considered as a $P(R)\text{-}P(R)$-bimodule.

\begin{defi}\label{defi:HML}Let $M$ be an $R\otimes R^{op}$-module. Mac Lane homology of $R$ with coefficients in $M$ is defined by the formula
$$HML_{\bullet}(R,M)=HH_{\bullet}(\Z[P(R)],M).$$ Similarly, Mac Lane cohomology of $R$ with coefficients in $M$ is defined by the formula
$$HML^{\bullet}(R,M)=HH^{\bullet}(\Z[P(R)],M).$$\end{defi}

Recall that the center of $R$ is naturally isomorphic to the center of $P(R).$ Below we tacitly identify these centers. In particular, for any $w\in Z(R)$ we have the element $[w]\in Z(\Z[P(R)]).$

\begin{defi}\label{defi:HML^II}Let $M$ be a complex of $R\otimes R^{op}$-modules, and $w\in Z(R).$ Assume that $w$ is central for $M.$ Mac Lane homology of the second kind of $(R,w)$ with coefficients in $M$ is defined by the formula
$$HML_{\bullet}^{II}(R,w;M)=HH_{\bullet}^{II}(\Z[P(R)],[w];M).$$
Similarly, Mac Lane cohomology of the second kind of $(R,w)$ with coefficients in $M$ is defined by the formula
$$HML^{II,\bullet}(R,w;M)=HH^{II,\bullet}(\Z[P(R)],[w];M).$$
\end{defi}

By Corollary \ref{cor:sp_seq_for_Hochschild_second_kind}, if $M$ is concentrated in degree zero, then we have a $\Z$-graded spectral sequence converging to $HML^{II,\bullet}(R,w;M):$
$$E_1^n=HML^n(R,M)\Rightarrow HML^{II,\bar{n}}(R,w;M).$$

\begin{prop}\label{prop:HML^II_via_Q^n}Let $R,$ $w$ and $M$ be as in Definition \ref{defi:HML^II} For any positive integer $n,$ we have natural isomorphisms in $D(\Mix_A):$
$$(\Hoch^{\bullet}(\Z[P(R)],M),\delta_{[w]})\cong (\Hoch^{\bullet}(Q^n_{\bullet}(\Mat_n(R)),M),\delta_{[w]}),$$
$$(\Hoch_{\bullet}(\Z[P(R)],M),\delta_{[w]})\cong (\Hoch_{\bullet}(Q^n_{\bullet}(\Mat_n(R)),M),\delta_{[w]}).$$
In particular, we have
$$HML^{II,\bullet}(R,w;M)\cong HH^{II,\bullet}(Q^n_{\bullet}(\Mat_n(R)),\Mat_n(M)),$$
$$HML^{II}_{\bullet}(R,w;M)\cong HH^{II}_{\bullet}(Q^n_{\bullet}(\Mat_n(R)),\Mat_n(M)).$$\end{prop}

\begin{proof}We will prove the first isomorphism, and the second one is proved analogously.

By Corollary \ref{cor:F_n_homological_epi}, the DG functor $F_n:\Z[P(R)]\to Q^n_{\bullet}(P(R))$ is a homological epimorphism. By 
Proposition \ref{prop:q_is_homological_epi} we obtain the isomorphism
$$(\Hoch^{\bullet}(\Z[P(R)],M),\delta_{[w]})\cong (\Hoch^{\bullet}(Q^n_{\bullet}(P(R)),M),\delta_{[w]})$$
in $D(\Mix_A).$

Finally, by the proof of Corollary \ref{cor:R^n_gen_of_Q^n(P(R))} we have an isomorphism
$$(\Hoch^{\bullet}(Q^n_{\bullet}(P(R)),M),\delta_{[w]})\cong (\Hoch^{\bullet}(Q^n_{\bullet}(\Mat_n(R)),M),\delta_{[w]})$$
in $D(\Mix_A).$ This proves the proposition.\end{proof}

\begin{remark}For $n=1,$ it was proved in \cite{JP} that the complexes $\Hoch^{\bullet}(\Z[P(R)],M)$ and $\Hoch^{\bullet}(Q_{\bullet}(R),M)$ are quasi-isomorphic.\end{remark}

\begin{theo}\label{th:HML^2_sq_zero_ext}(\cite{ML})  Let $A$ be a ring, and $M$ an $A\text{-}A$-bimodule. There is a bijection between the set $HML^2(R,M)$ and the set of square-zero extensions
\begin{equation}\label{eq:square_zero_extension}0\to M\to \widetilde{A}\to A\to 0.\end{equation}
Moreover, if we choose a set-theoretic splitting $s:A\to \widetilde{A}$ of \eqref{eq:square_zero_extension}, satisfying $s(0)=0,$ then we have an explicit cocycle $\alpha\in\Hoch^2(Q_{\bullet}(A),M)$ with components
$$\alpha_1:Q_1(A)\to M,\quad\alpha_1([(a,b)])=s(a)+s(b)-s(a+b),$$
$$\alpha_2:Q_0(A)\otimes Q_0(A)\to M,\quad\alpha_2([a],[b])=s(a)s(b)-s(ab).$$
The class $\bar{\alpha}$ corresponds to the square-zero extension \eqref{eq:square_zero_extension}.\end{theo}

For any associative ring $A$ (and more generally for any associative ring spectra) we have topological Hochschild homology $THH_{\bullet}(A),$ which was originally defined by B\"okstedt \cite{Bok}. The following theorem was proved by Pirashvili and Waldhausen \cite{PW}.

\begin{theo}\label{th:THH_equals_HML} (\cite{PW}) For any associative ring $A$ there is a natural isomorphism
$$THH_{\bullet}(A)\cong HML_{\bullet}(A).$$\end{theo}

Note that if the ring $R$ is commutative and $M$ is an $R$-module, then the complex $\Hoch_{\bullet}(Q_{\bullet}(R),M)$ is a complex of $R$-modules, and similarly for the Hochschild cochain complex. Moreover, for any element $w\in R$ (which is automatically central for $M$) the differential $\delta_w$ is $R$-linear. We need the following straightforward applications of Lemmas \ref{lem:duality_of_Hoch_complexes} and \ref{lem:duality_for_homology}.

\begin{lemma}\label{lem:duality_of_Mac_Lane_complexes}Let $R$ be a commutative ring and $M$ an h-projective complex of $R$-modules. Let $w\in R$ be any element.  Then we have natural isomorphism in $D(\Mix_R):$
$$\bR\Hom_R((\Hoch_{\bullet}(Q_{\bullet}(R),M),\delta_{[w]}),R)\cong (\Hoch^{\bullet}(Q_{\bullet}(R),\Hom_R(M,R)),\delta_w)$$\end{lemma}

\begin{proof}Indeed, note that \begin{equation}\label{eq:identification_for_Mac_Lane_chain_complex}\Hoch_{\bullet}(Q_{\bullet}(R),M)\cong\Hoch_{\bullet}^R(R\otimes Q_{\bullet}(R),M),\end{equation}
and similarly
\begin{equation}
\label{eq:identification_for_Mac_Lane_cochain_complex}
\Hoch^{\bullet}(Q_{\bullet}(R),M)\cong\Hoch^{\bullet}_R(R\otimes Q_{\bullet}(R),M)
\end{equation}
for the Hochschild cochain complex. The assertion now follows from Lemma \ref{lem:duality_of_Hoch_complexes}.\end{proof}

\begin{lemma}\label{lem:duality_for_Mac_Lane_homology}Let $R,$ $w$ and $M$ be as in Lemma \ref{lem:duality_of_Mac_Lane_complexes}. Assume that the following holds:

i) the ring $R$ is noetherian and regular;

ii) $M$ is bounded above;

iii) the $R$-module $HML_n(R,M)$ is finitely generated for all $n.$

Then

1) the natural morphism $$(\Hoch_{\bullet}(Q_{\bullet}(R),M),\delta_w)\to\bR\Hom_R((\Hoch^{\bullet}(Q_{\bullet}(R),\Hom_R(M,R)),\delta_w),R)$$
in $D(\Mix_R)$ is an isomorphism;

2) there is a natural isomorphism $$\Hoch^{II}_{\bullet}(Q_{\bullet}(R),[w];M)\cong \bR\Hom_R(\Hoch^{II,\bullet}(Q_{\bullet}(R),w;\Hom_R(M,R)),R)$$
in $D(R[u^{\pm 1}]).$
\end{lemma}

\begin{proof}Similarly to the proof of the previous Lemma, this follows immediately from the identifications \eqref{eq:identification_for_Mac_Lane_chain_complex}, \eqref{eq:identification_for_Mac_Lane_cochain_complex} and Lemma \ref{lem:duality_for_homology}.\end{proof}

\section{Change of rings spectral sequence}
\label{sec:change_of_rings_sp_seq}

Let $f:R_1\to R_2$ be a homomorphism of rings, and $M$ be a $R_2\text{-}R_2$-bimodule. Then $M$ can be also considered as an $R_1\text{-}R_1$-bimodule. 

\begin{prop}\label{prop:sp_seq_change_of_rings} There is a converging spectral sequence
\begin{equation}\label{eq:sp_seq_change_of_rings}E_2^{p,q}=HML^p(R_2,\Ext^q_{R_1}(R_2,M))\Rightarrow HML^{p+q}(R_1,M).\end{equation}
Here $\Ext^{\bullet}_{R_1}(-,-)$ denotes the Ext groups in the category of {\it right} $R_1$-modules.\end{prop}

\begin{proof}See \cite{L}, Section 13.4.22.\end{proof}

\begin{prop}\label{prop:HML_localizations} Let $R$ be a commutative ring, and $S\subset R$ a multiplicative system. Let $M$ be an $S^{-1}R$-module (hence also a bimodule). Then we have a natural isomorphism
$$HML^{\bullet}(R,M)\cong HML^{\bullet}(S^{-1}R,M).$$\end{prop}

\begin{proof}By \eqref{eq:sp_seq_change_of_rings}, we have a converging spectral sequence
$$E_2^{p,q}=HML^p(S^{-1}R,\Ext^q_R(S^{-1}R,M))\Rightarrow HML^{p+q}(R,M).$$
But we have $$\Ext^q_R(S^{-1}R,M)=\begin{cases}M & \text{for }q=0;\\
0 & \text{otherwise.}\end{cases}$$
This proves the proposition.\end{proof}

For an associative ring $R$ and a (two-sided) ideal $I\subset R$ we put $$\hat{R}_I:=\lim\limits_{\leftarrow}R/I^n.$$

\begin{prop}\label{prop:HML_completions} Let $R$ be a noetherian commutative ring, and $I\subset R$ an ideal. Let $M$ be a finitely generated $I$-torsion $\hat{R}_I$-module. Then we have a natural isomorphism
\begin{equation}\label{eq:HML_completions}HML^{\bullet}(R,M)\cong HML^{\bullet}(\hat{R}_I,M).\end{equation}\end{prop}

\begin{proof}We may and will assume that $M$ is actually an $R/I$-module. By \eqref{eq:sp_seq_change_of_rings}, we have converging spectral sequence
$$E_2^{p,q}=HML^p(R/I,\Ext^q_R(R/I,M))\Rightarrow HML^{p+q}(R,M),$$
$$E_2^{p,q}=HML^p(R/I,\Ext^q_{\hat{R}_I}(R/I,M))\Rightarrow HML^{p+q}(\hat{R}_I,M).$$ But we have a natural isomorphism
$$\Ext^{\bullet}_R(R/I,M)\cong \Ext^{\bullet}_{\hat{R}_I}(R/I,M).$$
Thus, we have an isomorphism on the $E_2$-terms of the spectral sequences. It follows that we have an isomorphism on their limits. This proves the proposition.\end{proof}

\section{Mac Lane (co)homology of finite fields}
\label{sec:HML_finite_fields}

Let $p$ be a prime number. Let us define the graded commutative algebra
$$\Lambda=\Z/p\Z[e_0,e_1,\dots]/(e_i^p;i\geq 0).$$
Here the generator $e_i$ is homogeneous of degree $2p^i.$
For any $k\geq 0,$ we put $$\Lambda_k:=\Lambda/(e_0,\dots,e_{k-1}).$$

\begin{theo}\label{th:HML_F_q_F_q}([\cite{FLS}]) For any finite field $\bbF_q$ of characteristic $p,$ we have an isomorphism of graded $\bbF_q$-algebras
$$HML^{\bullet}(\bbF_q)\cong\bbF_q\otimes\Lambda.$$

Dually ,we have
$$HML_n(\bbF_q)=\begin{cases}\bbF_q & \text{ for }n\text{ even;}\\
                 0 & \text{ for }n\text{ odd.}
                \end{cases}$$\end{theo}

We would like to write down an explicit cocycle for $e_0\in HML^2(\bbF_q).$                
                
Note that we have a non-trivial square-zero extension
$$0\to \bbF_q\to W_2(\bbF_q)\to\bbF_q\to 0.$$
Moreover, we have a multiplicative section
\begin{equation}\label{eq:s(a)=a^q}s:\bbF_q\to W_2(\bbF_q),\quad s(a)=\widetilde{a}^q,\end{equation}
where $\widetilde{\alpha}\in W_2(\F_q)$ is a lift of $\alpha.$
By Theorem \ref{th:HML^2_sq_zero_ext}, this extension corresponds to a non-zero class in $HML^2(\bbF_q,\bbF_q),$ which is actually the generator $e_0.$ Moreover, the section \eqref{eq:s(a)=a^q} defines a cocycle \begin{equation}\label{eq:cocycle_alpha_q}\alpha_q\in\Hoch^2(Q_{\bullet}(\bbF_q),\bbF_q),\quad\alpha_q([(a,b)])=\frac{a^q+b^q-(a+b)^q}{p},\quad [(a,b)]\in Q_1(\bbF_q),\end{equation}
where the fraction in the RHS is considered as a polynomial with integer coefficients. The restriction of $\alpha_q$ to $Q_0(\bbF_q)\otimes Q_0(\bbF_q)$ is zero (since \eqref{eq:s(a)=a^q} is multiplicative).

\begin{lemma}\label{lem:reduced_for_HML_of_F_q}Denote by $\Hoch^{red,\bullet}(Q_{\bullet}(\bbF_q),\bbF_q)\subset \Hoch^{\bullet}(Q_{\bullet}(\bbF_q),\bbF_q)$ the graded subspace of chains $\varphi:(Q_{\bullet}(\bbF_q))^{\otimes n}\to \bbF_q$ which are $Q_0(\bbF_q)$-linear, and such that $\varphi$ vanishes on $a_1\otimes\dots\otimes a_n$ whenever at least one of $a_i$ is in $Q_0(\bbF_q).$ Then $\Hoch^{red,\bullet}(Q_{\bullet}(\bbF_q),\bbF_q)\subset \Hoch^{\bullet}(Q_{\bullet}(\bbF_q),\bbF_q)$ is a DG subring, and the inclusion morphism is a quasi-isomorphism.\end{lemma}

\begin{proof}Let us choose some generator $x_0$ of the cyclic group $\bbF_q^{\times}.$ It induces an isomorphism of rings   $$\Z[t]/(t^{q-1}-1)\stackrel{\sim}{\to} Q_0(\bbF_q),\quad t\mapsto [x_0].$$ It follows that we have an isomorphism of $\bbF_q$-algebras \begin{equation}\label{eq:Q_0_for_F_q}\bbF_q\otimes Q_0(\bbF_q)\cong \bigoplus\limits_{x\in\bbF_q^{\times}}\bbF_q.\end{equation} In particular, $\bbF_q\otimes Q_0(\bbF_q)$ is a semi-simple $\bbF_q$-algebra. 

Note that we have an isomorphism $$\Hoch^{\bullet}(Q_{\bullet}(\bbF_q),\bbF_q)\cong \Hoch^{\bullet}_{\bbF_q}(\bbF_q\otimes Q_{\bullet}(\bbF_q),\bbF_q).$$
The natural homomorphism $\bbF_q\otimes Q_0(\bbF_q)\to\bbF_q$ corresponds under \eqref{eq:Q_0_for_F_q} to the projection onto the factor $\bbF_q$ for $x=x_0.$ From this and from the semi-simplicity of $\bbF_q\otimes Q_0(\bbF_q)$ we obtain that
the inclusion $$\Hoch^{\bullet}_{\bbF_q\otimes Q_0(\bbF_q)}(\bbF_q\otimes Q_{\bullet}(\bbF_q),\bbF_q)\hookrightarrow \Hoch^{\bullet}_{\bbF_q}(\bbF_q\otimes Q_{\bullet}(\bbF_q),\bbF_q)$$ is a quasi-isomorphism of DG rings. It remains to note that the reduced subcomplex of the LHS identifies with $\Hoch^{red,\bullet}(Q_{\bullet}(\bbF_q),\bbF_q).$\end{proof}

\section{Mac Lane (co)homology of the second kind of discrete valuation rings.}
\label{sec:HML_secomd_kind_discr_valuation}

In this section, we fix a finite extension $L$ of $\Q_p.$ We denote by $R=\cO_L$ the ring of integers, and by $\m\subset R$ the unique maximal ideal. Finally, the residue field $R/\m$ is denoted by $\mk.$ We put $q:=|\mk|,$ so that $\mk\cong \bbF_q.$ Recall that the inverse different $\cD_R^{-1}$ is given by the formula
$$\cD_R^{-1}=\{x\in L\mid \Tr_{L/\Q_p}(x)\in\Z_p\},$$
and $L/\Q_p$ is non-ramified if and only if $\cD_R^{-1}=R.$

\begin{theo}\label{th:THH_for_discr_valuation} (\cite{LM}) The non-zero topological Hochschild homology groups of $R$ are
$$THH_0(R)=R,\quad THH_{2n-1}(R)=\cD_R^{-1}/nR.$$\end{theo}

We have an immediate corollary of Theorem \ref{th:THH_for_discr_valuation}.

\begin{cor}\label{cor:HML_for_discr_valuation} We have non-canonical isomorphisms
$$HML_n(R)=\begin{cases}R & \text{for }n=0;\\
\cD_R^{-1}/kR & \text{for }n=2k-1>0;\\
0 & \text{otherwise;}\end{cases}$$
$$HML^n(R)=\begin{cases}R & \text{for }n=0;\\
\cD_R^{-1}/kR & \text{for }n=2k>0;\\
0 & \text{otherwise.}\end{cases}$$\end{cor}

\begin{proof}Indeed, by Theorem \ref{th:THH_equals_HML}, Mac Lane homology is just isomorphic to the topological Hochschild homology.

To compute Mac Lane cohomology, we apply Lemma \ref{lem:duality_of_Mac_Lane_complexes} to get an isomorphism
$$\Hoch^{\bullet}(Q_{\bullet}(R),R)\cong\bR\Hom_R(\Hoch_{\bullet}(Q_{\bullet}(R),R),R)$$ in $D(R).$ Since the ring $R$ has global dimension $1,$ we have an isomorphism
$$\Hoch_{\bullet}(Q_{\bullet}(R),R)\cong HML_{\bullet}(R)$$ in $D(R).$ It follows that
$$HML^n(R)\cong\Ext^0_R(HML_n(R),R)\oplus\Ext^1_R(HML_{n-1}(R),R).$$ Computing $\Ext$'s, we obtain the result.
\end{proof}

Let $w\in R$ be an element. Since $HML^{2n-1}(R)=0,$ it follows that we have $$HML^{II,1}(R,w)=0.$$
Further, by Proposition \ref{prop:product_on_HH_second_kind}, the $R$-algebra $HML^{II,0}(R,w)$ has an increasing exhausting filtration
$$F_0 HML^{II}(R,w)\subset F_1 HML^{II}(R,w)\subset\dots\subset HML^{II,0}(R,w)$$
by $R$-submodules such that $F_n\cdot F_m=F_{n+m}$ and we have an isomorphism of graded rings
$$\gr_{\bullet}^F HML^{II,0}(R,w)\cong \bigoplus\limits_{n\geq 0}HML^{2n}(R).$$

\subsection{Ramified case}
\label{ssec:ramified}

We first consider the case when the field extension $L/\Q_p$ is ramified, i.e. $\cD_R\ne R.$  We have $p\in\m^2.$

\begin{cor}\label{cor:HML_R_k_ramified}The groups $HML^{\bullet}(R,\mk)$ are given by
$$HML^n(R,\mk)=\mk,\quad n\geq 0.$$\end{cor}

\begin{proof}Taking any uniformizing element $\pi\in R,$ we get a short exact sequence of $R$-modules
\begin{equation}\label{eq:short_exact_uniformizing}0\to R\stackrel{\pi}{\to}R\to\mk\to 0.\end{equation}
Applying to it the long exact sequence of Mac Lane cohomology and using Corollary \ref{cor:HML_for_discr_valuation}, we obtain the result.\end{proof}

We would like to write an explicit cocycle representing a non-zero element in $HML^1(R,\mk).$ Fix a uniformizing element $\pi\in\m.$ Let us define the cochain $\xi\in\Hoch^1(Q_{\bullet}(R),\mk)$ by the formula
$$\xi:Q_0(R)\to\mk,\quad \xi([x])=\frac{x-x^q}{\pi}\text{ mod }\m.$$
Since $p\in\m^2,$ we have $\xi([x+y])=\xi([x])+\xi([y]).$ Further, we have
$$\frac{xy-(xy)^q}{\pi}-x\frac{y-y^q}{\pi}-\frac{x-x^q}{\pi}y=-\frac{(x-x^q)(y-y^q)}{\pi}\in\frac{\m^2}{\pi}=\m,$$
hence $\xi([xy])=x\xi([y])+\xi([x])y.$ Therefore $\xi$ is a cocycle.

Note that $\xi([\pi])=1,$ hence $\xi\ne 0.$ Finally, since the differential $b:\Hoch^0(Q_{\bullet}(R),\mk)\to\Hoch^1(Q_{\bullet}(R),\mk)$ is zero, it follows that $\xi$ represents a non-zero class in $HML^1(R,\mk).$

\begin{theo}\label{th:HML^II_R_k_ramified}Let $R$ be as above and $w\in R$ be an element such that $w-w^q\not\in\m^2.$ Denote by $\{(E_r^{\bullet},d_r)\}_{r\geq 1}$ the $\Z$-graded spectral sequence associated with $(\Hoch^{\bullet}(Q_{\bullet}(R),R),\delta_{[w]}).$ Then we have $E_2=E_{\infty}=0.$ In particular, $HML^{II,\bullet}(R,w;\mk)=0.$\end{theo}

\begin{proof}By Proposition \ref{prop:product_on_HH_second_kind} the differential $d_1$ on $E_1^{\bullet}$ is a derivation. Hence it suffices to show that $1\in E_1^0$ is contained in the image of $d_1:E_1^1\to E_1^0.$ Take the cocycle $\xi\in\Hoch^1(Q_{\bullet}(R),R)$ defined above, and denote by $\bar{\xi}$ its class in $HML^1(R)=E_1^1.$ By our assumption on $w,$ the element $\frac{w^q-w}{\pi}\in R$ is invertible. We have
$$d_1((\frac{w^q-w}{\pi})^{-1}\bar{\xi})=\delta_{[w]}((\frac{w^q-w}{\pi})^{-1}\xi)=1.$$
This proves the proposition.\end{proof}

We denote by $l>0$ the positive integer such that $\cD_R=\m^l.$

\begin{theo}\label{th:HML^II_discr_valuation_ramified}Let $R$ and $w$ be as in Theorem \ref{th:HML^II_R_k_ramified}. Then we have an isomorphism of filtered algebras $$HML^{II,0}(R,w)\cong L,$$ where the filtration on $L$ is given by the formula
\begin{equation}\label{eq:filtration_on_L_ramified}F_n L\cong \frac{1}{n!}\m^{-nl}.\end{equation}\end{theo}

\begin{proof}It follows from the short exact sequence \eqref{eq:short_exact_uniformizing} and the vanishing of $HML^{II,\bullet}(R,w;\mk)$ that the multiplication by $\pi$ in $HML^{II,0}(R,w)$ is an isomorphism. Since $HML^{II,0}(R,w)$ is an $R$-algebra, it follows that $\pi$ is invertible in $HML^{II,0}(R,w).$ In particular, the $R$-module $HML^{II,0}(R,w)$ has no torsion. This implies the isomorphism of $R$-algebras $HML^{II,0}(R,w)\cong L.$ The induced filtration $F_{\bullet}L$ on $L$ must satisfy the conditions
$$F_0L=R,\quad \gr_n^F L=R/n\m^{l},\quad n>0.$$ This implies \eqref{eq:filtration_on_L_ramified} by induction.
\end{proof}

As a corollary, we obtain the product structure on $HML^{\bullet}(R).$

\begin{cor}\label{cor:product_on_HML_ramified}We have natural isomorphism of graded $R$-algebras
$$HML^{\bullet}(R)\cong \Gamma_R(x)/(\cD_R\cdot x),$$
where $\deg(x)=2.$\end{cor}

\begin{proof}Indeed, let us choose any element $w\in R$ such that $w-w^q\not\in\m^2$ (for example, $w=\pi$). By Theorem \ref{th:HML^II_discr_valuation_ramified}, we have isomorphisms $HML^{\bullet}(R)\cong \gr_{\bullet}^F HML^{II}(R,w)\cong \gr_{\bullet}^F L,$ where the filtration $F_{\bullet}L$ is given by \eqref{eq:filtration_on_L_ramified}. Here the degree of $\gr_n^F$ is defined to be $2n.$ We define a morphism
$$\varphi:\Gamma_R(x)/(\cD_R\cdot x)\to \gr_{\bullet}^F L$$
by the formula $$\varphi(\gamma_n(x))=\frac{\pi^{-nl}}{n!}.$$
It is easy to see that $\varphi$ is an isomorphism. This proves the corollary.\end{proof}

We have the following corollary for the Mac Lane homology of the second kind.

\begin{theo}\label{th:HML^II_dual_for_discr_valuation_ramified}Let $R$ and $w$ be as in Theorem \ref{th:HML^II_discr_valuation_ramified}. Then we have $HML^{II}_{\bullet}(R,w)=0.$
\end{theo}

\begin{proof}By Theorem \ref{th:HML^II_discr_valuation_ramified} and Lemma \ref{lem:duality_for_homology}, we have isomorphisms $HML^{II}_0(R,w)\cong\Hom_R(L,R)$ and $HML^{II}_1(R,w)\cong\Ext^1_R(L,R).$ Obviously, we have $\Hom_R(L,R)=0.$ To compute $\Ext^1_R(L,R),$ take an injective resolution
$$0\to R\to L\to L/R\to 0.$$ We have $\Hom_R(L,L)=L=\Hom_R(L,L/R).$ Therefore, $\Ext^1_R(L,R)=0.$ This proves 1).
\end{proof}

\begin{cor}\label{cor:sp_seq_for_homology_ramified}Let $R$ and $w$ be as in Theorem \ref{th:HML^II_discr_valuation_ramified}.  Let $\{(E_r^{\bullet},d_r)\}$ be a $\Z$-graded spectral sequence associated with the mixed complex $(\Hoch_{\bullet}(Q_{\bullet}(R),R),\delta_{[w]}).$ Then for each $r$ the differential
$$d_r:E_r^0\to E_r^{-2r+1}=HML_{2r-1}(R)$$ is surjective.\end{cor}

\begin{proof}Indeed, this follows immediately from Theorem \ref{th:HML^II_dual_for_discr_valuation_ramified}.\end{proof}

\subsection{Non-ramified case}
\label{ssec:nonramified}

In this subsection we consider the case when the extension $L/\Q_p$ is non-ramified, i.e. $\cD_R=R.$ We have $\m=pR.$

\begin{cor}\label{cor:HML_for_torsion_modules} Denote by $R\text{-mod}_{\m}$ the category of finitely generated $R$-modules, which are $\m$-torsion (i.e. annihilated by some power of $\m$). Take some positive integer $n=p^tl,$ where $t\geq 0$ and $l$ is coprime to $p.$ Then one can choose isomorphisms of functors
\begin{equation}\label{eq:functorial_HML^2n} HML^{2n}(R,M)\cong M/\m^tM,\end{equation}
\begin{equation}\label{eq:functorial_HML^2n-1} HML^{2n-1}(R,M)\cong\Tor_1^R(R/\m^t,M)\cong \ker(\m^t\otimes_R M\to M),\end{equation}
where $M\in R\text{-mod}_{\m}.$\end{cor}

\begin{proof}Indeed, for any $M\in\Perf(R)$ we have natural isomorphism in $D(R):$
$$\Hoch^{\bullet}(Q_{\bullet}(R),M)\cong \Hoch^{\bullet}(Q_{\bullet}(R),R)\stackrel{\bL}{\otimes}_R M.$$
If we fix for each $n=p^tl$ as above an isomorphism $HML^{2n}(R)\cong R/\m^t$ of $R$-modules (it exists by Corollary \ref{cor:HML_for_discr_valuation}), then we get the functorial isomorphisms \eqref{eq:functorial_HML^2n}, \eqref{eq:functorial_HML^2n-1}.\end{proof}

\begin{theo}\label{th:HML_R_k}(\cite{FLS}) We have
$$HML^{n}(R,\mk)=\begin{cases}\mk & \text{for }n\equiv -1,0\text{ mod }2p;\\
0 & \text{otherwise.}\end{cases}$$
Moreover, we have an isomorphism of graded $\mk$-algebras
$$HML^{\bullet}(R,\mk)\cong \mk\otimes \Lambda_1\otimes\Lambda(\xi_1),$$
where $\xi_1$ is a variable of degree $2p-1$ and $\Lambda(\xi_1)$ is an exterior algebra in one variable.\end{theo}

\begin{proof}The same statement for $HML^{\bullet}(\Z,\Z/p\Z)$ is proved in \cite{FLS} and \cite{L}, Section 13.4. The same proof works in our case.
\end{proof}

Under the assumptions of Theorem \ref{th:HML_R_k}, we would like to write down the representative of the class $\xi_1\in HML^{2p-1}(R,\mk).$ Let us put $q:=|\mk|,$ so that $\mk\cong\bbF_q.$

We have a natural morphism of complexes $$f:\Hoch^{\bullet}(Q_{\bullet}(\mk),\mk)\to \Hoch^{\bullet}(Q_{\bullet}(R),\mk).$$  Recall the cocycle $\alpha_q\in\Hoch^2(Q_{\bullet}(\mk),\mk)$ from \eqref{eq:cocycle_alpha_q}, which represents the class $e_0\in HML^2(\mk).$ By Theorem \ref{th:HML_R_k}, we have $HML^2(R,\mk)=0.$ In particular, the cocycle $f(\alpha_q)\in\Hoch^2(Q_{\bullet}(R),\mk)$ is a coboundary. Moreover, we have
$$f(\alpha_q)=b(\beta),\quad \beta\in\Hoch^1(Q_{\bullet}(R),\mk),\quad\beta([a])=\frac{a-a^q}{p}\text{ mod }\m.$$
By Theorem \ref{th:HML_F_q_F_q}, the cocycle $(\alpha_q)^p\in\Hoch^{2p}(Q_{\bullet}(\mk),\mk)$ is actually a coboundary. Let us fix once and for all the cochain $\gamma_{2p-1}\in \Hoch^{2p-1}(Q_{\bullet}(\mk),\mk),$ satisfying $b(\gamma_{2p-1})=(\alpha_q)^p.$ By Lemma \ref{lem:reduced_for_HML_of_F_q}, we may and will assume that the cochain $\gamma_{2p-1}$ vanishes on $(a_1,\dots,a_i)$ if at least one of $a_j$ is in $Q_0(\mk).$

\begin{prop}\label{prop:cocycle_for_xi_1}The cochain $\beta\cdot f(\alpha_q)^{p-1}-f(\gamma_{2p-1})\in\Hoch^{2p-1}(Q_{\bullet}(R),\mk))$ is a cocycle, representing a non-zero class in $HML^{2p-1}(R,\mk),$ which we assume to be $\xi_1.$\end{prop}

\begin{proof} Indeed, we have
$$b(\beta\cdot f(\alpha_q)^{p-1}-f(\gamma_{2p-1}))=b(\beta)\cdot f(\alpha_q)^{p-1}-f(b(\gamma_{2p-1}))=f(\alpha_q)^p-f((\alpha_q)^p)=0,$$
so $\beta\cdot f(\alpha_q)^{p-1}-f(\gamma_{2p-1})$ is a cocycle.

Let us put $\cL^{\bullet}:=\coker(f),$ so that we have a short exact sequence of DG $\Hoch^{\bullet}(Q_{\bullet}(\mk),\mk)$-modules
$$0\to \Hoch^{\bullet}(Q_{\bullet}(\mk),\mk)\to \Hoch^{\bullet}(Q_{\bullet}(R),\mk)\to \cL^{\bullet}\to 0.$$

It induces a long exact sequence in cohomology:
\begin{equation}
\label{eq:long_exact_sequence}
\dots\to HML^n(\mk)\to HML^n(R,\mk)\to H^n(\cL^{\bullet})\to HML^{n+1}(\mk)\to\dots
\end{equation}
By Proposition \ref{prop:sp_seq_change_of_rings}, we have the change of rings spectral sequence 
$$E_2^{p,q}=HML^p(\mk,\Ext^q_R(\mk,\mk))\Rightarrow HML^{p+q}(R,\mk).$$
We have $\Ext^q_R(\mk,\mk)=\mk$ for $q=0,1$ and zero for $q>2.$ In particular, the only non-zero differential in the spectral sequence is $d_2.$ It follows that we have a natural isomorphism of graded $HML^{\bullet}(\mk)$-modules
$H^{\bullet}(\cL^{\bullet})\cong E_2^{\bullet,q}\cong \mk\otimes\Lambda[-1],$ and the map $d_2$ is just the boundary map in the sequence \eqref{eq:long_exact_sequence}. It follows that this map is the multiplication by $\lambda\otimes e_0,$ $\lambda\in\mk^{\times}.$ We may and will assume that $\lambda=1.$

Denote by $\bar{\beta}\in L^1$ the projection of $\beta.$ Clearly, $\bar{\beta}$ is a non-zero cocycle. Moreover, since $\cL^0=0,$ wee see that $\bar{\beta}$ represents a nonzero class in $H^1(\cL^{\bullet}).$ Therefore, the class $[\bar{\beta}]\in H^1(\cL^{\bullet})$ is actually a generator of $H^{\bullet}(\cL^{\bullet})$ as the free $HML^{\bullet}(\mk)$-module of rank $1.$ In particular, the cocycle $\bar{\beta}\cdot(\alpha_q)^{p-1}\in L^{2p-1}$ represents a non-zero class in cohomology. It follows that the cocycle $\beta\cdot f(\alpha_q)^{p-1}-f(\gamma_{2p-1})\in\Hoch^{2p-1}(Q_{\bullet}(R),\mk))$ represents non-zero class in $HML^{2p-1}(R,\mk).$ This proves the proposition.
\end{proof}

Let us now fix an element $w\in R.$ Our first goal is to compute the Mac Lane cohomology of the second kind
$$HML^{II,\bullet}(R,w;\mk),$$

By Proposition \ref{prop:product_on_HH_second_kind}, in the spectral sequence $\{(E_r^{\bullet},d_r)\}_{r\geq 1}$ associated with mixed complex $(\Hoch^{\bullet}(Q_{\bullet}(R),\mk),\delta_{[w]})$ each $E_r^{\bullet}$ is a graded $\mk$-algebra, and each $d_r$ is a derivation. We will use this to compute the differentials in this spectral sequence. 

\begin{theo}\label{th:HML^II(R,k)}Consider the (converging) spectral sequence
\begin{equation}\label{eq:sp_seq_HML^II_R_k}E_1^n=HML^n(R,\mk)\Rightarrow HML^{II,\bar{n}}(R,w;\mk).\end{equation}

1) If $w\equiv w^q\text{ mod }\m^2,$ then \eqref{eq:sp_seq_HML^II_R_k} degenerates at $E_1,$ i.e. $E_1=E_{\infty}.$ In particular, we have a non-canonical isomorphism of $\Z/2$-graded $\mk$-modules $$\Tot^{\oplus}(HML^{\bullet}(R,\mk))\cong HML^{II,\bullet}(R,w;\mk).$$

2) If $w\not\equiv w^q\text{ mod }\m^2,$ then we have $E_1=E_p,$ and $E_{p+1}=0,$ hence $HML^{II,\bullet}(R,w;\mk)=0.$ Moreover, the derivation $d_p$ on $E_p=\mk\otimes\Lambda_1\otimes\Lambda(\xi_1)$ is uniquely determined by the formula
$$d_p(\xi_1)=(\frac{w^q-w}{p})^p\text{ mod }\m;\quad d_p(e_k)=0,\quad k\geq 1.$$\end{theo}

\begin{proof}By Proposition \ref{prop:functoriality_of_sp_seq}, the morphism of mixed complexes $$f:(\Hoch^{\bullet}(Q_{\bullet}(\mk),\mk),\delta_{[\bar{w}]})\to (\Hoch^{\bullet}(Q_{\bullet}(R),\mk),\delta_{[w]})$$
induces the morphism of the corresponding $\Z$-graded spectral sequences. But the spectral sequence for the first mixed complex degenerates at $E_1$ for degree reasons. It follows that the spectral sequence \eqref{eq:sp_seq_HML^II_R_k} satisfies the property $d_r(\bar{e_k})=0$ for $r,k>0.$ Hence, it suffices to prove the equality
\begin{equation}\label{eq:d_p(xi_1)}d_p(\xi_1)=(\frac{w^q-w}{p})^p\text{ mod }\m.\end{equation}

Consider the cocycle $(\beta\cdot f(\alpha_q)^{p-1}-f(\gamma_{2p-1}))\in\Hoch^{2p-1}(Q_{\bullet}(R),\mk),$ which represents $\xi_1$ by Proposition \ref{prop:cocycle_for_xi_1}. By our assumption on $\gamma_{2p-1},$ we have that $\delta_{[w]}(f(\gamma_{2p-1}))=0.$ Moreover, we have
$$\delta_{[w]}(\beta\cdot f(\alpha_q)^i)=(\frac{w^q-w}{p})f(\alpha_q)^i,\quad b(\beta\cdot f(\alpha_q)^i)=f(\alpha_q)^{i+1},\quad i\geq 0.$$
This immediately implies the equality \eqref{eq:d_p(xi_1)}. Theorem is proved.\end{proof}

\begin{lemma}\label{lem:sp_seq_degen_for_p_delta}
Let $a_1,\dots,a_k\in R$ and $w_1,\dots,w_k\in R$ be two collections of elements.
Fix some integer $n>0.$ Consider the mixed complex
$$\cK^{\bullet}=(\Hoch^{\bullet}(R,R/\m^n),\delta),$$
where $$\delta=p(a_1\delta_{[w_1]}+\dots+a_k\delta_{[w_k]}).$$
Then the spectral sequence associated with $\cK^{\bullet}$ degenerates at $E_1.$\end{lemma}

\begin{proof}Consider the finite decreasing filtration $F^{\bullet}\cK^{\bullet}$ given by the formula
$$F^i\cK^{\bullet}:=(\Hoch^{\bullet}(R,\m^i/\m^n),\delta),\quad 0\leq i\leq n.$$ We have $F^0\cK^{\bullet}=\cK^{\bullet}$ and $F^n\cK^{\bullet}=0.$ For convenience, we put $F^i\cK^{\bullet}=0$ for $i\geq n.$ By definition of $\delta,$ we have $\delta(F^i\cK^{\bullet})\subseteq F^{i+1}\cK^{\bullet-1}.$

{\noindent{\bf Sublemma.}} {\it Let $m$ be a positive integer. Let us write it as $m=p^t l,$ where $l$ is coprime to $p.$ Then for any $j\geq t$ the map $H^{2m}(F^j\cK^{\bullet},b)\to H^{2m}(F^{j-t}\cK^{\bullet},b)$ is zero.

In particular, for any $b$-cocycle $\alpha\in F^j\cK^{2m}$ we can find a cochain $\beta\in F^{j-t}\cK^{2m-1}$ such that $b(\beta)=\alpha.$}

\begin{proof} We may and will assume that $j\leq n-1$ (otherwise $F^j\cK^{\bullet}=0$). Note that the inclusion $\m^j/\m^n\to \m^{j-t}/\m^n$ is annihilated by the functor $M\mapsto M/\m^t.$ Hence, Corollary \ref{cor:HML_for_torsion_modules} implies that the map $HML^{2m}(R,\m^j/\m^n)\to HML^{2m}(R,\m^{j-t}/\m^n)$ is zero. This proves the sublemma.\end{proof}

Note that for any $m\geq 0$ the map $$HML^{2m}(R,R)\to HML^{2m}(R,R/\m^n)$$ is surjective. But the spectral sequence of the mixed complex $(\Hoch^{\bullet}(Q_{\bullet}(R),R))$ degenerates for degree reasons. It follows that $E^{2m}_r=Z^{2m}_r,$ $r>0,$ $m\geq 0.$

Now we prove that $E^{2m-1}_r=Z^{2m-1}_r,$ $r,m>0.$ Let us write $m=p^t l,$ where $l$ is coprime to $p.$ First we claim that we may assume $t\geq n.$ Indeed, if $t<n,$ then by Corollary \ref{cor:HML_for_torsion_modules} the inclusion $R/\m^{n-1}\stackrel{p}{\to} R/\m^n$ induces an isomorphism $$HML^{2m-1}(R,R/\m^{n-1})\stackrel{\sim}{\to} HML^{2m-1}(R,R/\m^n).$$ By functoriality of spectral sequences, the problem then reduces to the case of smaller $n.$

Now we assume that $t\geq n.$ Take some cocycle $x\in\cK^{2m-1}.$ It suffices to construct a collection of cochains  $x_1,\dots,x_{p^n}$ of $\cK^{\bullet},$ such that
$$x_i\in\cK^{2m-2i+1},\quad b(x_{i+1})=\delta(x_i)\text{ for }1\leq i\leq p^n-1,\quad x_1=x,\,\delta(x_{p^n})=0.$$ We claim that such collection exists, and moreover it can be chosen to satisfy an additional assumption $x_i\in F^{i-1-\ord_p(i-1)!}K^{2m-2i+1}.$ To construct it, let us first note that
\begin{equation}\label{eq:inequality}i-1-\ord_p(i-1)!=i-1-\sum\limits_{j\geq 1}\left[\frac{i-1}{p^j}\right]\geq \frac{(p-2)(i-1)}{p-1}\geq 0\end{equation} for $i\geq 0.$
Suppose that we have constructed $x_1,\dots,x_i,$ $i\leq p^n-1,$ satisfying the required properties. We have that $\delta(x_i)\in F^{i-\ord_p(i-1)!}\cK^{2m-2i}.$ By \eqref{eq:inequality} we have that
 $$(i-\ord_p(i-1)!)-\ord_p(m-i)=i-\ord_p(i!)\geq 0.$$ Therefore, by the sublemma, there exists $x_{i+1}\in F^{i-\ord_p(i!)}\cK^{2m-2i-1},$ such that $b(x_{i+1})=\delta(x_i).$

 It remains to show that $\delta(x_{p^n})=0.$ By \eqref{eq:inequality}, we see that $$p^n-\ord_p(p^n-1)!-n=p^n-\ord_p(p^n)!\geq 0.$$ Thus, $$\delta(x_{p^n})\in F^{p^n-\ord_p(p^n-1)!}\cK^{2m-2p^n}\subseteq F^n\cK^{2m-2p^n}=0.$$
 This proves the proposition.
\end{proof}

\begin{lemma}\label{lem:sp_seq_degen_w^q=w}Suppose that $p=\cha(\mk)\ne 2.$ Assume that that $(w^q-w)\in\m^2.$ Then for each $n\in\Z_{>0}$ the spectral sequence associated with the mixed complex $(\Hoch^{\bullet}(R,R/\m^n),\delta_{[w]})$ degenerates at $E_1.$\end{lemma}

\begin{proof}We consider the following cases.

1) Suppose that $w\in\m.$ Then $w^q\in\m^2,$ hence $w\in\m^2.$ Let us write $w=p^2x.$ By Lemma \ref{lem:q_is_for_delta_w_1...w_k}, there is an $A_{\infty}$-isomorphism of mixed complexes
$$(\Hoch^{\bullet}(Q_{\bullet}(R),R/\m^n),\delta_{[w]})\simeq (\Hoch^{\bullet}(Q_{\bullet}(R),R/\m^n),p(\delta_{[px]}+x\delta_{[p]})),$$
hence the associated spectral sequences are isomorphic.
By Lemma \ref{lem:sp_seq_degen_for_p_delta}, the spectral sequence associated with the later mixed complex degenerates at $E_1.$

2) Suppose that $w\not\in\m.$ By our assumptions on $w$ and on $p,$ we can find $w'\in R$ such that $w=w'^p.$ Again by Lemma \ref{lem:q_is_for_delta_w_1...w_k} we have an $A_{\infty}$-quasi-isomorphism of mixed complexes
$$(\Hoch^{\bullet}(Q_{\bullet}(R),R/\m^n),\delta_{[w]})\simeq (\Hoch^{\bullet}(Q_{\bullet}(R),R/\m^n),pw'^{p-1}\delta_{[w']}).$$
By Lemma \ref{lem:sp_seq_degen_for_p_delta}, the spectral sequence associated with the later mixed complex degenerates at $E_1.$\end{proof}

\begin{lemma}\label{lem:extensions} Consider a a short exact sequence
$$0\to M'\to M\to M''\to 0$$
of finitely generated $R$-modules. Then the following are equivalent:

(i) we have a splitting $M=M'\oplus M'';$

(ii) for all $n>0,$ the boundary map $\Tor_1^R(R/\m^n,M'')\to\Tor_0^R(R/\m^n,M')$ is zero;

(iii) for all $n>0,$ the sequence $$0\to M'\otimes R/\m^n\to M\otimes R/\m^n\to M''\otimes R/\m^n$$ is short exact.\end{lemma}

\begin{proof}This is standard and straightforward, see \cite{AM}.\end{proof}

Now we are able to compute the Mac Lane cohomology of the second kind of $(R,w).$

\begin{theo}\label{th:HML^II_for_discr_valuation} Consider the following cases.

1) If $w\not\equiv w^q\text{ mod }\m^2,$ then we have an isomorphism of filtered $R$-algebras
$$HML^{II,0}(R,w)\cong L,$$
where $F_n L=\frac{1}{n!}R.$

2) If $w\equiv w^q\text{ mod }\m^2$ and $p=\cha(\mk)\ne 2,$ then we have a non-canonical isomorphism of filtered $R$-modules
\begin{equation}\label{eq:splitting_for_HML^II_discr_valuation} HML^{II,0}(R,w)\cong \bigoplus\limits_{m\geq 0} HML^{2m}(R).\end{equation}\end{theo}

\begin{proof}1) By Theorem \ref{th:HML^II(R,k)}, we have $HML^{II}(R,w;\mk)=0.$
The short exact sequence
$$0\to R\stackrel{p}{\to}R\to \mk\to 0$$
implies that the multiplication by $p$ in $HML^{II,0}(R,w)$ is an isomorphism. Therefore, the element $p$ is invertible in the ring $HML^{II,0}(R,w).$ This implies the isomorphism of $R$-algebras $HML^{II,0}(R,w)\cong R[\frac1{p}]\cong L.$ The induced filration $F_{\bullet}L$ on $L$ must satisfy the conditions $F_0 L=R,$ and $F_n L/F_{n-1} L\cong R/\m^{\ord_p(n)}\cong R/nR.$ This implies by induction that $F_n L=\frac1{n!} R$ for $n\geq 0.$ The statement 1) is proved.

2) By Lemma \ref{lem:sp_seq_degen_w^q=w}, for each $n\geq 0,$ $l>0,$ we have a short exact sequence
$$0\to F_n HML^{II,0}(R,w;R/\m^l)\to F_{n+1} HML^{II,0}(R,w;R/\m^l)\to HML^{2n+2}(R,w;R/\m^l)\to 0.$$
It follows by induction on $n$ that the natural map $$F_n HML^{II,0}(R,w)\otimes R/\m^l\to F_n HML^{II,0}(R,w;R/\m^l)$$ is an isomorphism. Then by Lemma \ref{lem:extensions} (implication (iii)$\Rightarrow$(i)), we see that the short exact sequence of $R$-modules
$$0\to F_n HML^{II,0}(R,w)\to F_{n+1} HML^{II,0}(R,w)\to HML^{2n+2}(R)\to 0$$
is split for all $n\geq 0.$ This gives an isomorphism \eqref{eq:splitting_for_HML^II_discr_valuation}.\end{proof}

As in the ramified case, we obtain the product structure on $HML^{\bullet}(R).$

\begin{cor}\label{cor:product_on_HML_non_ramified}We have natural isomorphism of graded $R$-algebras
$$HML^{\bullet}(R)\cong \Gamma_R(x)/(x),$$
where $\deg(x)=2.$\end{cor}

\begin{proof}This is proved in the same way as Corollary \ref{cor:product_on_HML_ramified}\end{proof}

We have the following corollary for the Mac Lane homology of the second kind.

\begin{theo}\label{th:HML^II_dual_for_discr_valuation}Let $R$ and $w$ be as above.

1) If $w\not\equiv w^q\text{ mod }\m^2,$ then $HML^{II}_{\bullet}(R,w)=0.$

2) If $w\equiv w^q\text{ mod }\m^2$ and $p=\cha(\mk)\ne 2,$ then we have a canonical isomorphism 
$$HML^{II}_0(R,w)\cong R,$$
and a non-canonical isomorphism
$$HML^{II}_1(R,w)\cong \prod\limits_{m>0}HML_{2m-1}(R).$$\end{theo}

\begin{proof}1) is proved analogously to Theorem \ref{th:HML^II_dual_for_discr_valuation_ramified}.

2) follows directly from Theorem \ref{th:HML^II_for_discr_valuation}, 2) and Lemma \ref{lem:duality_for_homology}.\end{proof}

\begin{cor}\label{cor:sp_seq_for_homology_non_ramified}Let $R$ and $w$ be as above, and $w\not\equiv w^q\text{ mod }\m^2.$ Let $\{(E_r^{\bullet},d_r)\}$ be a $\Z$-graded spectral sequence associated with the mixed complex $(\Hoch_{\bullet}(Q_{\bullet}(R),R),\delta_{[w]}).$ Then for each $r$ the differential
$$d_r:E_r^0\to E_r^{-2r+1}=HML_{2r-1}(R)$$ is surjective.\end{cor}

\begin{proof}Indeed, this follows immediately from Theorem \ref{th:HML^II_dual_for_discr_valuation} 1).\end{proof}

\section{Mac Lane (co)homology of localizations of number rings}
\label{sec:HML_loc_number_rings}

In this section $K$ denotes a number field (a finite extension of $\Q$), and $A=\cO_K\subset K$ the ring of integers. For any finite subset $S\subset\Spec_m(A)$ we put $$S^{-1}A:=\cO(\Spec A\setminus S)\subset K.$$
For an arbitrary (possibly infinite) subset $T\subset\Spec_m(A)$ we put
$$T^{-1}A:=\{x\in K\mid x\in A_{\rho}\text{ for }\rho\in\Spec_m(A)\setminus T\}.$$

Recall that the inverse different $\cD_{A}^{-1}\subset K$ is a fractional ideal given by the formula
$$\cD_{A}^{-1}=\{x\in K\mid \Tr_{K/\Q}(x)\in \Z\}.$$ We put $$\cD_{S^{-1}A}=\cD_{A}\otimes_A S^{-1}A.$$

We recall the following theorem of Lindenstrauss and Madsen.

\begin{theo}(\cite{LM})\label{th:THH_number_rings} The non-zero topological Hochschild homology groups of $A$ are
$$THH_0(A)=A,\quad THH_{2n-1}(A)=\cD_A^{-1}/nA.$$\end{theo}

We have an immediate corollary of Theorem \ref{th:THH_number_rings}.

\begin{cor}\label{cor:HML_for_loc_of_number_rings} For any finite subset $S\subset\Spec_m(A),$ we have non-canonical isomorphisms
$$HML_n(S^{-1}A)=\begin{cases}S^{-1}A & \text{for }n=0;\\
\cD_{S^{-1}A}^{-1}/kS^{-1}A & \text{for }n=2k-1>0;\\
0 & \text{otherwise;}\end{cases}$$
$$HML^n(S^{-1}A)=\begin{cases}S^{-1}A & \text{for }n=0;\\
\cD_{S^{-1}A}^{-1}/kS^{-1}A & \text{for }n=2k>0;\\
0 & \text{otherwise.}\end{cases}$$\end{cor}

\begin{proof}For $S=\emptyset,$ this is proved analogously to Corollary \ref{cor:HML_for_discr_valuation}. The general case is implied by Proposition \ref{prop:HML_localizations}.\end{proof}

From now on in this section we fix a finite subset $S\subset\Spec_m(A)$ and an element $w\in S^{-1}A.$

Since $HML^{2n-1}(S^{-1}A)=0,$ it follows that we have $$HML^{II,1}(S^{-1}A,w)=0.$$
Further, by Proposition \ref{prop:product_on_HH_second_kind}, the $S^{-1}A$-algebra $HML^{II,0}(S^{-1}A,w)$ has an increasing exhausting filtration
$$F_0 HML^{II}(S^{-1}A,w)\subset F_1 HML^{II}(S^{-1}A,w)\subset\dots\subset HML^{II,0}(S^{-1}A,w)$$
by $S^{-1}A$-submodules such that $F_n\cdot F_m=F_{n+m}$ and we have an isomorphism of graded rings
$$\gr_{\bullet}^F HML^{II,0}(S^{-1}A,w)\cong \bigoplus\limits_{n\geq 0}HML^{2n}(S^{-1}A).$$

We put
$$Crit(w):=\{\rho\in\Spec_m (S^{-1}A)\mid w^{|k(\rho)|}-w\in\rho^2\}\subset \Spec_m(S^{-1}A).$$
We identify the set $\Spec(S^{-1}A)$ with an open subset of $\Spec(A).$ In particular, for any $\rho\in \Spec_m(S^{-1}A)$ we have the localization $(\rho\cup S)^{-1}A.$

We need one more lemma.

\begin{lemma}\label{lem:lemma_on_completions}Let $A,S$ and $w$ be as above, and $\rho\subset S^{-1}A$ a maximal ideal. 

1) We have quasi-isomorphisms of mixed complexes 
$$(\Hoch^{\bullet}(Q_{\bullet}(S^{-1}A),S^{-1}A),\delta_{[w]})\otimes_{S^{-1}A}\widehat{A}_{\rho}\stackrel{\sim}{\to} (\Hoch^{\bullet}(Q_{\bullet}(S^{-1}A),\widehat{A}_{\rho}),\delta_{[w]}),$$
\begin{equation}\label{eq:q_is_completions_limit}(\Hoch^{\bullet}(Q_{\bullet}(\widehat{A}_{\rho}),\widehat{A}_{\rho}),\delta_{[w]})\stackrel{\sim}{\to}(\Hoch^{\bullet}(Q_{\bullet}(S^{-1}A),\widehat{A}_{\rho}),\delta_{[w]}).\end{equation}

2) We have a quasi-isomorphism of mixed complexes
$$(\Hoch_{\bullet}(Q_{\bullet}(S^{-1}A),S^{-1}A),\delta_{[w]})\otimes_{S^{-1}A}\widehat{A}_{\rho}\stackrel{\sim}{\to} (\Hoch_{\bullet}(Q_{\bullet}(\widehat{A}_{\rho}),\widehat{A}_{\rho}),\delta_{[w]})$$\end{lemma}

\begin{proof}The quasi-isomorphism \eqref{eq:q_is_completions_limit} is implied by Proposition \ref{prop:HML_completions} (passing to the homotopy limits). The rest follow from the computations of Mac Lane (co)homology, see Corollaries \ref{cor:HML_for_discr_valuation} and \ref{cor:HML_for_loc_of_number_rings}.\end{proof}

\begin{theo}\label{th:HML^II_for_loc_of_number_rings}Let $A,S$ and $w$ be as above.

1) We have an isomorphism of filtered $S^{-1}A$-algebras
\begin{equation}\label{eq:tensor_product_decomp}HML^{II,0}(S^{-1}A,w)\cong\bigotimes_{\rho\in\Spec_m(S^{-1}A)}B(\rho),\end{equation}
where $B(\rho)$ is a filtered $S^{-1}A$-algebra with $F_0 B(\rho)=S^{-1}A,$ and 
$$\gr_n ^F B(\rho)\cong (\gr_n^F HML^{II,0}(S^{-1}A,w))_{\rho},\quad n>0.$$

2) If $\rho\not\in Crit(w),$ then we have a natural isomorphism of filtered $(S^{-1}A)_{\rho}$-algebras
$$B(\rho)=(\rho\cup S)^{-1} A,$$
where the filtration on the RHS is given by the formula \begin{equation}\label{eq:filtration_by_factorials}F_n(\rho\cup S)^{-1} A=(\frac{1}{n!}\cD_{S^{-1}A}^{-n})\cap (\rho\cup S)^{-1}A.\end{equation}

3) If $\rho\in\Crit(w)$ and $\rho\not\in\supp(S^{-1}A/2\cD_{S^{-1}A}),$ then there is a non-canonical isomorphism of filtered $S^{-1}A$-modules 
$$B(\rho)\cong S^{-1}A\oplus\bigoplus\limits_{n\geq 1}S^{-1}A/\rho^{\ord_p(n)}.$$
\end{theo}

\begin{proof}1) Note that we have $F_0 HML^{II,0}(S^{-1}A,w)=S^{-1}A,$ and $\gr_n^F HML^{II,0}(S^{-1}A,w)$ is a torsion finitely generated $S^{-1}A$-module for $n>0.$. This immediately implies the unique tensor product decomposition \eqref{eq:tensor_product_decomp} with required properties. 

2) and 3). We have an isomorphism 
\begin{equation}\label{eq:B(rho)_and_HML^II(A)} HML^{II,0}(S^{-1}A,w)\otimes_{S^{-1}A}\widehat{A}_{\rho}\cong B(\rho)\otimes_{S^{-1}A}\widehat{A}_{\rho}.\end{equation}
On the other hand, Lemma \ref{lem:lemma_on_completions} implies an isomorphism
\begin{equation}\label{eq:HML^II(A)_completion} HML^{II,0}(S^{-1}A,w)\otimes_{S^{-1}A}\widehat{A}_{\rho}\cong HML^{II,0}(\widehat{A}_{\rho},w).\end{equation}

Combining \eqref{eq:B(rho)_and_HML^II(A)} and \eqref{eq:HML^II(A)_completion}, we get an isomorphism of filtered $\widehat{A}_{\rho}$-algebras
$$B(\rho)\otimes_{S^{-1}A}\widehat{A}_{\rho}\cong HML^{II,0}(\widehat{A}_{\rho},w).$$

Then 2) is implied by Theorem \ref{th:HML^II_discr_valuation_ramified} and Theorem \ref{th:HML^II_for_discr_valuation} 1). Finally, 3) is implied by Theorem \ref{th:HML^II_for_discr_valuation} 2). This proves the theorem.
\end{proof}

\begin{remark}\label{remark:formula_for_HML^II_of _loc_of_number_rings}1) Note that  Theorem \ref{th:HML^II_for_loc_of_number_rings} does not give a description of $B(\rho)$ for those $\rho\in Crit(w)$ which are  contained in $\supp(S^{-1}A/2\cD_{S^{-1}A}).$

2) If we have $Crit(w)\cap \supp(S^{-1}A/2\cD_{S^{-1}A})=\emptyset,$ then Theorem \ref{th:HML^II_for_loc_of_number_rings} gives an isomorphism
\begin{equation}\label{eq:formula_for_HML^II_of _loc_of_number_rings}HML^{0,II}(S^{-1}A,w)=(T\cup S)^{-1}A\oplus\bigoplus\limits_{\substack{\rho\in Crit(w);\\ n>0}}S^{-1}A/\rho^{\ord_p(n)},\end{equation}
where we put $$T:=\Spec_m(S^{-1}A)\setminus Crit(w)\subset \Spec_m(S^{-1}A).$$\end{remark}

As an application, we describe the commutative ring structure on $HML^{\bullet}(A).$

\begin{theo}\label{th:ring_structure_on_HML_S^-1A} We have an isomorphism of graded rings \begin{equation}\label{eq:HML_equals_Gamma}HML^{\bullet}(A)\cong \Gamma_{A}(x)/(\cD_A\cdot x),\end{equation} where $\deg(x)=2.$\end{theo}

\begin{proof} By Corollary \ref{cor:HML_for_loc_of_number_rings}, we have $HML^0(A)=A,$ and $HML^{2n}(A)$ is torsion for $n>0.$ This immediately implies the decomposition $$HML^{\bullet}(A)\cong\bigotimes\limits_{\rho\in\Spec_m(A)} C(\rho),$$
where $C(\rho)$ is a non-negative even graded algebra with $C(\rho)^0=A$ and $$C(\rho)^{2n}=HML^{2n}(A)_{\rho},\quad n>0.$$

By Lemma \ref{lem:lemma_on_completions}, we have an isomorphism of graded algebras \begin{equation}\label{eq:C_rho_otimes_completion}C(\rho)\otimes_A\widehat{A}_{\rho}\cong HML^{\bullet}(\widehat{A}_{\rho}).\end{equation} 
By Corollaries \ref{cor:product_on_HML_ramified} and \ref{cor:product_on_HML_non_ramified} we have an isomorphism
$$HML^{\bullet}(\widehat{A}_{\rho})\cong \Gamma_{\widehat{A}_{\rho}}(x)/(\cD_{\widehat{A}_{\rho}}\cdot x)\cong\widehat{A}_{\rho}\oplus (\Gamma_A^{>0}(x)/(\cD_A\cdot x))_{\rho}.$$
Combining this with \eqref{eq:C_rho_otimes_completion}, we obtain the isomorphism
$$C(\rho)\cong A\oplus (\Gamma_A^{>0}(x)/(\cD_A\cdot x))_{\rho}.$$
Taking the tensor product over all $\rho\in\Spec_m(A),$ we get \eqref{eq:HML_equals_Gamma}.
\end{proof}

\begin{theo}\label{th:HML^II_dual_of_loc_of_number_rings}We assume that the assumptions of Remark \ref{remark:formula_for_HML^II_of _loc_of_number_rings} 2) are satisfied. Assume that the set $T=\Spec_m(S^{-1}A)\setminus Crit(w)$ is non-empty. Then Mac Lane homology of the second kind of $(S^{-1}A,w)$ is given by
$$HML^{II}_0(S^{-1}A,w)=0$$
$$HML^{II}_1(S^{-1}A,w)\cong(\prod\limits_{\rho\in T}\widehat{A}_{\rho})/S^{-1}A\oplus \prod\limits_{\substack{\rho\in Crit(w);\\ n>0}}S^{-1}A/\rho^{\ord_p(n)}.$$\end{theo}

\begin{proof}From \eqref{eq:formula_for_HML^II_of _loc_of_number_rings} and Lemma \ref{lem:duality_for_homology} 
we immediately obtain
$$HML^{II}_0(S^{-1}A,w)=\Hom_{S^{-1}A}((T\cup S)^{-1}A\oplus\bigoplus\limits_{\substack{\rho\in Crit(w);\\ n>0}}S^{-1}A/\rho^{\ord_p(n)},S^{-1}A)=0.$$
Further, we have
$$HML^{II}_1(S^{-1}A,w)=\Ext^1_{S^{-1}A}((T\cup S)^{-1}A\oplus\bigoplus\limits_{\substack{\rho\in Crit(w);\\ n>0}}S^{-1}A/\rho^{\ord_p(n)},S^{-1}A).$$
Obviously, we have
$$\Ext^1_{S^{-1}A}(\bigoplus\limits_{\substack{\rho\in Crit(w);\\ n>0}}S^{-1}A/\rho^{\ord_p(n)},S^{-1}A)\cong\prod\limits_{\substack{\rho\in Crit(w);\\ n>0}}S^{-1}A/\rho^{\ord_p(n)}.$$

It remains to show that
\begin{equation}
\label{eq:Ext_from_localization}
\Ext^1_{S^{-1}A}((T\cup S)^{-1}A,S^{-1}A)\cong (\prod\limits_{\rho\in T}\widehat{A}_{\rho})/S^{-1}A. 
\end{equation}
For this, consider the short exact sequence
\begin{equation}\label{eq:short_exact_for_S^-1A} 0\to S^{-1}A\to (T\cup S)^{-1}A\to \bigoplus\limits_{\rho\in T}K/(S^{-1}A)_{\rho}\to 0.\end{equation}
Using the injective resolution
$$0\to S^{-1}A\to K\to \bigoplus\limits_{\rho\in\Spec_m (S^{-1}A)}K/(S^{-1}A)_{\rho}\to 0,$$
we get
\begin{multline}\label{eq:computation_of_Ext^1}\Ext^1_{S^{-1}A}(K/(S^{-1}A)_{\rho},S^{-1}A)\cong\Hom_{S^{-1}A}(K/(S^{-1}A)_{\rho},\bigoplus\limits_{\rho'}K/(S^{-1}A)_{\rho'})\\
\cong \End_{S^{-1}A}(K/(S^{-1}A)_{\rho})\cong \widehat{A}_{\rho}.\end{multline}

Applying the functors $\Ext^{\bullet}_{S^{-1}A}(-,S^{-1}A)$ to the short exact sequence \eqref{eq:short_exact_for_S^-1A}, we obtain the short exact sequence
$$0\to S^{-1}A\to \Ext^1_{S^{-1}A}(\bigoplus\limits_{\rho\in T}K/(S^{-1}A)_{\rho},S^{-1}A)\to \Ext^1_{S^{-1}A}((T\cup S)^{-1}A,S^{-1}A)\to 0.$$
Combining this with \eqref{eq:computation_of_Ext^1}, we obtain the isomorphism \eqref{eq:Ext_from_localization}. This proves the theorem.
\end{proof}

\begin{cor}\label{cor:surjective_diff_sp_seq_number_rings}Let  $w\in A$ be an element. Denote by $\{(E_r^{\bullet},d_r)\}_{r\geq 1}$ the spectral sequence associated with the mixed complex $$(\Hoch_{\bullet}(Q_{\bullet}(A),A),\delta_{[w]})$$

1) Let $\rho\subset A$ be a maximal ideal, such that $\rho\not\in Crit(w).$  Then for all $n>0$ the localized differential $$(d_n)_{\rho}:(E_n^0)_{\rho}\to (E_n^{-2n+1})_{\rho}=(HML_{2n-1}(A))_{\rho}$$ is surjective.

2) Suppose that for some $n>0$ we have $Crit(w)\cap\supp(A/n\cD_A)=\emptyset.$ Then the differential
$$d_n:E_n^0\to E_n^{-2n+1}=HML_{2n-1}(A)$$ is surjective.\end{cor}

\begin{proof}1) By Lemma \ref{lem:lemma_on_completions}, we have a natural quasi-isomorphism of mixed complexes
\begin{equation}\label{eq:q_is_for_completion}(\Hoch_{\bullet}(Q_{\bullet}(A),A),\delta_{[w]})\otimes_A\widehat{A}_{\rho}\stackrel{\sim}{\to}(\Hoch_{\bullet}(Q_{\bullet}(\widehat{A}_{\rho}),\widehat{A}_{\rho}),\delta_{[w]}).\end{equation}
It follows that we have an isomorphism of spectral sequences
$$\{(E_r^{\bullet},d_r)\}_{r\geq 1}\otimes_A\widehat{A}_{\rho}\cong \{(E_r'^{\bullet},d_r')\}_{r\geq 1}.$$
By Corollaries \ref{cor:sp_seq_for_homology_ramified} and \ref{cor:sp_seq_for_homology_non_ramified}, the differential
$$d_r':E_r'^0\to E_r'^{-2n+1}$$ is surjective. This proves 1).

2) By Corollary \ref{cor:HML_for_loc_of_number_rings}, we have a direct sum decomposition
$$HML_{2n-1}(A)\cong \bigoplus\limits_{\rho\in\supp(A/n\cD_A)}HML_{2n-1}(A)_{\rho}.$$
Thus, 2) is implied by 1).\end{proof}

\section{Adams operations on Mac Lane homology}
\label{sec:Adams_on_HML}

\subsection{Operations on simplicial sets}
\label{ssec:operations_on_simplicial_sets}

In this section we use our results to compute the Adams operations on Mac Lane homology. We follow \cite{McC} to define the Adams operations on the Mac Lane homology of a commutative ring.

Recall that $\Delta$ denotes the category of finite ordered sets. For $n\in\Z_{\geq 0},$ we denote by $[n]$ the ordered set $\{0<1<\dots<n\}.$ We tacitly identify $\Delta$ with its (equivalent) full subcategory, which consists of objects $[n],$ $n\geq 0.$

\begin{defi}1) For any $r\in\Z_{>0},$ we denote by $sd_r:\Delta\to\Delta$ the functor of $r$-th edgewise subdivision. On objects, it is defined by the formula $$sd_r([n-1])=[rn-1],\quad n\geq 1.$$ For a morphism $f:[n-1]\to [m-1]$ in $\Delta,$ the morphism $sd_r(f):[rn-1]\to [rm-1]$ is defined by the formula $$sd_r(f)(an+b)=am+f(b),\quad 0\leq a\leq r-1,\quad 0\leq b\leq n-1.$$

2) For any category $\cC$ and a simplicial object $X:\Delta^{op}\to\cC,$ we define the simplicial object $sd_r(X)$ as the composition $X\circ sd_r.$\end{defi}

\begin{remark}Note that we have a natural isomorphism of functors $sd_{rs}=sd_r\circ sd_s$ for $r,s\geq 1.$ In particular, for any category $\cC$ and a simplicial object $X_{\cdot}\in\cC^{\Delta^{op}}$ we have a natural isomorphism $sd_{rs}(X)\cong sd_s(sd_r(X)).$\end{remark}

Recall the standard $n$-dimensional simplex
$$\Delta^n=\{(x_0,\dots,x_n)\in\R^{n+1}\mid x_i\geq 0;\,\sum\limits_{i=0}^n x_i=1\}\subset\R^{n+1},\quad n\geq 0.$$
We denote by $d_r:\Delta^{n-1}\to\Delta^{rn-1}$ the affine map given by the formula $d_r(u)=(\frac{u}{r},\dots,\frac{u}{r}).$

\begin{lemma}(\cite{BHM}) Given a simplicial set $X_{\cdot},$ there is a homeomorphism of geometric realizations $D_r:|sd_r(X)|\to |X_{\cdot}|,$ defined by the maps $$\id\times d_r:X^{rn-1}\times\Delta^{n-1}\to X^{rn-1}\times\Delta^{rn-1},\quad n\geq 1.$$\end{lemma}

\begin{defi}\label{defi:operations_on_sset} Let $X_{\cdot}$ be a pointed simplicial set. A collection of morphisms $\varphi^r:sd_r(X)\to X_{\cdot},$ $r\in\Z_{\geq 0},$ is called a natural system of operators if the following diagrams are commutative for $r,s\geq 1:$
$$
\xymatrixcolsep{5pc}
\xymatrix{
   sd_{rs}(X) \ar[r]^{sd_r(\varphi^s)} \ar[dr]_{\varphi^{rs}} & sd_r(X) \ar[d]^{\varphi^r}\\
   & X_{\cdot}
  }$$
Moreover, we assume that $\varphi^1=\id$ and $\varphi^0$ is the basepoint.
\end{defi}

From now on, given a pointed simplicial set $X_{\cdot}$ with a natural system of operators $\{\varphi^r\}_{r\geq 0},$ we denote by $\Phi^r$ the following composite selfmaps of realizations:
\begin{equation}\label{eq:maps_Phi^r}\Phi^r:|X_{\cdot}|\stackrel{D_r^{-1}}{\to} |sd_r(X)|\stackrel{|\varphi^r|}{\to} |X_{\cdot}|.\end{equation}

\subsection{The case of simplicial abelian groups}
\label{ssec:operations_on_simplicial_abelian_groups}

Now suppose that $X_{\cdot}$ is a simplicial abelian group. We denote by $C_{\bullet}(X_{\cdot})$ the (non-normalized) chain complex associated to $X_{\cdot}.$ We have natural isomorphisms
$$H_n(C_{\bullet}(X))\cong \pi_n(|X_{\cdot}|).$$ We are going to define a map of simplicial abelian groups
$$D(r):X_{\cdot}\to sd_r(X),$$ such that the induced map $|D(r)|$ on geometric realizations is the homotopy inverse of $D_r.$

Let us put
$$S(r,n):=\{(\sigma,\eta)\in\mS_n\times\Hom_{\Delta}([n-1],[r-1])\mid \sigma(i)>\sigma(i+1)\Rightarrow \eta(i-1)<\eta(i)\}.$$

Given $\mu=(\sigma,\eta)\in S(r,n),$ we put $\mu_{-1}=-1,$ $$\mu_i=\eta(i)(n+1)+\sigma(i+1)-1,\quad 0\leq i\leq n-1,$$
and $\mu_n=rn+r-1.$ We define $\epsilon_{\mu}\in\Hom_{\Delta}([rn+r-1],[n])$ by the formula
$$\epsilon_{\mu}^{-1}(j)=\{\mu_{j-1}+1,\dots,\mu_j\}.$$

\begin{lemma}(\cite{McC}, Lemma 3.2) There is a bijection between $S(r,n)$ and the set of $n$-dimensional non-degenerate simplices of $sd_r(\Hom_{\Delta}(-,[n])),$ given by $\mu\mapsto\epsilon_{\mu}.$ \end{lemma}

\begin{defi}\label{defi:maps_D_n}(\cite{McC}, Definition 3.3) We define the homomorphisms $\cD_n(r):C_n(X)\to C_n(sd_r(X))$ by the formula
$$\cD_n(r)=\sum\limits_{\mu\in S(r,n)}\sgn(\mu)\epsilon_{\mu}.$$
Here $\sgn(\mu)=\sgn(\sigma)$ for $\mu=(\sigma,\eta)\in S(r,n).$\end{defi}

\begin{prop}\label{prop:map_D_bullet}(\cite{McC}, Proposition 3.4) The maps $\cD_n(r)$ from Definition \ref{defi:maps_D_n} define the morphism of complexes $D_{\bullet}(r):C_{\bullet}(X)\to C_{\bullet}(sd_r(X)),$ which is functorial w.r.t. simplicial abelian group $X,$ and passes to the normalized chain complexes.\end{prop}

\begin{prop}\label{prop:D(r)_inverse}(\cite{McC}, Corollary 3.7) For each $n\geq 0,$ the map $H_n(\cD_{\bullet}(r)):H_n(C_{\bullet}(X))\to H_n(C_{\bullet}(sd_r(X)))$ is inverse to $\pi_n(D_r):\pi_n(|sd_r(X)|)\to \pi_n(|X_{\cdot}|),$ where for any simplicial group $Y_{\cdot}$ we use the identification
$$H_{\bullet}(C_{\bullet}(Y))\cong \pi_{\bullet}(|Y_{\cdot}|).$$
In particular, the map $\cD_{\bullet}(r)$ is a quasi-isomorphism of complexes.\end{prop}

Suppose that we have a system of natural operators $\{\varphi^r\}_{r\geq 0}$ on a simplicial abelian group $X_{\cdot}.$ It follows from Proposition \ref{prop:D(r)_inverse} that the operations $$\psi^r=\pi_{\bullet}(\Phi^r):H_{n}(C_{\bullet}(X))\to H_{n}(C_{\bullet}(X)),\quad r\geq 1,\,n\geq 0,$$
are given by the chain maps
\begin{equation}\label{eq:formula_for_psi^r}\Psi^r=C_{\bullet}(\varphi^r)\circ\cD_{\bullet}(r):C_{\bullet}(X)\to C_{\bullet}(X).\end{equation}

\subsection{Computation of Adams operations on Mac Lane homology}
\label{ssec:computing_Adams_on_HML}

We now apply the results from the previous subsection to the Mac Lane homology of a commutative ring, following \cite{McC}, Section 6.

Recall that for any small category $\cC$ its nerve $N_{\cdot}(\cC)$ is a simplicial set with components
$$N_n(\cC)=\{X_0\stackrel{f_1}{\to}X_1\stackrel{f_2}{\to}\dots\stackrel{f_n}{\to}X_n\mid X_i\in Ob(\cC),\,f_i\in \Hom_{\cC}(X_{i-1},X_i)\}.$$
The face maps are defined by the formula
$$d_i:N_n(\cC)\to N_{n-1}(\cC),\quad d_i(f_1,\dots,f_n)=\begin{cases}(f_1,\dots,f_{i+1}f_i,\dots,f_n) & \text{for }1\leq i\leq n-1;\\
(f_2,\dots,f_n) & \text{for }i=0;\\
(f_1,\dots,f_{n-1}) & \text{for }i=n.\end{cases}$$
The degeneracy maps are defined by the formula
$$\sigma_i:N_n(\cC)\to N_{n+1}(\cC),\quad \sigma_i(f_1,\dots,f_n)=(f_1,\dots,f_i,{\bf{1}}_{X_i},f_{i+1},\dots,f_n),$$
where $0\leq i\leq n$

Let $R$ be an associative ring and $M$ an $R\otimes R^{op}$-module. Consider the simplicial abelian group $F_{\cdot}(R,M),$ where
$$F_n(R,M)=\bigoplus\limits_{P_0,\dots,P_n\in P(R)}\Z[\Hom_R(P_0,P_1)]\otimes\dots\Z[\Hom_R(P_{n-1},P_0)]\otimes \Hom_R(P_n,P_0\otimes_R M).$$

It is clear that there is a natural isomorphism of chain complexes
\begin{equation}\label{eq:C(F(R,M))=CML(R,M)}C_{\bullet}(F_{\cdot}(R,M))\cong \Hoch_{\bullet}(\Z[P(R)],M).\end{equation}

From now on we assume that $R$ is commutative and $M$ is a symmetric bimodule (that is, left and right multiplications by an element of $R$ are equal to each other). We introduce a natural system of operators $\{\varphi^r\}_{r\geq 0}$ on the simplicial abelian group $F_{\cdot}(R,M).$

Take some positive integer $r.$ Note that since $R$ is commutative, we have an additive functor
$$T^r:P(R)^{\times r}\to P(R),\quad T^r(P_1,\dots,P_r)=P_1\otimes_R P_2\otimes_R\dots\otimes_R P_n.$$
The abelian group $sd_r(F(R,M))_n$ is generated by the elements of the form
$$[f_1]\otimes \dots\otimes [f_{rn+r-1}]\otimes m,$$
where $f_i:P_{i-1}\to P_i,$ $1\leq i\leq rn+r-1,$ and $m\in\Hom_R(P_{rn+r-1},P_0\otimes_R M).$

We put $$\varphi^r([f_1]\otimes \dots\otimes [f_{rn+r-1}]\otimes m):=[g_1]\otimes\dots\otimes[g_n]\otimes m',$$
where \begin{multline*}g_i=T^r(f_i,f_{n+1+i},\dots,f_{(r-1)(n+1)+i}):\\T^r(P_{i-1},P_{n+i},\dots,P_{(r-1)(n+1)+i-1})\to T^r(P_i,P_{n+1+i},\dots,P_{(r-1)(n+1)+i}),\end{multline*} 
and \begin{multline*}m'=m\otimes f_{n+1}\otimes f_{2n+2}\otimes\dots\otimes f_{(r-1)(n+1)}\in\\
\Hom_R(T^r(P_n,P_{2n+1},\dots,P_{rn+r-1}),T^r(P_0,P_{n+1},\dots,P_{(r-1)(n+1)})\otimes_R M).\end{multline*}
It is easy to check that $\varphi^r$ is a morphism of simplicial abelian groups, and the collection $\{\varphi^r\}_{r\geq 0}$ is indeed a system of natural operators on $F_{\cdot}(R,M)$ (see also \cite{McC}, Section 6 for the case $M=R$).

It follows from the isomorphism \eqref{eq:C(F(R,M))=CML(R,M)} that we have a system of operations
$$\psi^r:HML_{\bullet}(R,M)\to HML_{\bullet}(R,M),$$
which are induced by the actual maps of complexes
\begin{equation}\label{eq:maps_Psi^r_on_CML}\Psi^r:\Hoch_{\bullet}(\Z[P(R)],M)\to \Hoch_{\bullet}(\Z[P(R)],M).\end{equation}
From now on we denote by $\Psi^r$ the morphisms introduced in \eqref{eq:maps_Psi^r_on_CML}

\begin{lemma}\label{lem:commutation_of_phi^r_and_delta_w} For any $w\in R$ we have an equality of morphisms of complexes
$$\Psi^r\circ\delta_{[w]}=r\cdot\delta_{[w]}\circ\Psi^r.$$\end{lemma}

\begin{proof}This is straightforward checking.\end{proof}



\begin{theo}\label{th:operations_psi^r_number_rings} Let $A$ be a global number ring. The operation $\psi^r$ on $HML_0(A)$ is equal to identity and we have
$$\psi^r(x)=r^nx\text{ for }x\in HML_{2n-1}(A),\,k\geq 1.$$\end{theo}

\begin{proof}The first statement follows immediately from the definition of $\psi^r.$

We now compute the operation $\psi^r$\ on $HML_{2n-1}(A),$ where $n>0.$ Take some $w\in A$ such that $Crit(w)\cap\supp(A/n\cD_A)=\emptyset.$ 

Denote by $\{(E_r^{\bullet},d_r)\}_{r\geq 1}$ the spectral sequence associated with the mixed complex $(\Hoch_{\bullet}(\Z[P(R)],M),\delta_{[w]}).$ By Proposition \ref{prop:HML^II_via_Q^n}, this mixed complex is naturally quasi-isomorphic to the mixed complex $(\Hoch_{\bullet}(Q_{\bullet}(A),A),\delta_{[w]})$,  hence the associated spectral sequences are isomorphic. 

By Corollary \ref{cor:surjective_diff_sp_seq_number_rings}, the differential
\begin{equation}\label{eq:surj_diff_d_n} d_n:E_n^0\to E_n^{-2n+1}=HML_{2n-1}(A)\end{equation}
is surjective. 

Take some class $\alpha\in HML_{2n-1}(A).$ By the surjectivity of \eqref{eq:surj_diff_d_n}, we can find a collection of chains $\gamma_1,\gamma_2\dots,\gamma_n\in\Hoch_{\bullet}(A),$ such that $\gamma_i\in \Hoch_{2i-2}(A),$ and
$$b(\gamma_{i+1})=\delta_{[w]}(\gamma_i),\quad 1\leq i\leq n-1,$$ and the cocycle $\delta_{[w]}(\gamma_n)\in\Hoch_{2n-1}(A)$ represents $\alpha.$ We have $\alpha=d_n(\bar{\gamma_1}).$ 

Put $$\gamma_i':=r^{n+1-i}\Psi^r(\gamma_i).$$
By Lemma \ref{lem:commutation_of_phi^r_and_delta_w}, for $1\leq i\leq n-1$ we have

$$b(\gamma_{i+1}')=r^{n-i}\Psi^r(b(\gamma_{i+1}))=r^{n-i}\Psi^r(\delta_{[w]}(\gamma_i))=r^{n+1-i}\delta_{[w]}(\Psi^r(\gamma_i))=\delta_{[w]}(\gamma_i').$$
We also have
$$\delta_{[w]}(\gamma_n')=r\delta_{[w]}(\Psi^r(\gamma_n))=\Psi^r(\delta_{[w]}(\gamma_n)).$$
Therefore, we have $$\psi^r(\alpha)=d_n(\bar{\gamma_1'})=r^nd_n(\bar{\gamma_1})=r^n\alpha.$$
This proves the theorem.
\end{proof}

\appendix

\section{$\Z$-graded spectral sequences}
\label{sec:appendix_on_sp_seq}

First we define the category of $\Z$-graded spectral sequences. Let $\cA$ be an abelian category.

\begin{defi}\label{defi:Z-graded_spectral_sequence}A $\Z$-graded spectral sequence in $\cA$ is the following data:

1) For some non-negative integer $r_0,$ a sequence $E=\{E_r^{\bullet}\}_{r\geq r_0},$ of $\Z$-graded objects in $\cA;$

2) A homogeneous differential $d_r:E_r^{\bullet}\to E_r^{\bullet}$ of degree $1-2r,$ for $r\geq r_0;$

3) For each $r\geq r_0,$ an isomorphism of $\Z$-graded objects \begin{equation}\label{eq:structure_isomorphisms}E_{r+1}^{\bullet}\cong H^{\bullet}(E_r^{\bullet},d_r).\end{equation}.

A morphism of spectral sequences $f:E\to E'$ (with the same $r_0$) is a sequence of homogeneous maps of degree zero $f_r:E_r^{\bullet}\to E_r^{\prime\bullet},$ $r\geq r_0,$ which commute with differentials $d_r$ and the isomorphisms \eqref{eq:structure_isomorphisms}.\end{defi}

The basic source of $\Z$-graded spectral sequence is the category of exact couples.

\begin{defi}\label{defi:exact_pair} A $\Z$-graded exact couple in $\cA$ is the following data:

1) $\Z$-graded objects $D^{\bullet}$ and $E^{\bullet}$ in $\cA;$

2) homogenous morphisms $i:D^{\bullet}\to D^{\bullet},$ $j:D^{\bullet}\to E^{\bullet},$ $k:E^{\bullet}\to D^{\bullet},$ such that $\deg(i)=2,$ $\deg(k)=-1$ and $\deg(j)=-2r$ for some $r\in\Z_{\geq 0}.$

These data must satisfy the following condition: the $3$-periodic sequence
$$\xymatrixcolsep{5pc}
\xymatrix{
   D^{\bullet} \ar[rr]^{i} & & D^{\bullet} \ar[dl]^{j}\\
   & E^{\bullet} \ar[ul]^{k}
  }$$
is exact.

Given two $\Z$-graded exact couples $(D_1^{\bullet},E_1^{\bullet},i_1,j_1,k_1)$ and $(D_2^{\bullet},E_2^{\bullet},i_2,j_2,k_2)$ (with the same $r$), a morphism  $$f:(D_1^{\bullet},E_1^{\bullet},i_1,j_1,k_1)\to (D_2^{\bullet},E_2^{\bullet},i_2,j_2,k_2)$$
is a pair of homogeneous maps of degree zero $f_D:D_1^{\bullet}\to D_2^{\bullet},$ $f_E:E_1^{\bullet}\to E_2^{\bullet},$ such that $i_2f_D=f_Di_1,$ $j_2f_D=f_Ej_1,$ $k_2f_E=f_Dk_1.$
\end{defi}

Given an exact couple $(D^{\bullet},E^{\bullet},i,j,k),$ we see that $jk:E^{\bullet}\to E^{\bullet}$ is a homogeneous odd differential. Let us put
$$D'^{\bullet}=\im(i),\quad E'^{\bullet}=H^{\bullet}(E^{\bullet},jk).$$
We define the morphisms $i':D'^{\bullet}\to D'^{\bullet}$ and
$k':E'^{\bullet}\to D'^{\bullet}$ to be induced by $i$ and $k$ respectively.
It is clear that the morphism $k'$ is well-defined. Further, the morphism $j':D'^{\bullet}\to E'^{\bullet}$ is defined by the formula
$$j'(x)=\overline{j(\widetilde{x})},$$
where $\widetilde{x}\in D^{\deg(x)-2}$ is any element such that $i(\widetilde{x})=x$ (it is clear that this definition makes sense for a general abelian category $\cA$).
It is easy to check that $j'$ is also well-defined.

\begin{prop}\label{prop:derived_exact_pair} The collection $(D'^{\bullet},E'^{\bullet},i',j',k')$ is an exact couple. The construction
$$(D^{\bullet},E^{\bullet},i,j,k)\mapsto (D'^{\bullet},E'^{\bullet},i',j',k')$$
is functorial.
\end{prop}

\begin{proof}This is straightforward.\end{proof}

\begin{defi}\label{defi:derived_exact_couple} Given an exact couple $(D^{\bullet},E^{\bullet},i,j,k),$ its derived exact couple is by definition $(D'^{\bullet},E'^{\bullet},i',j',k').$ Note that
$$\deg(i')=\deg(i),\quad \deg(j')=\deg(j)-2,\quad \deg(k')=\deg(k).$$

The $n$-th derived exact couple is denoted by $(D^{(n)\bullet},E^{(n)\bullet},i^{(n)},j^{(n)},k^{(n)}).$\end{defi}

\begin{prop}\label{prop:sp_seq_from_exact_pair} Let $(D^{\bullet},E^{\bullet},i,j,k)$ be an exact couple in $\cA,$ with $\deg(j)=-2r_0+2.$ Then we have a spectral sequence $$\{(E_r^{\bullet},d_r)\}_{r\geq r_0},\quad E_r^{\bullet}=(E^{(r-r_0)})^{\bullet},\quad d_r=j^{(r-r_0)}k^{(r-r_0)}.$$
Moreover, this construction is functorial.\end{prop}

\begin{proof}This follows immediately from the definitions.\end{proof}

\begin{prop}\label{prop:sp_seq_from_filtered} Let $(C^{\bullet},d)$ be a $\Z/2$-graded complex in $\cA.$ Suppose that we have an increasing filtration $F_{\bullet}C^{\bullet}$ on the $\Z/2$-graded object $C^{\bullet},$ such that $F_nC^{\bar{n}}=F_{n+1}C^{\bar{n}},$ $n\in\Z,$ and $d(F_nC^{\bullet})\subset F_{n+1}C^{\bullet}.$
Then we have a natural $\Z$-graded spectral sequence $\{(E_r^{\bullet},d_r)\}_{r\geq 1},$ where
$E_1^n$ is the $n$-th cohomology of the complex
\begin{equation}\label{eq:complex_F_k/F_k-2}\dots\to F_{k-1}C^{\bar{k}-1}/F_{k-3}C^{\bar{k}-1}\to F_kC^{\bar{k}}/F_{k-2}C^{\bar{k}}\to\dots,\end{equation}
and $d_1:E_1^n\to E_1^{n-1}$ is induced by $d.$\end{prop}

\begin{proof}Let us construct a functorial exact couple from $(C^{\bullet},F_{\bullet}).$ Consider the complex $\cK^{\bullet}$ with $\cK^n=F_nC^{\bar{n}},$ and the differential induced by $d.$ We have an obvious inclusion $\iota:\cK^{\bullet}[-2]\to \cK^{\bullet}.$ We put $\cL^{\bullet}=\coker(\iota),$ so that we have a short exact sequence of complexes
$$0\to \cK^{\bullet}[-2]\to \cK^{\bullet}\to\cL^{\bullet}\to 0.$$
It induces a long exact sequence in cohomology, which can be considered as an exact couple $(H^{\bullet}(\cK^{\bullet}),H^{\bullet}(\cL^{\bullet}),i,j,k),$ where the maps $i$ and $j$ are induced by the morphisms of complexes, and $k$ is the boundary map. Clearly, we have $\deg(i)=2,$ $\deg(j)=0$ and $\deg(k)=-1.$ Hence, by Proposition \ref{prop:sp_seq_from_exact_pair} we have a natural $\Z$-graded spectral sequence $\{(E_r^{\bullet},d_r)\}_{r\geq 1},$ starting with $(H^{\bullet}(\cL^{\bullet}),jk).$ It remains to note that the complex $\cL^{\bullet}$ is exactly the complex \eqref{eq:complex_F_k/F_k-2}, and the differential $jk$ is exactly the one induced by $d.$ This proves the proposition.\end{proof}

From now on we assume for simplicity that the category $\cA$ has small coproducts, small products and the small filtered colimits are exact in $\cA$ (for example, $\cA$ can be a category of modules over an associative ring).

\begin{cor}\label{cor:sp_seq_from_mixed_appendix} Let $(\cK^{\bullet},b,\delta)$ be a $\Z$-graded mixed complex in $\cA.$ Then we have a natural spectral sequence $\{(E_r^{\bullet},d_r)\}{r\geq 1},$ with
$$(E_1^{\bullet},d_1)=(H^{\bullet}(\cK^{\bullet},b),\delta).$$
Moreover, this spectral sequence is functorial under $A_{\infty}$-morphisms of mixed complexes.\end{cor}

\begin{proof}This is a direct consequence of Proposition \ref{prop:sp_seq_from_filtered}. Indeed, let us take the $\Z/2$-graded complex $\Tot^L(\cK^{\bullet},b,\delta).$ We have a natural filtration defined by the formula
$$F_n\Tot^L(\cK^{\bullet})^0=\prod\limits_{2k\leq n}\cK^{2k},\quad F_n\Tot^L(\cK^{\bullet})^1=\prod\limits_{2k+1\leq n}\cK^{2k+1}.$$
This filtration satisfies the assumptions of Proposition \ref{prop:sp_seq_from_filtered}. The complex \eqref{eq:complex_F_k/F_k-2} is exactly the complex $(\cK^{\bullet},b).$ This gives us the desired spectral sequence.

Finally, note that our filtration is functorial under $A_{\infty}$-morphisms. Hence, so is the associated spectral sequence.\end{proof}

\begin{defi}\label{defi:convergence_of_sp_seq}Let $\{(E_r^{\bullet},d_r)\}_{r\geq r_0}$ be a $\Z$-graded spectral sequence in $\cA.$ We assume that $E_{r_0}^{\bullet}$ is bounded below (hence so are $E_r^{\bullet}$ for $r>r_0$).

1) The limiting $\Z$-graded object $E_{\infty}^{\bullet}$ is defined as follows. Taka some $n\in\Z.$ Let us choose $r_1\geq r_0$ such that $E_{r_0}^{n-2r+1}=0$ for $r\geq r_1.$ Then we have a natural surjective map $E_r^n\to E_{r+1}^n$ for $r\geq r_1.$ We put $E_{\infty}^n:=\colim\limits_{r\geq r_1} E_r^n.$

2) Let $E^{\bullet}$ be a $\Z/2$-graded object with an exhaustive filtration $F_{\bullet}E^{\bullet},$ satisfying $F_nE^{\bar{n}}=F_{n+1}E^{\bar{n}}.$ and $F_{<<0}E^{\bullet}=0.$ We say that the spectral sequence $\{(E_r^{\bullet},d_r)\}_{r\geq r_0}$ converges to the filtered object $(E^{\bullet},F_{\bullet})$ if we have isomorphisms
$$E_{\infty}^n\cong F_nE^{\bar{n}}/F_{n-2}E^{\bar{n}}.$$\end{defi}

\begin{prop}\label{prop:convergence_of_sp_seq_appendix} In the assumptions of Proposition \ref{prop:sp_seq_from_filtered}, suppose that $F_nC^{\bullet}=0$ for $n<<0,$ and the filtration $F_{\bullet}C^{\bullet}$ is exhaustive. Then the associated spectral sequence converges to $H^{\bullet}(C^{\bullet}).$\end{prop}

\begin{proof}Consider the increasing filtration $F_{\bullet}H^{\bullet}(C^{\bullet}),$ such that $F_nH^k(C^{\bullet})$ is the image of the natural morphism from $$\ker(d:F_nC^k\to F_{n+1}C^{k+1})\to H^k(C^{\bullet}).$$ Clearly, this filtration is exhaustive. It is easy to see that $F_nH^{\bar{n}}(C^{\bullet})=F_{n+1}H^{\bar{n}}(C^{\bullet}),$ and $F_nH^{\bar{n}}(C^{\bullet})/F_{n-2}H^{\bar{n}}(C^{\bullet})\cong E_{\infty}^n.$ This proves the proposition.\end{proof}

\begin{cor}\label{cor:convergence_of_sp_seq_from_mixed} Let $(\cK^{\bullet},b,\delta)$ be a bounded below mixed complex in $\cA.$ Then the spectral sequence of Corollary \ref{cor:sp_seq_from_mixed_appendix} converges to $H^{\bullet}(\Tot^{\oplus}(\cK^{\bullet},b,\delta)).$\end{cor}

\begin{proof}Indeed, this follows immediately from Proposition \ref{prop:convergence_of_sp_seq_appendix} and the proof of Corollary \ref{cor:sp_seq_from_mixed_appendix}.\end{proof}

\begin{prop}\label{prop:sp_seq_of_DG_rings} Let $\cA$ be a DG ring, and $\delta:\cA\to\cA$ a derivation of degree $-1$ satisfying $\delta^2=0.$ Take the spectral sequence $\{(E_r^{\bullet},d_r)\}_{r\geq 1}$ associated with the mixed complex $(\cA,\delta).$ 

1) Each $E_r^{\bullet}$ carries a structure of a unital graded ring for which $d_r$ is a derivation. Moreover, the product on $E_1^{\bullet}=H^{\bullet}(\cA)$ is induced by that on $\cA,$ and the product on $E_{r+1}^{\bullet}$ is induced by that on $E_r^{\bullet}.$

2) Suppose that $\cA$ is bounded below. Then we have an isomorphism of graded algebras $$E_{\infty}^{\bullet}\cong\gr_{\bullet}^FH^{\bullet}(\Tot^L(\cA,\delta)).$$\end{prop}

\begin{proof}[Proof of Proposition \ref{prop:product_on_Tot_of_DGAs}] Statement 1) follows immediately from the definitions.

2) By Corollary \ref{cor:sp_seq_from_mixed_appendix}, the first term is $E_1^{\bullet}=H^{\bullet}(\cA),$ and $d_1$ is induced by $\delta.$ So, the statement 2) obviously holds for $r=1.$ Now we prove the general case by induction.

Suppose that the statement is proved for all $r\leq r_0,$ where $r_0\geq 1.$ We will show it for $r=r_0+1.$ First, since $d_{r-1}$ is a derivation, the product on $E_{r-1}^{\bullet}$ induces a product on $E_r^{\bullet}=H^{\bullet}(E_{r-1}^{\bullet},d_{r-1}).$ It remains to show that $d_r$ is a derivation.

Take some homogeneous elements $x\in E_r^n,$ $y\in E_r^m.$ Then we can find sequences of cochains $x_i\in\cA^{n-2i+2},$ $y_i\in \cA^{m-2i+2},$ $1\leq i\leq r$ such that
$$d(x_1)=d(y_1)=0,\quad d(x_{i+1})=\delta(x_i),\,d(y_{i+1})=\delta(y_i)\text{ for }1\leq i\leq r-1,$$
so that $x_1$ represents $x$ and $y_1$ represents $y.$ The element $d_r(x)$ (resp. $d_r(y)$) of $E_r^{\bullet}$ is represented by $\delta(x_r)$ (resp. $\delta(y_r)$).We need to show the equality
\begin{equation}
\label{eq:Leibniz_for_d_r} d_r(xy)=d_r(x)y+(-1)^n x d_r(y).
\end{equation}
Let us put $$z_i:=\sum\limits_{j=1}^i x_j y_{i+1-j}\in E_r^{n+m-2i+2},\quad 1\leq i\leq r.$$ By the inductive assumption, the cocycle $z_1$ represents $xy\in E_r^{n+m}.$
We claim that $d(z_{i+1})=\delta(z_i)$ for $1\leq i\leq r-1.$ Indeed, we have
\begin{multline*}\delta(z_i)=d(\sum\limits_{j=1}^i x_j y_{i+1-j})=\sum\limits_{j=1}^i \delta(x_j) y_{i+1-j}+(-1)^n\sum\limits_{j=1}^i x_j \delta(y_{i+1-j})\\
=\sum\limits_{j=1}^{i+1}d(x_j)y_{i+2-j}+(-1)^n\sum\limits_{j=1}^{i+1}x_j d(y_{i+2-j})=d(\sum\limits_{j=1}^{i+1}x_j y_{i+2-j})=d(z_{i+1}).\end{multline*}
To prove \eqref{eq:Leibniz_for_d_r}, it suffices to show that $(\delta(z_r)-\delta(x_r)y_1-(-1)^nx_1\delta(y_r))$ is a coboundary in $\cA.$ But we have \begin{multline*}\delta(z_r)-\delta(x_r)y_1-(-1)^nx_1\delta(y_r)=\sum\limits_{j=1}^{r-1}\delta(x_j)y_{r+1-j}+(-1)^n\sum\limits_{j=2}^r x_j\delta(y_{r+1-j})
\\=\sum\limits_{j=2}^{r}d(x_j)y_{r+2-j}+(-1)^n\sum\limits_{j=2}^r x_j d(y_{r+2-j})=d(\sum\limits_{j=2}^r x_j y_{r+2-j}).\end{multline*}
This proves \eqref{eq:Leibniz_for_d_r} and the statement 2).

3) By Corollary \ref{cor:convergence_of_sp_seq_from_mixed} we have an isomorphism of graded $A$-modules $E_{\infty}^{\bullet}$ and $\gr_{\bullet}^FH^{\bullet}(\Tot^L(\cA,\delta)).$ It is immediate from the proof of 2) that this isomorphism is compatible with the product and preserves units. This proves 3).
\end{proof}

\end{document}